\documentclass[reqno,12pt]{amsart}

\def\vint_#1{\mathchoice%
          {\mathop{\kern 0.2em\vrule width 0.6em height 0.69678ex depth -0.58065ex
                  \kern -0.8em \intop}\nolimits_{\kern -0.4em#1}}%
          {\mathop{\kern 0.1em\vrule width 0.5em height 0.69678ex depth -0.60387ex
                  \kern -0.6em \intop}\nolimits_{#1}}%
          {\mathop{\kern 0.1em\vrule width 0.5em height 0.69678ex
              depth -0.60387ex
                  \kern -0.6em \intop}\nolimits_{#1}}%
          {\mathop{\kern 0.1em\vrule width 0.5em height 0.69678ex depth -0.60387ex
                  \kern -0.6em \intop}\nolimits_{#1}}}
\def\vintslides_#1{\mathchoice%
          {\mathop{\kern 0.1em\vrule width 0.5em height 0.697ex depth -0.581ex
                  \kern -0.6em \intop}\nolimits_{\kern -0.4em#1}}%
          {\mathop{\kern 0.1em\vrule width 0.3em height 0.697ex depth -0.604ex
                  \kern -0.4em \intop}\nolimits_{#1}}%
          {\mathop{\kern 0.1em\vrule width 0.3em height 0.697ex depth -0.604ex
                  \kern -0.4em \intop}\nolimits_{#1}}%
          {\mathop{\kern 0.1em\vrule width 0.3em height 0.697ex depth -0.604ex
                  \kern -0.4em \intop}\nolimits_{#1}}}

\usepackage{a4wide}
\usepackage{amsmath}
\usepackage{mathtools}
\usepackage{amsthm}
\usepackage{amssymb}
\usepackage{hyperref}

\usepackage[capitalise]{cleveref}
\usepackage{glossaries}
\expandafter\def\csname ver@etex.sty\endcsname{3000/12/31}

\usepackage{autonum}
\usepackage{color}
\usepackage{enumerate}
\usepackage[obeyFinal]{todonotes}

\usepackage{csquotes}

\newcommand{\R}{\mathbb{R}}
\newcommand{\N}{\mathbb{N}}

\newcommand{\id}{\mathrm{id}}
\newcommand{\m}{\mathfrak{m}}
\newcommand{\geo}{\mathrm{Geo}}
\newcommand{\opt}{\mathrm{Opt}}
\newcommand{\og}{\mathrm{OptGeo}}
\newcommand{\vol}{\mathrm{vol}}
\renewcommand{\d}{\,\mathrm{d}}
\newcommand{\ap}[1][t,s]{\hat P_{#1}}
\renewcommand{\liminf}{\varliminf}
\renewcommand{\limsup}{\varlimsup}
\newcommand{\Ric}{\mathrm{Ric}}
\newcommand{\Hess}{\mathrm{Hess}}

\newcommand{\cd}{\mathrm{CD}}
\newcommand{\rcd}{\mathrm{RCD}}
\newcommand{\ent}{\mathrm{Ent}}
\newcommand{\ch}{\mathrm{Ch}}

\newcommand{\p}{\mathcal{P}}
\newcommand{\spt}{\mathrm{spt}\,}

\renewcommand{\L}{\Delta_f\,}

\newcommand{\rfe}{\mathrm{RFex}}

\numberwithin{equation}{section}
\newtheorem{theorem}{Theorem}[section]
\newtheorem{proposition}[theorem]{Proposition}
\newtheorem{lemma}[theorem]{Lemma}
\newtheorem{corollary}[theorem]{Corollary}
\theoremstyle{definition}

\newtheorem{definition}[theorem]{Definition}
\newtheorem{assumption}[theorem]{Assumption}
\theoremstyle{remark}
\newtheorem{rem}[theorem]{Remark}
\newtheorem{example}[theorem]{Example}

\newcommand{\aint}[2][]{
\ifthenelse{\equal{#1}{}}%
{%
\mathchoice%
     {\mathop{\kern 0.2em\vrule width 0.6em height 0.69678ex depth -0.58065ex
             \kern -0.8em \intop}\nolimits_{\kern -0.45em#2}^{#1}}%
     {\mathop{\kern 0.1em\vrule width 0.5em height 0.69678ex depth -0.60387ex
             \kern -0.6em \intop}\nolimits_{#2}^{#1}}%
     {\mathop{\kern 0.1em\vrule width 0.5em height 0.69678ex depth -0.60387ex
             \kern -0.6em \intop}\nolimits_{#2}^{#1}}%
     {\mathop{\kern 0.1em\vrule width 0.5em height 0.69678ex depth -0.60387ex
             \kern -0.6em \intop}\nolimits_{#2}^{#1}}}%
{%
\mathchoice%
     {\mathop{\kern 0.2em\vrule width 0.6em height 0.69678ex depth -0.58065ex
             \kern -0.8em \intop}\nolimits_{\kern -0.45em#1}^{#2}}%
     {\mathop{\kern 0.1em\vrule width 0.5em height 0.69678ex depth -0.60387ex
             \kern -0.6em \intop}\nolimits_{#1}^{#2}}%
     {\mathop{\kern 0.1em\vrule width 0.5em height 0.69678ex depth -0.60387ex
             \kern -0.6em \intop}\nolimits_{#1}^{#2}}%
     {\mathop{\kern 0.1em\vrule width 0.5em height 0.69678ex depth -0.60387ex
             \kern -0.6em \intop}\nolimits_{#1}^{#2}}}} 
 
\begin{document}

\parindent=0in

\title[Synthetic notions of Ricci flow for metric measure spaces]
{Synthetic notions of Ricci flow for metric measure spaces}

\author{Matthias Erbar$^\dagger$}
\author{Zhenhao Li$^\dagger$}
\author{Timo Schultz$^\dagger$$^\ddag$}

\address{$^\dagger$Faculty of Mathematics \\university
         Bielefeld University \\
         Postfach 10 01 31 \\
         33501 Bielefeld \\
         Germany.}
\address{$^\ddag$Department of Mathematics and Statistics\\
        University of Jyv\"askyl\"a \\
        P.O.Box 35 (MaD)\\
        FI-40014\\
        Finland.}

\email{matthias.erbar@math.uni-bielefeld.de}        
\email{zhenhao.li@math.uni-bielefeld.de}
\email{timo.m.schultz@jyu.fi}

\date{\today}


\subjclass[2020]{}
\thanks{}


\begin{abstract}
We develop different synthetic notions of Ricci flow in the setting of time-dependent metric measure spaces based on ideas from optimal transport. They are formulated in terms of dynamic convexity and local concavity of the entropy along Wasserstein geodesics on the one hand and in terms of global and short-time asymptotic transport cost estimates for the heat flow on the other hand. We show that these properties characterise smooth (weighted) Ricci flows. Further, we investigate the relation between the different notions in the non-smooth setting of time-dependent metric measure spaces.
\end{abstract}

\maketitle
\tableofcontents
 \section{Introduction}\label{sec:introduction}
 The goal of this paper will be to develop synthetic notions of Ricci flow in the setting of time-dependent metric measure spaces.

A smooth manifold $M$ equipped with a smooth one-parameter family of
Riemannian metrics $(g_t)_{t\in I}$ evolves according to \emph{Ricci
  flow}, if
\begin{equation}
  \label{eq:RF}
  \partial_tg_t = -2\Ric_{g_t}\;.
\end{equation}
Since the groundbreaking work of Hamilton \cite{H82,H95}, Ricci flow has
received a lot of attention and has become a powerful tool in many applications, most prominently in Perelman's work on the Poincar\'e conjecture and Thurston's geometrization conjecture \cite{P02,P03,P03b}, see also \cite{CZ06,KL08,MT07}. For a detailed account on the Ricci flow we refer e.g.~to \cite{CK04}.
\medskip

A challenging feature of Ricci flow is that it typically develops singularities in finite time. In Hamilton's and Perelman's approach e.g.~such singularities require a careful surgery procedure. Therefore it seems desirable to obtain robust characterisations of Ricci flow that make it possible to consider evolutions of non-smooth spaces and eventually flows through singularities. In recent years, a lot of activity has been devoted in this direction. 

Let us highlight some of these developments. Ricci flows with irregular or incomplete metrics as initial data have been intensely investigated, see e.g. \cite{Sim02, Sim12, KL12, GT13, MRS15, Yin10}. In \cite{ST21} e.g., Simon and Topping proved existence of a smooth Ricci flow in dimension 3 starting from non-collapsed Ricci limits spaces in the sense of Cheeger-Colding \cite{CC97}.
A different major challenge is to define and analyze Ricci flows through singularities and study evolution of spaces with changing dimension and/or topological type. Among the exciting recent contributions, Bamler, Kleiner and Lott \cite{KL17, KL20, BK22} have introduced a weak notion of Ricci flow in 3 dimensions and have constructed of a canonical Ricci flow through singularities as the unique limit of Ricci flows with surgery. In \cite{Bam23} Bamler develops a compactness theory for super Ricci flows providing in particular the basis for a partial regularity and structure theory for non-collapsed limits of Ricci flows established in \cite{Bam20}.
An alternative approach is to develop characterizations of Ricci flow in terms of robust properties that can eventually provide a synthetic definition in a non-smooth setting. In this direction, Haslhofer and Naber \cite{Haslhofer-Naber18} and Cheng and Thalmaier \cite{CT17} have characterized Ricci flow (and two-sided bounds on the Ricci curvature on static manifolds) in terms of functional inequalities on the path space equipped with the Wiener measure. McCann and Topping \cite{McCann-Topping}, Sturm \cite{Sturm2018Super} and Kopfer and Sturm \cite{Kopfer-Sturm2018} have used the heat flow and ideas from optimal transport to characterize super-Ricci flows by extending the characterizations of lower Ricci bounds to a dynamic setting. 

Our main contribution in the present paper is to provide and analyse two synthetic notions of Ricci flow based on optimal transport and the short time behaviour of the heat flow in the setting of time-dependent metric measure spaces.
In the following we will describe this approach and our results in more detail.
\medskip

Since the seminal work of Cordero, McCann, and Schmuckenschl\"ager \cite{Mccann2001} and von Renesse and Sturm \cite{vonRenesseSturm} it is well known that lower bounds on the Ricci curvature can be encoded using optimal transport and the heat flow. Namely, let $(M,g)$ be a Riemannian manifold. We denote by $W$ the $L^2$ Wasserstein distance on the space $\mathcal P_2(M)$ of probability measures over $M$ built from the Riemannian distance $d$ associated to $g$ (see Section 2 for recalling the definition. For a probability measure $\mu$ on $M$ the Boltzmann entropy is given by 
\[\ent(\mu)=\int \rho\log \rho\; \d \vol_g\;,\]
provided $\mu=\rho\vol_g$ is absolutely continuous w.r.t. the volume measure $\vol_g$, and by $\ent(\mu)=+\infty$ else. We denote by $P_t=e^{t\Delta}$ the heat semigroup generated by the Laplace-Beltrami operator $\Delta$ and by $\hat P_t$ 
the dual semigroup acting on measures. Now, the following are equivalent
\begin{itemize}
\item[($0$)] $\Ric_g\geq 0$\;;
\item[($1$)] \emph{geodesic convexity of the entropy}: for any constant speed Wasserstein geodesic $(\mu^a)_{a\in[0,1]}$ and all $a\in[0,1]$ we have
\[\ent(\mu^a)\leq (1-a)\ent(\mu^0) + a \ent(\mu^1)\;; \]
\item[($2$)] \emph{Wasserstein contractivity of the heat flow}: for all $\mu,\nu\in \mathcal P_2(M)$
\[W(\hat P_t\mu, \hat P_t\nu) \leq W(\mu,\nu) \;.\]
\end{itemize}
The latter properties are robust and can be used to give a synthetic definition of Ricci curvature bounds for non-smooth spaces with the only structure required being a distance and a reference measure. Starting from the pioneering works by Sturm \cite{sturm2006--1,Sturm06-2} and Lott and Villani \cite{LottVillani09} a rich and still rapidly growing theory of metric measure spaces with synthetic Ricci bounds has been developed. For a concise partial overview, we refer e.g. to \cite{Amb18}.

McCann and Topping \cite{McCann-Topping} and later Sturm \cite{Sturm2018Super} showed that this approach can be generalised to a dynamic setting of a time-dependent family of metrics $(g_t)$ on the manifold $M$ to obtain the following characterisation:

\begin{itemize}
\item[($0_{\sf dyn}$)] $(M,g_t)$ is a super-Ricci flow, i.e. $\partial_t g_t \geq -2\Ric_{g_t}$\;;
\item[($1_{\sf dyn}$)] \emph{dynamic convexity of the entropy}: for all $t$ and any constant speed geodesic $(\mu^a)_{a\in[0,1]}$ in $(P(M),W_t)$ we have
\[\partial_a\big|_{a=1-}\ent_t(\mu^{a})-\partial_a\big|_{a=0+}\ent_t(\mu^{a})\geq -\frac12 \partial_t W_t(\mu^0,\mu^1)^2\;; \]
\item[($2_{\sf dyn}$)] \emph{Wasserstein contractivity of the heat flow}: for any $s\leq t$ and $\mu,\nu\in \mathcal P(M)$
\[W_s(\hat P_{t,s}\mu, \hat P_{t,s}\nu) \leq W_t(\mu,\nu) \;.\]
\end{itemize}
Here, $ P_{t,s}$ denotes the heat propagator and $\hat P_{t,s}$ its dual, i.e. $P_{t,s}f$ gives the solution at time $t$ of $\partial_t u = \Delta_t u$ with initial datum $u=f$ at time $s$. Further $\Delta_t$, $W_t$, and $\ent_t$ denote the Laplace-Beltrami operator, the Wasserstein distance, and the entropy associated with the metric tensor $g_t$. 
Sturm \cite{Sturm2018Super} and Kopfer and Sturm \cite{Kopfer-Sturm2018} have used this to define a synthetic notion of super-Ricci flow for time-dependent families of metric measure spaces $\left(X, d_t, m_t\right)_{t\in I}$. Suitable regularity assumptions are needed for the second notion in order to ensure existence of the heat flow, as we shall discuss below.
\medskip

The goal of the present paper is to develop and analyse synthetic notions of \emph{Ricci flow} for time-dependent families of metric measure spaces. To this end, we complement the above notions of super-Ricci flow with corresponding notions of sub-Ricci flow. This will be achieved by reversing the inequalities in ($1_{\sf dyn}$) and ($2_{\sf dyn}$) up to an arbitrarily small error for sufficiently localised transports and small times. We build on recent ideas and results in the static case \cite{sturm2017remarks} to encode Ricci upper bounds. This has also been used in the Lorentzian setting \cite{Mondino-suhr} to give characterisations of the Einstein equations in terms of optimal transport. Locally reverting ($1_{\sf dyn}$) to characterise sub-Ricci flows has already been proposed in \cite{Sturm2018Super}. Let us now specify the setting and the notions of Ricci flow we consider and describe our main results.
\medskip

\subsection{Setting and main results}

Let $(X,d_t,\m_t)_{t\in I}$ be a time-dependent family of metric measure spaces with $I\subset \R$ an interval. That is, $X$ is a Polish space, and $(d_t)$, $(\m_t)$ are Borel families of distance functions inducing the given topology on $X$ and locally finite Radon measures on $X$. Let $W_t$ denote the $L^2$-Wasserstein distance on the space $\p_t(X)$ of probability measures with finite second moment w.r.t. $d_t$ and $\ent_t$ denote the relative entropy w.r.t. $\m_t$, i.e. for a probability measure $\mu\in \p_t(X)$ we set
\[\ent_t(\mu)=\int\rho\log\rho \d \m_t\;,\]
provided $\mu=\rho \m_t$ is absolutely continuous and $\ent_t(\mu)=+\infty$ else. 
\medskip

\emph{Synthetic notions of Ricci flow.}\smallskip

The first notion of Ricci flow we consider is based on dynamic convexity and almost concavity of the entropy and was in a slightly different form already proposed in \cite{Sturm2018Super}.\medskip

The family $(X,d_t,\m_t)_{t\in I}$  is called a {\bf weak Ricci flow} if it satisfies
\begin{enumerate}[(i)]
\item[($w_{super}$)] (\emph{weak super-Ricci flow}): the entropy is strongly dynamically convex, i.e. for a.e. $t\in I$ and every $W_t$-geodesic $(\mu^a)_{a\in [0,1]}$ in $\p_t(X)$ with finite entropy at endpoints, the function $a\mapsto \ent_t(\mu^{a})$ is absolutely continuous on $[0,1]$ and
\begin{align}\label{ineq:weaksuperRF-intro}
    \partial^+_a \ent_t(\mu^a)\lvert_{a=1-}-\partial^-_a \ent_t(\mu^a)\lvert_{a=0+}\ge -\frac12\partial^-_tW_{t^-}^2(\mu^0,\mu^1),
\end{align}
holds, and

\item[($w_{sub}$)] (\emph{weak sub-Ricci flow}) for a.e. $t$ and every $\varepsilon>0$, there exists an open cover $\{U_i\}$ such that for all $i$ and every open subsets $V_0,V_1\subset U_i$, there exists a $W_t$--Wasserstein geodesic $(\mu^a)$ with $\spt\mu^0\subset V_0$, $\spt\mu^1\subset V_1$, and
\begin{align}\label{eq:EntSub-intro}
    \partial_a^+\ent_t(\mu^a)\lvert_{a=1-}-\partial_a^-\ent_t(\mu^a)\lvert_{a=0+}\le -\frac12\partial^-_t W_{t^-}^2(\mu^0,\mu^1)+\varepsilon W_t^2(\mu^0,\mu^1).
\end{align}
\end{enumerate}

Here, $\partial^\pm_a\big|_{a=c\mp}$ denotes the upper/lower right/left derivative.

The second notion of Ricci flow we consider is based on expansion properties of the heat flow. Let us assume that $(X,d_t,\m_t)_{t\in I}$ satisfies additional regularity properties as specified in \cref{asm:dualheatflow}. Namely, we require a uniform log-Lipschitz control in the time parameter $t$ on the distance functions $d_t$ and the measures $\m_t$ and that for each fixed $t$, the space $(X,d_t,\m_t)$ satisfies the Riemannian curvature-dimension condition RCD$(K,N)$ for some $K\in \R$ and $N\in[1,\infty)$. Under these conditions, Kopfer and Sturm \cite{Kopfer-Sturm2018} have shown the existence of a (dual) heat flow $(P_{t,s})$ (see Section \ref{sec:HF} for more details).\medskip

We call $(X,d_t,\m_t)_{t\in I}$  a {\bf rough Ricci flow} if it satisfies
\begin{enumerate}[(i)]
    \item[($r_{super}$)] (\emph{rough super-Ricci flow}) for all $s\leq t$ and all $x,y\in X$
    \begin{align}\label{ineq:defroughsuperRicci-intro}
        W^2_s(\ap{\delta_x},\ap{\delta_y})\le d^2_t(x,y)\;;
    \end{align}
    \item[($r_{sub}$)] (\emph{rough sub-Ricci flow}) for a.e. $t$ and every $\varepsilon>0$ there exists an open cover $\{U_i\}_{i\in\N}$ such that for every $i$ and all $x,y\in U_i $, there exists $s_0<t$ so that
    \begin{align}\label{ineq:defroughsubRicci-intro}
        W^2_s(\ap{\delta_x},\ap{\delta_y})\ge d^2_t(x,y)-\varepsilon d^2_t(x,y)(t-s)
    \end{align}
    for every $s\in (s_0,t)$.
\end{enumerate}

Note that the notions of weak/rough sub-Ricci flow above both formalise the idea of locally reverting the inequalities in ($1_{\sf dyn}$) and ($2_{\sf dyn}$) respectively up to a small error. It will further be convenient to consider the following localised quantities. 

For $\varepsilon>0$ and $t\in I$, $x,y\in X$ we set
 \[
  \eta_{\varepsilon}(t,x,y)\coloneqq  \inf\Big\{
\frac1{W_t^2(\mu^0,\mu^1)}\cdot\Big[ \partial^{+}_a \ent_t(\mu^a)\big|_{a=1}-\partial_a^{-} \ent_t(\mu^a)\big|_{a=0}
+\frac12\partial_{t}^{-}W_{t^-}^2(\mu^0,\mu^1)\Big]\Big\}\;,
\]
where the infimum is taken over all $W_t$-geodesics 
$(\mu^a)_{a\in[0,1]}$ such that $\mu^0$ and $\mu^1$ have finite entropy and are supported in balls of radius $\varepsilon$ w.r.t. $d_t$ around $x$ and $y$ respectively. Moreover, we set
\[
\eta(t,x,y)\coloneqq \sup_{\varepsilon> 0}\eta_{\varepsilon}(t,x,y)\;,\quad
  \eta^*(t,x)\coloneqq\limsup_{y,z\to x}\eta(t,y,z)\;.
\]
Bounds on these quantities describe the dynamic convexity/concavity of the entropy for transports between measures concentrated around the points $x,y$. Similarly, we define for $t\in I$ and $x,y \in X$
\begin{equation}\label{eq:theta-intro}
    \vartheta(t,x,y)\coloneqq -\liminf_{s\nearrow t}\frac1{t-s}\log\frac{W_s(\ap\delta_x,\ap\delta_y)}{d_t(x,y)}\;,
\end{equation}
as well as
\begin{equation}\label{eq:thetastar-intro}
    \vartheta^*(t,x)=\limsup_{y,z\to x}\vartheta(t,y,z)\;,
\end{equation}
describing the short time asymptotics of the transport cost between two heat flows starting from $x,y$ respectively.

Our first main result shows that for families of smooth weighted Riemannian manifolds the notions of weak/rough Ricci flow indeed yield a characterisation of the classical notion of (weighted) Ricci flow. Let $(M,g_t)_{t\in I}$ be a smooth family of closed Riemannian manifolds with Riemannian distance $d_t$ and let $(f_t)_{t\in I}$ be a smooth family of functions on $M$. The associated metric measure spaces $(M,d_t, e^{-f_t}\vol_{g_t})_{t\in I}$ will be called a \emph{smooth flow}. The weighted Ricci tensor is defined by
\[\Ric_{f_t}:= \Ric_{g_t}+\Hess f_t\;.\]

\begin{theorem}\label{thm:consistency-intro}
Let $(M,d_t, e^{-f_t}\vol_{g_t})_{t\in I}$ be a smooth flow. Then the following are equivalent:
\begin{enumerate}[(i)]
\item The family of metric measure spaces $(M, d_t,e^{-f_t}\vol_{g_t})$ is a weak Ricci flow;
\item The family of metric measure spaces $(M, d_t,e^{-f_t}\vol_{g_t})$ is a rough Ricci flow;
\item for all $t\in I$ and $x,y\in M$ we have
\[\eta(t,x,y) \geq 0\;,\quad \eta^*(t,x)\leq 0\;;\]
\item for all $t\in I$ and $x,y\in M$ we have
\[\vartheta(t,x,y) \geq 0\;,\quad \vartheta^*(t,x)\leq 0\;;\]
\item $(M,g_t,e^{-f_t}\vol_{g_t})$ is weighted Ricci flow, i.e.
\[\partial_t g_t = - 2\Ric_{f_t}\;.\]
\end{enumerate}
\end{theorem}

In fact, independently the first property in (iii) or (iv) or weak/rough super-Ricci flow characterises weighted super-Ricci flows $\partial_t g_t \geq - 2\Ric_{f_t}$, while the second property in (iii) or (iv) or weak/rough sub-Ricci flow equivalently characterises weighted sub-Ricci flow $\partial_t g_t \leq - 2\Ric_{f_t}$, see Theorems \ref{thm:smoothrough_infty}, \ref{thm:weak-char} below.  The characterisation of smooth weighted super-Ricci flows through weak/rough super-Ricci flows has already been shown in \cite{Sturm2018Super, Kopfer-Sturm2018} and the characterisation of weighted sub-Ricci flow through weak sub-Ricci flow has been sketched. The characterisation of sub-Ricci flow through short time asymptotics of the heat flow we present here is genuinely new. We give detailed proofs of all characterisations above in Section \ref{sec:smoothRF}. 
\medskip

The characterisation above is obtained through a detailed local analysis of the short-time asymptotics of transport costs along the heat flow and of the dynamic convexity/almost concavity of the entropy along geodesics. Let us define the Ricci flow excess of a smooth flow given for $t\in I$ and $x,y\in M$ by
\[
{\rm RFex}(t,x,y) :=\inf \frac{1}{d_t(x,y)^2}\int_0^1\big[\Ric_{f_t}(\dot\gamma^a) +\frac12 \partial_tg_t(\dot\gamma^a)\big] {\rm d}a\;,
\]
where the infimum is taken over all geodesics from $x$ to $y$. We show (see Corollaries \ref{cor:theta+}, \ref{cor:eta+}) the following estimates:
\begin{theorem} We have for all $t\in I$, $x,y\in M$:
\begin{align}\label{eq:excess-intro1}
{\rm RFex}(t,x,y) \leq \vartheta(t,x,y)\;.
\end{align}
For every $t_0\in I$ there is $\varepsilon>0$ such that for all $t\in I$, $x,y\in M$ non-conjugate with $|t-t_0|, d_t(x,y)<\varepsilon$: 
\begin{align}\label{eq:excess-intro2}
 \vartheta(t,x,y) \leq {\rm RFex}(t,x,y) + \sigma_t \tan^2\big(\sqrt{\sigma_t} d_t(x,y)\big)\;,
\end{align}
where $\sigma_t$ is an upper bound on the modulus of the Riemann tensor along the geodesic from $x$ to $y$. The same estimates hold for $\eta$ in place of $\vartheta$.
\end{theorem}
\medskip

\emph{Unweighted/non-collapsed Ricci flows}

Note that the weight $e^{-f_t}$ on the volume measure presents an additional degree of freedom that influences the evolution of the metric in a weighted Ricci flow $\partial_tg_t=\Ric_{f_t}$ but whose own evolution is not constraint. It is therefore desirable to be able to single out \emph{unweighted} Ricci flows through a synthetic characterisation. This can be achieved by a dimensional refinement of the notion of weak/rough super-Ricci flow proposed in \cite{Sturm2018Super,Kopfer-Sturm2018}.  A family of m.m.s. $\left(X,d_t,\m_t\right)_{t\in I}$ is called a \emph{weak $N$-super-Ricci flow} for $N\in [1,\infty]$ if the inequality in $(w_{super})$ above is strengthened to 
\[
 \partial^+_a \ent_t(\mu^a)\lvert_{a=1-}-\partial^-_a \ent_t(\mu^a)\lvert_{a=0+}\ge -\frac12\partial^-_tW_{t^-}^2(\mu^0,\mu^1) +\frac{1}{N} \Big|\ent_t(\mu^1)- \ent_t(\mu^0)\Big|^2\;.
\]
Similarly, we call it a \emph{rough $N$-super-Ricci flow} if for all $s<t$ and $\mu,\nu\in \mathcal P_t(X)$ we have
\[
        W_s^2(\ap\mu,\ap\nu)\le W_t^2(\mu,\nu)-\frac2N\int_{[s,t]}[\ent_r(\hat P_{t,r}\mu)-\ent_r(\hat P_{t,r}\nu)]^2\d r\;.
\]
Now, a smooth flow $(M,g_t,e^{-f_t}\vol_{g_t})$ of $n$-dimensional manifolds is a weak/rough $N$-super-Ricci flow, if and only if we have
\[\partial_tg_t\geq -2 \Ric_{N,f_t}\;,\quad \Ric_{N,f_t}:= \Ric_{g_t} +\Hess f_t -\frac{1}{N-n}\nabla f_t\otimes \nabla f_t\;,\]
where the latter is the so-called weighted-$N$-Ricci tensor. We then obtain the following synthetic characterisation of Ricci flows. 
\begin{corollary}\label{cor:non-collapsed Ricciflow}
A smooth flow $(M,d_t, e^{-f_t}\vol_{g_t})_{t\in I}$ is a \emph{Ricci flow}, i.e. $\partial_tg_t=-2\Ric_{g_t}$ and $f_t$ is constant for all $t$, if and only if it is a weak/rough sub-Ricci flow and a weak/rough $N$-super-Ricci flow for some $N\in[1,\infty)$.
\end{corollary}
This is due to the observation that the combination of the bounds
\[ -\Ric_{f_t}\geq  \frac12\partial_tg_t\geq -\Ric_{N,f_t}= -\Ric_{f_t}+\frac{1}{N-n}\nabla f_t\otimes\nabla f_t\]
for some $N\geq n$ necessarily implies that $\nabla f_t\equiv 0$. Based on this results we call a family of m.m.s $\left(X,d_t,\m_t\right)_{t\in I}$ a \emph{non-collapsed weak/rough Ricci flow} if it is a weak/rough sub-Ricci flow and weak/rough $N$-super-Ricci flow for some finite $N$. We conjecture that such flows are indeed non-collapsed in the sense that the reference measure $\mathfrak m_t$ is a multiple of the Hausdorff measure w.r.t. $d_t$ for a.e. $t$.
\medskip

\emph{Relating notions of Ricci flow}\smallskip

Our second set of results concerns the relation between the different notions of synthetic Ricci flow considered in this paper. 
Let $(X,d_t,\m_t)_{t\in I}$ be a time-dependent family of m.m.s. that satisfies the regularity properties specified in \cref{asm:dualheatflow}. 
From the work of Kopfer and Sturm \cite{Kopfer-Sturm2018} it is known that the notions weak and rough super-Ricci flow are equivalent. 
On the other hand, weak and rough sub-Ricci flow turn out not to be equivalent. 
Indeed, for static spaces it has been shown in \cite{sturm2017remarks} that upper Ricci bounds in terms of transport cost asymptotics for the heat flow imply upper Ricci bounds in terms of almost concavity of the entropy. 
However, the latter notion does not detect the positive Ricci curvature in the vertex $o$ of a cone while the former does. 
More precisely, a Euclidean cone is Ricci flat in terms of convexity/almost concavity of the entropy, while $\vartheta^+(x,o)=+\infty$ for any point $x$ as shown in \cite{Erbar-Sturm}.
\smallskip

We introduce a relaxation $\vartheta^\flat\leq \vartheta^*$ of the quantity $\vartheta^*$ in \eqref{eq:thetastar-intro} obtained by considering transports between meassures in schrinking balls around $x,y$ as in the definition of $\eta^*$, see Sec. \ref{sec:comparingRF}. Generalising the results in \cite{sturm2017remarks} to a dynamic setting, we show 

\begin{theorem}\label{thm:relation-intro}
For a.e. $t\in I$ we have $\vartheta^\flat(t,x)= \eta^*(t,x)$ for all $x\in X$. 
In particular, any rough sub-Ricci flow is also a weak sub-Ricci flow.
\end{theorem}

We will see below that the reverse implication fails, i.e. rough sub-Ricci flow is strictly stronger then weak sub-Ricci flow.

\subsection{Examples}

\emph{Smooth flows.} As already discussed, any smooth (sub/super)-Ricci flow $(M, d_t,e^{-f_t}\vol_{g_t})_{t\in I}$ is also a weak/rough (sub/super)-Ricci flow, this holds in particular for smooth Ricci flows starting from non-smooth initial data as considered e.g. in \cite{Sim12}, i.e $I=[t_0,t_1)$ and $(X,d_t,\m_t)$ approximates to a non-smooth m.m.s. as $t\searrow t_0$.
\smallskip

\emph{Gaussian weights.} Consider $(X,d_t,\m_t)=(\R^n, d_t, e^{-f_t}\mathcal L^n)$ with
\[
    f_t(x)=\frac12\langle x,a_t x\rangle+\langle x,b_t\rangle+c_t\;, 
\]
where $a: I\to \R^{n\times n}$, $b: I\to \R^n$, $c:I\to\R$ are suitably regular functions and the distance $d_t$ is induced by the inner product $\langle \cdot, A_t\cdot\rangle$ where $A:I\to \R^{n\times n}$ is positive definite for all $t\in I$. Then $(X,d_t,\m_t)_{t\in I}$ is a weak/rough super-resp.~sub-Ricci flow if and only if 
\[\dot A_t \geq -2 a_t\;,\quad \text{resp.}\quad \dot A_t\leq -2a_t\;.\]
However, it will not be a $N$-super-Ricci flow for some $N\in [n,\infty)$ unless $a\equiv 0$ and $b\equiv 0$.
\smallskip

\emph{Cones and suspensions.} Let $(M,g_M)$ be a $n$-dimensional Riemannian manifold with $\mathrm{diam}(M)\leq \pi$ and consider the spherical suspicion $\Sigma(M)=M\times [0,\pi]/\sim$ obtained by contracting $\mathcal{S}=M\times\{0\}$ and $\mathcal{N}=M\times \{\pi\}$ to a point and equipped with the metric $d_0$ given by 
\[
    \cos\left(d_0((x,s),(x',s')) \right)\coloneqq \cos s\cos s'+\sin s\sin s'\cos(d_{M}(x,x'))\;,
\]
and volume measure $\m_0(\d x,\d s)\coloneqq \mathrm{vol}_M(\d x)\otimes (\sin^n s\d s)$. Assume that $(M,g_M)$ is an $n$-dimensional Einstein manifold with $\Ric_{g_M}\equiv (n-1)g_M$. We show in Section \ref{sec:examples} that the time-scaled spherical suspension $(\Sigma(M),d_t,\m_t)_{t\in [0,\frac{1}{2n})}$ with
    \[
        d_t\coloneqq(1-2nt)^{\frac12}d_0,\quad \m_t\coloneqq(1-2nt)^{\frac{n+1}{2}}\m_0
    \]
is a weak Ricci flow. However, it is a rough Ricci flow if and only if $(\Sigma(M), d_0, \m_0)$ is the unit sphere $\mathbb{S}^{n+1}$ with the round metric and a multiple of volume measure.

A similar result can be obtained for Euclidean cones. This yields a dynamic counterpart to the observation from \cite{Erbar-Sturm} that an RCD$(K,N')$  space that is a Euclidean $N$-cone has a synthetic upper Ricci bound in the rough sense if and only it is isomorphic to Euclidean space $\R^N$.

\subsection{Heuristics} 
Finally, let us briefly give some intuitive ideas behind the notions of rough and weak Ricci flow and the proof of Theorem \ref{thm:consistency-intro}.
Let $\left(M, d_t, e^{-f_t}\vol_{g_t}\right)_{t\in I}$ be a smooth flow. A well-known formal computation in optimal transport \cite{Mccann2001,vonRenesseSturm} shows that the second derivative of the entropy along a Wasserstein geodesic $(\mu^a)$ w.r.t.~$dt_t$ is given by
\[
\frac{d^2}{da^2}\ent_t(\mu^a) = \int_M\Big[\Ric_{f_t}(\nabla\psi^a) +\|\Hess\;\psi^a\|^2_{\rm HS}\Big]{\rm d}\mu^a\;,
\]
where $(\psi^a)$ is the family of Kantorovich potentials associated with the geodesic $(\mu^a)$. Assuming (weighted) super-Ricci flow $\partial_tg_t \geq -2\Ric_{f_t}$ and neglecting the positive term $\|\Hess \psi^a\|_{\rm HS}^2$ formally yields the dynamic convexity inequality in $(w_{super})$. However, to obtain a concavity estimate under the assumption of (weighted) sub-Ricci flow $\partial_tg_t \leq -2\Ric_{f_t}$, the Hessian term can not be neglected as it has the wrong sign. 
The key idea in \cite{sturm2017remarks, Sturm2018Super, Mondino-suhr} on which we build is to consider transports that are sufficiently concentrated around a single geodesic on $M$ and where $\psi^a$ thus behaves almost linearly, so that the Hessian term becomes an arbitrarily small error. This leads to the almost concavity estimate in $(w_{sub})$. A similar reasoning applies to the non-expansion of the heat flow under super-Ricci flow and the almost non-constraction under sub-Ricci flow for sufficiently ``linear" transports.

In Section \ref{sec:smoothRF} we will make this ideas precise and rigorous. The key technical challenge in the proof of consistency will be to carefully construct suitable transports via their Kantorovich potentials and to obtain sufficient control on the Hessians of the latter.
\medskip

{\bf Plan of the paper.}
In Section 2 we collect notions and results concerning optimal transport and the heat flow on time-dependent metric measure spaces. In Section 3 we introduce the synthetic notions of weak and rough super-/sub- Ricci flow for metric measure spaces. Consistency of these notions with classical Ricci flow for smooth time-dependent families of Riemannian manifolds will be established in Section 4. In Section 5 we investigate the relation between weak and rough Ricci flows. We discuss several examples in Section 6. In the Appendix we collect results on the construction of smooth solutions to the Hamilton-Jacobi equation on time-dependent Riemannian manifolds, measure-theoretic prerequisits, and auxiliary results related to the curvature-dimension condition.
\medskip

{\bf Acknowledgements.}
M.E.~and T.S.~were funded through the SPP 2026 "Geometry at Infinity" by the Deutsche Forschungsgemeinschaft (DFG) - project number 441873017.

 \section{Preliminaries}\label{sec:preliminaries}
 We collect several preliminaries on optimal transport, the concept of dynamic convexity and the heat flow on time-dependent metric-measure spaces.
\subsection{Optimal transport and displacement convexity}
Let $(X,d)$ be a complete and separable metric space. A (constant speed) geodesic in $X$ is a curve $(\gamma^a)_{a\in[s,t]}$ such that $d(\gamma^a,\gamma^b)=\frac{b-a}{t-s}d(\gamma^s,\gamma^t)$ for all $s\leq a\leq b\leq t$. We denote by $\geo(X)$ the space of all geodesics $(\gamma^a)_{a\in[0,1]}$ in $X$ equipped with the supremum distance.

Let $\mathcal P_2(X)$ denote the set of all Borel probability measures on $X$ with finite second moment. For $\mu,\nu\in \mathcal P_2(X)$ we define the $2$-Wasserstein distance by
\begin{equation}
    W^2(\mu,\nu)\coloneqq \inf\left\{\int_{X\times X} d^2(x,y)\d \sigma(x,y)\right\},
\end{equation}
where the infimum is taken over all couplings $\sigma$ of $\mu$ and $\nu$.
The above infimum is achieved, and any minimizer is called an optimal coupling or an optimal (transport) plan.
The set of all optimal plans between $\mu$ and $\nu$ is denoted by $\opt(\mu,\nu)$.

The Wasserstein space $(\mathcal P_2(X),W)$ is complete and separable and it is a length (resp. geodesic) space if and only if so is $(X,d)$. A curve $(\mu_t)$, $t\in[0,1]$, in the Wasserstein space is a (constant speed) geodesic if and only if there exists a measure $\pi\in \mathcal{P}(\geo(X))$ such that $(e_s,e_t)_\#\pi\in \opt(\mu_s,\mu_t)$ for all $s,t\in[0,1]$,
where $e_t$ denotes the evaluation map $\gamma\mapsto \gamma_t$.
Such a measure $\pi$ is called an optimal dynamical plan, and the set of all dynamical plans is denoted by $\og(\mu,\nu)$.

The Wasserstein distance can be expressed by a dual maximisation problem. To this end we recall that the $c$-transform of a function $\phi:X\to\R\cup\{\pm\infty\}$ for the cost $c=d^2/2$ is the function $\phi^c:X\to\R\cup\{\pm\infty\}$ defined by 
\begin{equation}\label{eq:ctransform}
\phi^c(y)\coloneqq \inf_{x\in X}[d^2(x,y)/2-\phi(x)]\;.
\end{equation}
$\phi^c$ is also called the conjugate function of $\phi$ (w.r.t. $d^2$/2). 
A function $\phi$ is called $d^2/2$-concave if $\phi=\phi^{cc}$. In this case, we call $(\phi,\phi^c)$ a conjugate pair. The Kantorovich duality states that 
\begin{equation}\label{eq:KantorovichDuality}
    \frac12 W^2(\mu,\nu)=\sup \Big\{\int\phi\d \mu+\int\psi\d \nu\Big\}=\sup \Big\{\int\phi\d \mu+\int\phi^c\d \nu\Big\}\;,
\end{equation}
  where the first supremum is taken over all pairs of functions $\phi,\psi$ such that $\phi(x)+\psi(y)\leq d^2(x,y)/2$ for all $x,y$, and the second supremum is taken over all $d^2/2$-concave functions $\phi$. If $(\phi,\psi)$ attains the supremum then $\psi=\phi^c$ with $\phi$ $c$-concave and the pair is called a pair of Kantorovich potentials for $\mu,\nu$.
\smallskip

\subsection{Metric measure spaces and curvature-dimension conditions}
A metric measure space $(X,d,\m)$ consists of a Polish metric space $(X,d)$ and a locally finite Borel measure $\m$ on $X$.
The relative entropy w.r.t. $\m$ of a Borel probability measure $\mu\in P(x)$ is defined by
\[\ent(\mu|\m):= \int \rho\log\rho \d\m\;,\]
provided $\mu=\rho\m$ and $(\rho\log\rho)^-$ is integrable. Otherwise, we set $\ent(\mu|\m)=+\infty$.

We say that $(X,d,\m)$ satisfies $\cd(K,N)$ for some $K\in\R$ and $N\in[1,\infty]$ if it has Ricci curvature bounded below by $K$ and dimension bounded above by $N$ in the sense of Lott--Sturm--Villani; see \cite{LottVillani09,sturm2006--1,Sturm06-2}. 
In the case $N=\infty$, this means that the entropy is $K$-convex along at least one Wasserstein geodesic connecting two given measures.

On a CD$(K,N)$ space the Cheeger energy of a function $f\in L^2(X,\m)$ is defined as 
  \[
  \mathrm{Ch}(f)\coloneqq \inf\left\{\liminf_{k\to\infty}\frac{1}{2}\int_X \mathrm{lip}(f_k)^2\d \m:f_k\in\mathrm{Lip}(X,d),f_k\to f\text{ in $L^2(X,\m)$}\right\},
  \]
where $\mathrm{lip}(f)(x)\coloneqq \limsup_{y\to x}\frac{|f(y)-f(x)|}{d(x,y)}$ denotes the local Lipschitz constant of $f$. $\ch$ is a convex and l.s.c. function on $L^2(X,\m)$. The Laplacian operator $\Delta f$ is defined as the element of minimal norm in the subdifferential of $\ch$ at $f$, see \cite{AGSInvention}. The space $(X,d,\m)$ is called \emph{infinitesimally Hilbertian} if $\mathrm{Ch}$ is a quadratic form on $L^2(X,\m)$. In this case, $\ch$ is a Dirichlet form and the Laplacian is the corresponding generator characterised by $\ch(u,v)=-\int_X \Delta u\cdot v\d \m$ for all $u\in D(\Delta)$ and $v\in D(\ch)$.
The space is said to satisfy the Riemannian curvature-dimension condition $\rcd(K,N)$ if it is infinitesimally Hilbertian and satisfies $\cd(K,N)$; cf. \cite{Ambrosio-Gigli-Mondino-Rajala,AGSRCD2014,GigliRCD}.
\medskip

Let $(M,g)$ be a smooth, complete Riemannian manifold of dimension $n$ and Riemannian distance $d$, let $f:M\to\R$ be a smooth weight function and consider the measure $\nu=e^{-f}\mathrm{vol}_g$.

The following result discusses the behaviour of the relative entropy $\ent(\cdot|\nu)$ along Wasserstein geodesics on $M$.
We refer to \cite{Sturm05} for a proof in the unweighted case $f=0$ and $N=n$. The general case can be obtained from there as in \cite[p. 381]{Villani}.

\begin{theorem}\label{thm:displaceconvex}
Let $(M,g)$ be a complete, weighted smooth Riemannian manifold of dimension $n$ with the reference measure $\nu\coloneqq e^{-f}\mathrm{vol}_g$.
    Let $(\mu^a)_{a\in[0,1]}$ be a $2$-Wasserstein geodesic by absolutely continuous probability measures and $\pi\in\mathrm{OptGeo}(\mu^0,\mu^1)$.
    Denote by $\rho^a$ the density of $\mu^a$ with respect to $\nu$ for each $a\in[0,1]$.
    Then there exists a map $\ell:[0,1]\times M\to \R$ such that
    \begin{enumerate}
        \item for all $a\in [0,1]$, the function $x\mapsto \ell(a,x)$ is Borel and 
        \begin{equation}
            \rho^0(\gamma^0)=\rho^a(\gamma^a)\cdot e^{-\ell(a,\gamma^0)},\quad \text{for $\pi$-a.e. $\gamma$}.
        \end{equation}
        \item for all $N\in[n,\infty]$ and $x\in M$, the function $a\mapsto \ell(a,x)$ is semiconvex on $[0,1]$ and continuous on $(0,1)$. 
        For $x=\gamma^0$, the centered second order derivative $\underline{\partial}^2_a\ell(a,x)$ satisfies
        \begin{equation}\label{ineq:2rdDlogJacobian}
    \underline{\partial}^2_a\ell(a,\gamma^0)\geq \frac{(\partial_a\ell(a,\gamma^0))^2}{N}+\Ric_{N,f}(\dot\gamma^a,\dot\gamma^a).
        \end{equation}
    \end{enumerate}
    Moreover, when $\Ric_{N,f}\geq K$ for some $K\in \R$ (see \eqref{eq:N-Riccitensor} for the notation), then for $\mu^0,\mu^1$ having finite entropy, $a\mapsto \ent(\mu^a)$ is absolutely continuous and semi-convex on $[0,1]$, and satisfies
    \begin{align}
    \ent(\mu^a)&=\int \log\rho^0(\gamma^0)+\ell(a,\gamma^0)\d\pi(\gamma)\\
        \ent(\mu^1)-\ent(\mu^0)&=\iint \partial_a\ell(a,\gamma^0)\d a\d \pi(\gamma)\label{eq:1orderDentropy}\\
        \frac{\d^2 }{\d a^2}\ent(\mu^a)|_{a=\tau }&\geq  \int \frac{(\partial_a\ell(a,\gamma^0))^2}{N}+\Ric_{N,f}(\dot\gamma^a,\dot\gamma^a)\d \pi(\gamma).\label{ineq:2orderDentropy}
    \end{align}
\end{theorem}

The centered second order derivative appearing in the statement above is defined as
\begin{align}
    \underline{\partial}_t^2h(t)\coloneqq \liminf_{s\to 0}\frac{1}{s^2}\cdot\big( h(t+s)+h(t-s)-2h(t)\big).
\end{align}

\subsection{Dynamic convexity}
In this subsection, we revisit several notions introduced by Sturm in \cite{Sturm2018Super} on the analysis of time-dependent metric spaces.

Let $I$ be a left open interval and $(X,d_t)_{t\in I }$ a one-parameter family of geodesic metric spaces with $V\colon I\times X\to (-\infty,+\infty]$.
For $t\in I $, we write $V_t\coloneqq V(t,\cdot)$ and $D(V_t)\coloneqq \{x\in X:V_t(x)<\infty\}$.
\medskip

Throughout this paper, we will adopt the following notation:
Given a function $u:I \to \R$, we define the \emph{upper left} and \emph{lower right} derivative of $u$ by 
	\begin{align}
    \partial^+_t u(t-)\coloneqq\limsup\limits_{s\nearrow t}\frac{1}{t-s}(u(t)-u(s)),\qquad \partial^-_t u(t+)\coloneqq\liminf\limits_{s\searrow t}\frac{1}{t-s}(u(t)-u(s))\;.
	\end{align}
Analogously, we define the \emph{lower left} and \emph{upper right} derivative of $u$.    
    
\begin{definition}[dynamic convexity] Given a time-dependent geodesic spaces $(X,d_t)_{t\in I}$, a function $V\colon I\times X\mapsto (-\infty,\infty]$ is called
\begin{enumerate}
    \item \emph{strongly/weakly dynamically $(0,N)$-convex} if for a.e. $t\in I$ and every $x^0,x^1\in D(V_t)$, $\forall/\exists$ $d_t$-geodesic $(x^a)_{a\in[0,1]}$ from $x^0$ to $x^1$, the function $a\mapsto V_t(x^a)$ is u.s.c. on $[0,1]$, ac on $(0,1)$ and 
    \begin{equation}\label{ineq:dynamicalconvex}
        \partial_a^+V_t(x^{1-})- \partial_a^-V_t(x^{0+})\geq -\frac12 \partial_t^-d_{t^-}^2(x^0,x^1)+\frac{|V_t(x^0)-V_t(x^1)|^2}{N}.
    \end{equation}
    \item \emph{locally strongly/weakly dynamically $(0,N)$-convex} if for a.e. $t\in I$ and $x\in X$, there exists $r>0$ s.t. for every $x^0,x^1\in D(V_t)\cap B_t(x,r)$, $\forall/\exists$ $d_t$-geodesic $(x^a)_{a\in[0,1]}$ from $x^0$ to $x^1$, the function $a\mapsto V_t(x^a)$ is u.s.c. on $[0,1]$, ac on $(0,1)$ and satisfies \eqref{ineq:dynamicalconvex}.
\end{enumerate}
When $N=\infty$, the last term in \eqref{ineq:dynamicalconvex} is understood to be zero and $V$ is called (locally) weakly/strongly dynamically convex.
\end{definition}

\begin{definition}[Upper regular] A function $V\colon X\to (-\infty,\infty]$ is called \emph{upper regular} if for each geodesic $(\gamma^a)_{a\in[0,1]}$ in $X$ with $\gamma^0,\gamma^1\in D(V)$ the function $u=V\circ\gamma$ is u.s.c. on $[0,1]$, ac on $(0,1)$ and satisfies $\partial^+_au(a-)\leq \partial_a^-u(a+)$ for all $a\in (0,1)$ as well as\footnote{Note that the condition \eqref{ineq:upperregular} is slightly different from the one used in \cite{Sturm2018Super}.}
\begin{align}\label{ineq:upperregular}
       \liminf_{a\nearrow 1}\partial_\tau^-u(a+)\leq \partial_\tau^-u(1-),\qquad \limsup_{a\searrow 0}\partial_\tau^+u(a-)\geq \partial_\tau^+u(0+).
\end{align}
   If $V$ is a function on $I\times X$, then it is called upper regular if for a.e. $t\in I$, $V_t\coloneqq V(t,\cdot)$ is upper regular on $(X,d_t)$.
\end{definition}

We note that a convex function $V\colon [0,1]\to (-\infty,\infty]$ is upper regular.
In particular, this applies to the relative entropy on a metric measure space that satisfies the strong $\cd(K,\infty)$ space (in the sense of \cite{Rajala-Sturm14}), or that is essentially non-branching (e.n.b.) and satisfies $\cd(K,N)$ for some $K\in\R$, $N<\infty$.
In fact in both situations the Wasserstein geodesic between two measures having finite entropy and finite variance is unique, see \cite{Rajala-Sturm14} and \cite{cavalletti2017optimal}, and therefore the upper regularity follows from the weak $K$-convexity. 
On the other hand, any e.n.b. $\cd(K,\infty)$-space carrying upper regular entropy has to satisfy strong $\cd(K,\infty)$-condition, see \cref{sec:appendix3}.

\begin{definition}[log-Lipschitz control on metrics]
    We say that a family of distances on $X$ admits an \emph{upper or lower log-Lipschitz control} by a non-negative function $\kappa\in L^1_{\mathrm{loc}}(I)$ if for all $s<t$ and $x,y\in X$
    \begin{align}
        \log d_t(x,y)-\log d_s(x,y)\leq \int^t_s\kappa_r\d r \quad \text{or}\quad  \log d_t(x,y)-\log d_s(x,y)\geq -\int^t_s\kappa_r\d r
    \end{align}
    respectively. We say that the family has \emph{log-Lipschitz control} if it admits both lower and upper log-Lipschitz control.
    \end{definition}
    
    One can check that upper/lower log-Lipschitz control is equivalent to the requirement that for all $x,y\in X$, the map $t\mapsto d_t(x,y)$ is locally upper (resp. lower) absolutely continuous (see e.g. \cite{Ponomarev} for equivalent definitions) and we have
    \begin{align}
        \partial^+_t d_t(x,y)\leq \kappa_t d_t(x,y)\quad (\text{resp.}\quad \partial^-_t d_t(x,y)\ge -\kappa_t d_t(x,y))\qquad \text{for each } t\;.
    \end{align}
   
\subsection{Time-dependent metric measure spaces and heat flows}\label{sec:HF}

A \emph{time-dependent metric measure space} $(X,d_t,\m_t)_{t\in I}$ is a one-parameter Borel family of metric measure spaces, where $I$ is a left open interval and for each $t\in I$, $d_t$ is a geodesic metric generates the given topology of $X$.

We denote by $\mathcal{P}_t(X)$ the set of all probability measures on $(X,d_t)$ with finite second moment, by $W_t$ the $L^2$-Wasserstein distance w.r.t. $d_t$, and by $\ent_t$ the relative entropy w.r.t. $\m_t$.
\medskip

We call a time-dependent metric measure space $(X,d_t,\m_t)_{t\in I}$ \emph{uniformly comparable} if there exist constants $C,L>0$ s.t.
\begin{itemize}
    \item   for all $s,t\in I$ and $x,y\in X$, 
    \begin{equation}
        \left|\log \frac{d_t(x,y)}{d_s(x,y)}\right|\leq L\cdot |t-s|;
    \end{equation}
    \item measures $\m_t$ are mutually absolutely continuous with bounded, Lipschitz logarithmic densities; that is to say, there exists a reference measure $\m$ that for each $t\in I$ $\m_t=e^{-f_t}\m$, where functions $f_t$ satisfy $|f_t(x)|\leq C$, $|f_t(x)-f_t(y)|\leq Cd_t(x,y)$, and $|f_t(x)-f_s(y)|\leq L|t-s|$ for all $s,t\in I$, $x,y\in X$.
\end{itemize}
Our main assumption will be
\begin{assumption}\label{asm:dualheatflow}
    The time-dependent metric measure space $(X,d_t,\m_t)_{t\in I}$ is uniformly comparable and there exist $K\in\R$ and $N<\infty$ s.t. the static space $(X,d_t,\m_t)$ satisfies $\rcd(K,N)$-condition for each $t\in I$.
\end{assumption}

Under \cref{asm:dualheatflow} the results developed in \cite{AGSRCD2014,Lierl-JMPA,Sturm95} apply. In particular, we denote by $\ch_t$ and $\Delta_t$ the the Cheeger energy and Laplace operator of the static space $(X,d_t,\m_t)$.
It is useful to observe that when $(X,d_t,\m_t)_{t\in I}$ is uniformly comparable, various spaces including $L^2(X,\m_t)$, $\mathcal{P}_t(X)$, $D(\ent_t)$, $D(\mathrm{Ch}_t)$ do not depend on $t$.
In the following theorem, we collect results of heat flows on time-dependent metric measure spaces from \cite{Kopfer-Sturm2018}.
\begin{theorem}\label{thm:heatkernel}
    Let $(X,d_t,\m_t)_{t\in I}$ be given as in \cref{asm:dualheatflow}. Then
    \begin{enumerate}
        \item There exists a heat kernel $p$ on $\{(t,s,x,y)\in I^2\times X^2:t>s\}$, H\"older-continuous in all variables s.t. for each $s\in I$ and $h\in L^2(X,\m_s)$,
        \begin{equation}\label{eq:heatflow}
            (t,x)\mapsto P_{t,s}h(x)\coloneqq \int p_{t,s}(x,y)h(y)\d \m_s(y)
        \end{equation}
        is the unique solution to the heat equation $\partial_t u_t=\Delta_t u_t$ on $(s,T)\times X$ with $u_s=h$. 
        More precisely, denoting $u_t\coloneqq P_{t,s}h$, $(u_t)_{t\geq s}$ is a locally absolutely continuous curve in $L^2(X)$, $u_t\in D(\Delta_t)$ for a.e. $t$ and 
        \begin{equation}
            -\int^\tau_s \ch_t(u_t,w_t)\d t=\int^{\tau}_s\int_X \partial_t u_t\cdot w_t\d \m_t\d t,\quad \forall \tau>s
        \end{equation}
        for all absolutely continuous curve $(w_t)_{t\geq s}$ on $L^2(X)$.
       \item The heat kernels satisfies the propagator property
       \begin{equation}
           p_{t,r}(x,z)=\int p_{t,s}(x,y)p_{s,r}(y,z)\d \m_s(y),\quad \forall s\in (r,t),x,z\in X.
       \end{equation}
       The operators $\{P_{t,s}\}_{s\leq t}$ defined in \eqref{eq:heatflow} satisfy the propagator property $P_{t,r}=P_{t,s}\circ P_{s,r}$ for all $r\leq s\leq t$.
       \item the heat kernel admits upper Gaussian bound i.e. there exists a constant $C>0$ that for all $\tau\in I$ and $s<t$
       \begin{equation}\label{ineq:Gaussian}
           p_{t,s}(x,y)\leq \frac{C}{\m_{\tau}(B_{\tau}(x,\sqrt{|s-t|}))}\exp\left(-\frac{d^2_{\tau}(x,y)}{C|s-t|}\right).
       \end{equation}
       \item the heat kernel is Markovian 
       \begin{equation}
           \int p_{t,s}(x,y)\d \m_s(y)=1,\forall s<t,x\in X.
       \end{equation}
       Hence we can define dual propagators $\hat{P}_{t,s}:\mathcal{P}(X)\to \mathcal{P}(X)$ and dual heat flow of measures for all $s<t$, given by
       \begin{equation}\label{eq:dualheatflow}
           (\hat{P}_{t,s}\mu)\coloneqq\left[\int p_{t,s}(x,y)\d \mu(x) \right]\d \m_s(y).
       \end{equation}
       In particular, $\hat{P}_{t,s}\delta_x(\d y)=p_{t,s}(x,\d y)$ for all $x\in X$.
    \end{enumerate}
\end{theorem}

Note that, on any smooth family of compact Riemannian manifolds, the time-dependent heat kernel is smooth, so is the heat flow.
In the following, we summarize regularity results for dual heat flows on general metric measure spaces from \cite{Kopfer-Sturm2018}.
\begin{lemma}[Regularity of dual heat flows] \label{lemma:dualheatflow} Let $(X,d_t,\m_t)_{t\in I}$ be given as in \cref{thm:heatkernel}.
\begin{enumerate}
    \item $\hat{P}_{t,s}:\p(X)\to \p(X)$ is continuous with respect to the weak convergence.
    \item There exists $C>0$ s.t. for all $s,s'<t$, $\tau\in I$ and $\mu\in\p(X)$,
    \begin{equation}
        W^2_{\tau}(\hat{P}_{t,s'}\mu,\hat{P}_{t,s}\mu)\leq C\cdot |s-s'|.
    \end{equation}
\end{enumerate}
    Let $\mu,\nu\in \p_t(X)$, $t\in I$ and $\mu_s\coloneqq \hat{P}_{t,s}\mu$, $\nu_s\coloneqq \hat{P}_{t,s}\nu$.
    \begin{enumerate}
        \item[(3)] For all $s<t$, $\mu_s\in D(\ent_s)\cap \p_s(X)$.
        Moreover, denoting by $\rho_{t,s}$ the density of $\mu_s$ w.r.t. $\m_s$ (given by \eqref{asm:dualheatflow}), then $\rho_{t,s}\in C_b(X)\cap D(\ch_s)$.
        \item[(4)]\label{item:acdualheatflow} For any $\tau\in I$, the curve $s\mapsto\mu_s$ belongs to $\mathrm{AC}^2_{\mathrm{loc}}([0,t)\colon\p(X),W_\tau)$ and if $\mu\in D(\ent)$, then also to $\mathrm{AC}^2([0,t]\colon\p(X),W_\tau)$.
        \item[(5)] the function $s\mapsto W^2_s(\ap\mu,\ap\nu)$ is continuous on $[0,t]$ and absolutely continuous on $[0,r]$ for all $r<t$.
    \end{enumerate}
\end{lemma}
    
\subsection{Flows of smooth Riemannian manifolds}\label{sec:smoothRFpreli}

Let $(M,g)$ be a smooth complete Riemannian manifold of dimension $n$ and $f:M\to\R$ a smooth weight function. For $N\in [n,\infty]$, denote by $\Ric_{N,f}$ the \emph{weighted $N$-Ricci curvature tensor}:
\begin{align}\label{eq:N-Riccitensor}
    \Ric_{N,f}\coloneqq \Ric_{g}+\mathrm{Hess}_{g}f-\frac1{N-n}\nabla f\otimes \nabla f
\end{align}
the $N$-Ricci curvature tensor (also referred to as the $N$-Bakry--\'Emery curvature tensor). For $N=\infty$ the last term is understood as zero and we write $\Ric_f$ instead of $\Ric_{\infty,f}$.

We denote by $\Delta_f=\Delta-\nabla f\cdot \nabla$ the weighted Laplacian, where $\Delta$ denotes the Laplace-Beltrami operator of $(M,g)$. We recall the Bochner-Weitzenb\"ock identity (see e.g. \cite[Chapter 14]{Villani}) for the weighted Laplacian $\Delta_f$. For a smooth function $\psi$ on $M$ we have
\begin{align}\label{eq:NBochner}
    \L\frac{|\nabla \psi|^2}{2}-\nabla \psi\cdot \nabla(\L\psi)=&\frac{(\L\psi)^2}{N}+\Ric_{N,f}(\nabla \psi)+\big\lVert \nabla^2\psi-\left(\frac{\Delta \psi}{n}\right)I_n\rVert^2_{\mathrm{HS}}
    \\ &+\frac{n}{N(N-n)}\left[\left(\frac{N-n}n\right)\Delta\psi+\nabla f\cdot \nabla\psi\right]^2\;.
\end{align}
When $N=\infty$ this reduces to 
\begin{equation}\label{eq:inftyBochner}
    \Delta_f\frac{|\nabla \psi|^2}{2}-\nabla \psi\cdot \nabla(\L\psi)=\Ric_{f}(\nabla\psi)+\|\nabla^2\psi\|^2_{\mathrm{HS}},
\end{equation}
where $\|\cdot\|_{\mathrm{HS}}$ denotes the Hilbert-Schmidt norm and $\Ric(X)\coloneqq \Ric(X,X) $.
\medskip

Let now $(M,g_t)_{t\in I}$ be a smooth family of complete Riemannian manifolds and $(f_t)_{t\in I}$ be a smooth family of functions on $M$. This gives rise to the time-dependent metric measure space $(M,d_t, e^{-f_t}\vol_{g_t})_{t\in I}$, where $d_t$ is the Riemannian distance induced by $g_t$. We call $(M,g_t,f_t)_{t\in I}$ a \emph{smooth flow}.

Let $\nabla_t$ denote the gradient associated $g_t$ and write for short $\Delta_t\coloneqq \Delta_{g_t}-\nabla_{t}f_t\cdot \nabla_t$ for the time-dependent weighted Laplacian, where $\Delta_{g_t}$ denotes  the Laplace-Beltrami operator on $(M,g_t)$. If the heat flow exists (for instance if $M$ is compact or if the metrics and weights are uniformly comparable in the sense of Assumption \ref{asm:dualheatflow}) we denote by $(P_{t,s})_{s\leq t}$ the associated family of propagators, i.e. $P_{t,s}\phi$ denotes the solution $u(t)$ to $(\partial_r-\Delta_r)u=0$ with initial datum $u(s)=\phi$.

We call $(M,g_t,f_t)_{t\in I}$ a weighted \emph{super(sub) - Ricci flow} if for all $t\in I$ the inequality 
\begin{align}\label{ineq:defSRF}
\Ric_{f_t}+\frac12\partial_tg_t\ge 0\; (\le 0)\;.
\end{align}
holds respectively. Moreover, we call it a weighted \emph{$N$-super-Ricci flow} for $N<\infty$ if the stronger inequality \[\Ric_{N,f_t}+\frac12\partial_tg_t\ge 0\]
holds.
Here all inequalities are understood in the sense of quadratic forms.
If both inequalities in \eqref{ineq:defSRF} hold, that is $ \Ric_{f_t}=-\frac12\partial_tg_t$, we say $(M,g_t,f_t)_{t\in I}$ is a \emph{weighted Ricci flow}.

For any smooth function $\psi$ on $I\times M$ we have that 
\begin{equation}
    \partial_t \left(|\nabla_t\psi|^2_{g_t}\right)=2g_t(\nabla_t\psi,\nabla_t\partial_t\psi)-\partial_tg_t(\nabla_t\psi,\nabla_t\psi).
\end{equation}
Combining with \eqref{eq:inftyBochner}, we obtain
\begin{align}\label{eq:Bochnerflow}
    \left(\partial_t-\Delta_t\right)\frac{|\nabla_t\psi|^2_{g_t}}{2}&=\nabla_t\psi\cdot\nabla_t(\partial_t-\Delta_t)\psi-\frac12\partial_tg_t(\nabla_t\psi)-\Ric_{f_t}(\nabla_t\psi)-\|\nabla^2_t\psi\|^2_{\mathrm{HS}}
\end{align}
which can be regarded as a Bochner-type identity for the heat operator.
In particular, if $\psi$ is a solution to the heat equation $(\partial_t-\Delta_t)P_{t,s}\phi=0$, this yields
\begin{equation}\label{eq:Bochner-heat}
    -(\partial_t-\Delta_t)|\nabla_t\psi|^2_{g_t}=\partial_tg_t(\nabla_t\psi)+2\Ric_{f_t}(\nabla_t\psi)+2\|\nabla^2_t\psi\|^2_{\mathrm{HS}}\;.
\end{equation}

As a consequence one obtains a gradient estimate for solutions to the heat equation on $(K,\infty)$-super-Ricci flows. Assume that $\Ric_{f_t}\ge-\frac12\partial_tg_t+Kg_t$ for $K\in \R$ and let $psi$ be a solution to the heat equation $(\partial_t-\Delta_t)P_{t,s}\phi=0$. Then we have
\begin{align}\label{ineq:gradientestimate}
    |\nabla_{t}P_{t,s}\phi|^2_{g_t}\leq e^{-2K(t-s)}P_{t,s}(|\nabla_{s}\phi|^2_{g_s}).
\end{align}
Indeed, setting
\begin{equation}
    \Phi(\tau)\coloneqq e^{2K\tau}P_{t,\tau}(|\nabla_{\tau}P_{\tau,s}\phi|^2_{g_\tau}),\quad \tau\in(s,t).
\end{equation}
we have, with \eqref{eq:Bochner-heat} and the equation $\partial_\tau P_{t,\tau}\phi=-P_{t,\tau}\Delta_{\tau}\phi$, that
\begin{align}
    \Phi'(\tau)&=2K\Phi(\tau)+e^{2K\tau}P_{t,\tau}(-\Delta_{\tau}|\nabla_{\tau}P_{\tau,s}\phi|^2_{g_{\tau}}+\partial_\tau|\nabla_\tau P_{\tau,s}\phi|_{g_\tau}^2)\\
    &=2K\Phi(\tau)-e^{2K\tau}P_{t,\tau}(\partial_tg_t(\nabla_tP_{t,s}\phi)+2\Ric_{f_t}(\nabla_tP_{t,s}\phi)+2\|\nabla^2_tP_{t,s}\phi\|^2_{\mathrm{HS}})\\
    &\leq0\;.
\end{align}
Hence $\Phi(\tau)$ is non-increasing and we conclude.

\section{Synthetic Ricci flow}\label{sec:syntRF}
In this section we introduce several synthetic notions of super- and sub-Ricci flow for time-dependent metric measure spaces.

\subsection{Weak Ricci-flow}

\begin{definition}[Weak Ricci flow]\label{def:WeakRF}
A time dependent metric measure space $(X,d_t,\m_t)_{t\in I}$  is called
\begin{enumerate}[(i)]
\item \label{def:weakinftySuper}\emph{weak super-Ricci flow} if for almost every $t\in I$ and every $W_t$-geodesic $(\mu^a)_{a\in [0,1]}$ in $\p_t(X)$ with finite entropy at endpoints, the function $a\mapsto \ent_t(\mu^{a})$ is absolutely continuous on $[0,1]$ and we have
\begin{align}\label{ineq:weaksuperRF}
    \partial^+_a \ent_t(\mu^a)\lvert_{a=1-}-\partial^-_a \ent_t(\mu^a)\lvert_{a=0+}\ge -\frac12\partial^-_tW_{t^-}^2(\mu^0,\mu^1)\;;
\end{align}

\item \label{def:weakSub}\emph{weak sub-Ricci flow} if for almost every $t\in I$, and every $\varepsilon>0$, there exists an open cover $\{U_i\}$ such that for every open subsets $V_0,V_1\subset U_i$, there exists a $W_t$--Wasserstein geodesic $(\mu^a)_{a\in[0,1]}$ so that $\spt\mu^0\subset V_0$, $\spt\mu_1\subset V_1$, and
\begin{align}\label{eq:EntSub}
    \partial_a^+\ent_t(\mu^a)\lvert_{a=1-}-\partial_a^-\ent_t(\mu^a)\lvert_{a=0+}\le -\frac12\partial^-_t W_{t^-}^2(\mu^0,\mu^1)+\varepsilon W_t^2(\mu^0,\mu^1)\;.
\end{align}
\end{enumerate}
If both \eqref{def:weakinftySuper} and \eqref{def:weakSub} are satisfied, we call $(X,d_t,\m_t)_{t\in I}$ a \emph{weak Ricci flow}.
\end{definition}

We consider also the following dimensional refinements of the notion above.

\begin{definition}[Weak non-collapsed Ricci flow]\label{def:weakNSRF}
A time dependent metric measure space $(X,d_t,\m_t)_{t\in I}$  is called
\begin{enumerate}[(i)]
\item \emph{weak $N$-super-Ricci flow} for $N\in[1,\infty)$ if for almost every $t\in I$, every $W_t$-geodesic $(\mu^a)_{a\in [0,1]}$ in $\p_t(X)$ with finite entropy at endpoints, the function $a\mapsto \ent_t(\mu^{a})$ is absolutely continuous on $[0,1]$ and
\begin{align}\label{ineq:weakNSRF}
    \partial^+_a \ent_t(\mu^a)\lvert_{a=1-}-\partial^-_a \ent_t(\mu^a)\lvert_{a=0+}\ge -\frac12\partial^-_tW_t^2(\mu^0,\mu^1)+\frac1{N}|\ent_t(\mu^1)-\ent_t(\mu^0)|^2
\end{align}
\item \emph{weak non-collapsed Ricci flow} if it is a weak sub-Ricci flow and a weak $N$-super-Ricci flow for some $N$.
\end{enumerate}
\end{definition}

We note that the almost concavity inequality \eqref{eq:EntSub} \emph{at the endpoints} on a weak super-Ricci flow backround self improves to an almost concavity at the intermediate points $\rho,\sigma\in(0,1)$. Note however, that the form of the error term in this estimate does not allow to obtain an estimate on the second derivative of the entropy along the geodesic. For a smooth flow, we will obtain a more precise error estimate, where this is possible, see Remark \ref{rem:smooth-error}. 

\begin{proposition}\label{prop:endtointer}
Let $(X,d_t,\m_t)_{t\in I}$ be a weak super-Ricci flow with upper regular entropy.
For any $W_t$-geodesic $(\mu^a)_{a\in [0,1]}$ satisfying \eqref{eq:EntSub} at any intermediate points $0\leq \sigma<\rho\leq 1$ we have that 
    \begin{equation}
         \partial_a^+\ent_t(\mu^a)\lvert_{a=\rho-}-\partial_a^-\ent_t(\mu^a)\lvert_{a=\sigma+}\le -\frac{1}{2(\rho-\sigma)}\partial^-_t W_{t^-}^2(\mu^\sigma,\mu^\rho)+\varepsilon W_t^2(\mu^0,\mu^1).
    \end{equation}
\end{proposition}
\begin{proof}
    To simplify the presentation, we introduce the following notation
    \begin{equation}
        s(\mu;\sigma,\rho)\coloneqq   \partial_a^+\ent_t(\mu^a)\lvert_{a=\rho-}-\partial_a^-\ent_t(\mu^a)\lvert_{a=\sigma+}+\frac{1}{2(\rho-\sigma)}\partial^-_t W_{t^-}^2(\mu^\sigma,\mu^\rho).
    \end{equation}
For any partition $0=a_0< a_1< \cdots < a_n=1$ of $[0,1]$ and $s<t$, by triangle inequality we have 
\begin{equation}
	W_t(\mu^0,\mu^1)-W_s(\mu^0,\mu^1)\geq \sum_{i=1}^nW_t(\mu^{a_{i-1}},\mu^{a_i})-W_s(\mu^{a_{i-1}},\mu^{a_i})\;.
\end{equation}
which implies
\begin{equation}\label{ineq:WassersteinDeri}
\partial^-_t W_{t^-}(\mu^0,\mu^1)\geq \sum^n_{i=1} \partial^-_t W_{t^-}(\mu^{a_{i-1}},\mu^{a_i}).
\end{equation}
By upper regularity of entropy, one gets
  \begin{align}
     \partial_a^+\ent_t(\mu^a)\lvert_{a=1-}-\partial_a^-\ent_t(\mu^a)\lvert_{a=0+}
     \geq \sum_{i=1}^n\partial_a^+\ent_t(\mu^a)\lvert_{a=a_i-}-\partial_a^-\ent_t(\mu^a)\lvert_{a=a_{i-1}+}
  \end{align}
Combining both estimates yields
  \begin{align}
	s(\mu;0,1)\geq&\sum_{i=1}^n s(\mu;a_{i-1},a_i).\label{ineq:etaalongcurve}
\end{align}
   Note that each summand in \eqref{ineq:etaalongcurve} is non-negative due to the super-Ricci flow assumption and $s(\mu;0,1)\leq \varepsilon W^2_t(\mu^0,\mu^1)$ by assumption. Thus we conclude $s(\mu;\sigma,\rho)\leq \varepsilon W^2_t(\mu^0,\mu^1)$ for arbitrary $\sigma<\rho$.
\end{proof}

Let us also introduce the following infinitesimal quantities.
\begin{definition}\label{def:eta} For $t\in I$, $x,y\in X$ and $\varepsilon>0$, define
    \begin{align}\label{ric}
 & \eta^\pm_{\varepsilon}(t,x,y)\coloneqq  \inf\left\{
\frac1{W_t^2(\mu^0,\mu^1)}\left[ \partial^{\pm}_a \ent_t(\mu^a)\big|_{a=1-}-\partial_a^{\mp} \ent_t(\mu^a)\big|_{a=0+}
+\frac12\partial_{t}^{-}W_{t^-}^2(\mu^0,\mu^1)\right]\right\}\;,
\end{align}
where the infimum is taken over all non-constant $W_t$-geodesics $(\mu^a)_{a\in[0,1]}$ with $\ent_t(\mu^0)$, $\ent_t(\mu^1)<\infty$ and such that $\spt(\mu^0)\subset B_t(x,\varepsilon)$, $\spt(\mu^1)\subset B_t(y,\varepsilon)$. We further set 
\begin{equation}
\eta^\pm(t,x,y)\coloneqq\lim_{\varepsilon\to 0}\eta^\pm_{\varepsilon}(t,x,y)=\sup_{\varepsilon> 0}\eta^\pm_{\varepsilon}(t,x,y)\;, \quad
  \eta^*(t,x)\coloneqq\limsup_{y,z\to x}\eta^+(t,y,z)\;.
\end{equation}  
\end{definition}

Note that obviously $\eta^+_\varepsilon\geq \eta^-_\varepsilon$. In the definition of $\eta^\pm$ the choice of upper and lower derivative of the entropy along geodesics and the Wasserstein distance along the flow are made such that $\eta^-$ is super-additive under partitioning of transports, see e.g. Lemma \ref{lemma:D_aEnt} below. The next lemma shows that under additional regularity assumptions upper/lower derivatives of the  entropy can be switched.

\begin{lemma}\label{lemma:freeofchoice_eta}
	Let $(X,d_t,\m_t)_{t\in I}$ be a time-dependent metric measure spaces admitting a lower log-Lipschitz control and assume that for each $t$ the entropy $\ent_t$ is upper regular. Then we have $\eta^+_\varepsilon=\eta^-_\varepsilon$ on $I\times X^2$ for each  $\varepsilon>0$.
\end{lemma}

\begin{proof}
	It remains to prove $\eta^+_\varepsilon\leq \eta^-_\varepsilon$.
	For any $\delta>0$, take a $W_t$-geodesic $(\mu^a)_{a\in [0,1]}$ almost achieving $\eta_{\varepsilon}(t,x,y)$ up to error $\delta$.
	Recall from the upper regularity of entropy that we have 
	\begin{equation}
		\partial_a^-\ent_t(\mu^a)|_{a=1-}\geq \liminf_{\rho\nearrow1}\partial^-_a\ent_t(\mu^a)|_{a=\rho+}\geq \liminf_{\rho\nearrow1}\partial^+_a\ent_t(\mu^a)|_{a=\rho-}.
	\end{equation}
	Combining this with the analogous inequality at $a=0$, one can find $\rho,\sigma\in (0,1)$ arbitrarily close to $1$ and $0$ respectively such that $\spt(\mu^\sigma)\subset B_t(x,\varepsilon)$, $\spt(\mu^\rho)\subset B_t(y,\varepsilon)$ and 
	\begin{equation}
		\partial_a^-\ent_t(\mu^a)|_{a=1-}-\partial_a^+\ent_t(\mu^a)|_{a=0+}\geq \partial^+_a\ent_t(\mu^a)|_{a=\rho-}-\partial^-_a\ent_t(\mu^a)|_{a=\sigma+}-\delta.\label{ineq:15/07-2}
	\end{equation}
	Similar to \eqref{ineq:WassersteinDeri}, now together with the lower log-Lipschitz condition, we have
	\begin{align}
		\frac12\partial^-_tW^2_{t^-}(\mu^0,\mu^1)\geq &W_t(\mu^0,\mu^1)\cdot \left(\partial^-_tW_{t^-}(\mu^\sigma,\mu^\rho)+\partial^-_tW_{t^-}(\mu^\sigma,\mu^0)+\partial^-_tW_{t^-}(\mu^\rho,\mu^1)\right)\\
		\geq & \frac{1}{2(\rho-\sigma)}\partial^-_tW^2_{t^-}(\mu^\sigma,\mu^\rho)-L(1-\rho+\sigma)W^2_t(\mu^0,\mu^1).\label{ineq:15/07-1}
	\end{align}
	Hence using the restricted geodesic $(\mu^a)_{a\in[\sigma,\rho]}$ as a candidate in the infimum defining $\eta^+_\varepsilon$, we obtain that $\eta^+_\varepsilon(t,x,y)$ is bounded above by $\eta^-_{\varepsilon}(t,x,y)$ up to an error that vanishes as $\delta,\sigma\to0$ and $\rho\to1$.
\end{proof}

Let us note that without further regularity assumptions on the time-dependence of the metrics, sudden change of the flow cannot be detected by the above weak formulation as the following example shows.

\begin{example}[Gluing two flows]
Let $(X,\bar{d}_t,\bar{\m}_t)_{t\in(0,t_1]}$ and $(X,\hat{d}_t,\hat{\m}_t)_{t\in(0,t_2]}$ be two families of time-dependent metric measure spaces.
Define a time-dependent mms $(X,d_t,\m_t)_{t\in (0,t_1+t_2]}$ by taking $d_t$ and $\m_t$ to be $\bar{d}_t$ and $\bar{\m}_t$ when $t\in (0,t_1]$ and $\hat{d}_{t-t_1}$ and $\hat{\m}_{t-t_1}$ when $t\in (t_1,t_1+t_2]$ respectively.
Then, if both $(X,\bar{d}_t,\bar{\m}_t)_{t\in(0,t_1]}$ and $(X,\hat{d}_t,\hat{\m}_t)_{t\in(0,t_2]}$ are weak $N$-super-/sub-Ricci flows, so is $(X,d_t,\m_t)_{t\in (0,t_1+t_2]}$. 
\end{example}

\subsection{Rough Ricci-flow}

\begin{definition}[Rough Ricci flow]\label{def:rough2RF} A time-dependent metric measure space $(X,d_t,\m_t)_{t\in I}$ satisfying \cref{asm:dualheatflow} is called
\begin{enumerate}[(i)]
    \item\label{def:roughSuper}(\emph{rough super-Ricci flow}) if for all $s\leq t$ and every $x,y\in X$ we have
    \begin{align}\label{ineq:defroughsuperRicci}
        W^2_s(\ap{\delta_x},\ap{\delta_y})\le d^2_t(x,y)\;;
    \end{align}
  \item\label{def:roughSub}(\emph{rough sub-Ricci flow}) if for every $\varepsilon>0$ and almost every $t\in I$, there exists an open cover $\{U_i\}_{i\in\N}$ for which the following holds. For every $i$ and $x,y\in U_i $, there exists $s_0<t$ so that for every $s\in (s_0,t)$
    \begin{align}\label{ineq:defroughsubRicci}
        W^2_s(\ap{\delta_x},\ap{\delta_y})\ge d^2_t(x,y)-\varepsilon d^2_t(x,y)(t-s)\;.
    \end{align}
\end{enumerate}
If both \eqref{def:roughSuper} and \eqref{def:roughSub} are satisfied we call $(X,d_t,\m_t)_{t\in I}$ a \emph{rough Ricci flow}.
\end{definition}

\begin{rem}\label{rmk:roughsuperRF}
     The definition of rough super-Ricci-flow is equivalent to the following a priori stronger one: for all $s\leq t$ and every $\mu,\nu\in \p_t(X)$ we have
    \begin{align}
        W^2_s(\ap{\mu},\ap{\nu})\le W^2_t(\mu,\nu)\;.
    \end{align}
    Indeed, take $\sigma$ to be an optimal plan from $\mu$ to $\nu$ for the cost $d^2_t$.
    Recall that the dual propagator satisfies $\ap\mu=\int_X \ap\delta_x\d\mu(x)$. 
    Then by convexity of the Wasserstein distance and \eqref{ineq:defroughsuperRicci} we obtain for any $s\leq t$
    \begin{align}
W^2_s(\ap\mu,\ap\nu)\leq \int W^2_s(\ap\delta_x,\ap\delta_y)\d\sigma(x,y)
{\leq}\int d^2_t(x,y)\d\sigma(x,y)=W^2_t(\mu,\nu)\;.
    \end{align}
\end{rem}
Again we consider a dimensional refinement of the notion above.
\begin{definition}[Rough non-collapsed Ricci flow] \label{def:NroughSuper}  A time dependent metric measure space $(X,d_t,\m_t)$ satisfying \cref{asm:dualheatflow} is called 
\begin{enumerate}[(i)]
    \item \emph{rough $N$-super-Ricci flow} for $N\in[1,\infty)$ if
for all $s\leq t$ and every $\mu,\nu\in \p_t(X)$ we have
    \begin{align}
        W_s^2(\ap\mu,\ap\nu)\le W_t^2(\mu,\nu)-\frac2N\int_{[s,t]}[\ent_r(\hat P_{t,r}\mu)-\ent_r(\hat P_{t,r}\nu)]^2\d r\;;
    \end{align}
    \item \emph{rough non-collapsed Ricci flow} if it is both a rough sub-Ricci flow and a rough $N$-super-Ricci flow for some $N<\infty$.
    \end{enumerate}
\end{definition}

We also introduce the following quantities detecting the initial exponential expansion rates of the Wasserstein distances between heat flows starting from Dirac masses. In the static case, this notion has been introduced in \cite{sturm2017remarks}.

\begin{definition}\label{def:theta+}
Define functions $\vartheta^\pm:I\times X^2\to \R\cup\{\pm\infty\}$ by
\begin{align}
    \vartheta^+(t,x,y)\coloneqq -\liminf_{s\nearrow t}\frac1{t-s}\log\frac{W_s(\ap\delta_x,\ap\delta_y)}{d_t(x,y)}
\end{align}
and 
\begin{align}
    \vartheta^-(t,x,y)\coloneqq -\limsup_{s\nearrow t}\frac1{t-s}\log\frac{W_s(\ap\delta_x,\ap\delta_y)}{d_t(x,y)}.
\end{align}
Moreover, we define $\vartheta^*\colon I\times X\to \R\cup\{\pm\infty\}$ by
\begin{align}
    \vartheta^*(t,x)=\limsup_{y,z\to x}\vartheta^+(t,y,z)
\end{align}
\end{definition}

\subsection{Characterizing synthetic notions by infinitesimal quantities} \label{sec:infvsloc}
We will now give equivalent characterisations of the weak/rough super- and sub-Ricci flows in terms of non-negativity or non-positivity of the quantities $\eta$ and $\theta$ respectively.

\begin{theorem}\label{thm:RFvsTheta}
	A time-dependent mms~$(X,d_t,\m_t)_{t\in I}$  satisfying \cref{asm:dualheatflow} is a rough super-Ricci flow if and only if 
    \[\vartheta^-(t,x,y)\geq 0\quad  \text{ for all  $t\in I$ and $x,y\in X$}\;.\]
    Moreover, $(X,d_t,\m_t)_{t\in I}$ is a rough sub-Ricci flow if and only if for a.e. $t\in I$ we have
    \[\vartheta^{*}(t,x)\leq 0\quad  \text{ for all $x\in X$}\;.\]
\end{theorem}
\begin{proof}
For both statements, it suffices to show the ``if" part as the other direction is evident.
	Assume that $\vartheta^-\geq 0$. Fix $x,y\in X$ and $t\in I$. For any $r<t$, as in \cref{rmk:roughsuperRF}, we have for all $\sigma\in\mathrm{Opt}_{d^2_r}(\hat{P}_{t,r}\delta_x,\hat{P}_{t,r}\delta_y)$ and $\tau<r$ that
\begin{align}
W^2_{r}(\hat{P}_{t,r}\delta_x,\hat{P}_{t,r}\delta_y)-W^2_{\tau}(\hat{P}_{t,\tau}\delta_x,\hat{P}_{t,\tau}\delta_y)
	\geq \int \big[d^2_r(z_1,z_2)-W^2_{\tau}(\hat{P}_{r,\tau}\delta_{z_1},\hat{P}_{r,\tau}\delta_{z_2})\big]\d\sigma(z_1,z_2).
	\end{align}
Then by Fatou's lemma,
\begin{align}
	&\partial_\tau^-\big|_{\tau=r-}W^2_{\tau}(\hat{P}_{t,\tau}\delta_x,\hat{P}_{t,\tau}\delta_y) = \liminf_{\tau\nearrow r}\frac{W^2_{\tau}(\hat{P}_{t,\tau}\delta_x,\hat{P}_{t,\tau}\delta_y)-W^2_{r}(\hat{P}_{t,r}\delta_x,\hat{P}_{t,r}\delta_y)}{r-\tau}	 \geq \\
&  \int \liminf_{\tau\nearrow r}\frac{d^2_r(z_1,z_2)-W^2_{\tau}(\hat{P}_{r,\tau}\delta_{z_1},\hat{P}_{r,\tau}\delta_{z_2})}{r-\tau}\d\sigma(z_1,z_2)=\int 2d^2_r(z_1,z_2)\vartheta^-(r,z_1,z_2)\d\sigma(z_1,z_2)\geq 0.
\end{align}
By \cref{lemma:dualheatflow}, $r\mapsto W^2_{r}(\hat{P}_{t,r}\delta_x,\hat{P}_{t,r}\delta_y)$ is absolutely continuous. Integrating the above inequality from $s$ to $t$ yields \eqref{ineq:defroughsuperRicci}.

Now assume that for a.e. $t\in I$, $\vartheta^{*}(t,\cdot)\leq 0$. 
Then by definition of $\vartheta^*$ for all $\varepsilon>0$, there exists $\delta>0$ s.t. for all $y,z\in B_t(x,\delta)$ we have 
\begin{equation}
\vartheta^+(t,y,z)=\limsup_{s\nearrow t} \frac{-1}{t-s}\log\frac{W_s(\hat P_{t,s}\delta_y,\hat P_{t,s}\delta_z)}{d_t(y,z)}\leq \varepsilon\;,
\end{equation}
which means one can find $s_0<t$ so that inequality \eqref{ineq:defroughsubRicci} holds for all $s\in(s_0,t)$ and $y,z\in B_t(x,\delta)$. 
Finally, to obtain the open cover required in \cref{def:rough2RF}, one exhausts $X$ by compact sets and notes that any of these compact sets can be covered by finitely many balls as above such that \eqref{ineq:defroughsubRicci} holds for any two points in the same ball.
\end{proof}

We now move to a characterisation of weak super-/sub-Ricci flows in terms of the quantity $\eta$. We need the following preparatory lemma.

\begin{lemma}\label{lemma:D_aEnt}
    Let $(X,d,\m)$ be an e.n.b. m.m.s.
    Let $(\mu^a)_{a\in [0,1]}$ be a 2-Wasserstein geodesic in $D(\ent)$ with optimal dynamical plan $\pi$.
    Let $\{\Gamma_\alpha\}_\alpha$ be a finite partition of $\spt(\pi)$ with each having positive $\pi$-measure.
    Denote by $\mu^a_{\alpha}$ the normalized measure $(e_a)_{\#}\frac{\pi\llcorner \Gamma_\alpha}{\pi(\Gamma_\alpha)}$ for each $a,\alpha$.
    Then
    \begin{align}
        \partial_a^-\ent(\mu^{a\pm})&\geq \sum_\alpha \pi(\Gamma_\alpha)\partial_a^-\ent(\mu^{a\pm}_\alpha),\quad \forall a\in[0,1)\label{ineq:D_a^-Ent}\\
        \partial_a^+\ent(\mu^{a\pm})&\leq \sum_\alpha \pi(\Gamma_\alpha)\partial_a^+\ent(\mu^{a\pm}_\alpha),\quad \forall a\in(0,1]\label{ineq:D_a^+Ent}.
    \end{align}
\end{lemma}
\begin{proof}
    By the essential non-branching property, $\pi$ is concentrated on a Borel non-branching set $\Gamma$, and the evaluation map $e_a$ is injective for each $a\in (0,1)$.
    Therefore, by setting  $\rho^a_\alpha\coloneqq\frac{\d \mu^a_\alpha}{\d \m}$, we have 
   \begin{equation}
    \begin{array}{ll}
\rho^a_\alpha\leq\frac{\rho^a}{\pi(\Gamma_\alpha)},\quad & a=0,1\\
\rho^{a}_\alpha=
    \frac{\rho^a}{\pi(\Gamma_\alpha)}\llcorner e_a(\Gamma_\alpha),\quad & a\in(0,1).\label{eq:restrictdensity}
\end{array}
\end{equation}
Hence for all $h\in(0,1)$,
\begin{align}
     \ent(\mu^1)-\ent(\mu^{1-h})
    &\geq \sum_\alpha\pi(\Gamma_\alpha)(\ent(\mu^{1}_\alpha)-\ent(\mu^{1-h}_\alpha))\label{ineq:05/03-1}\\
     \ent(\mu^h)-\ent_t(\mu^{0})
    &\leq \sum_\alpha\pi(\Gamma_\alpha)(\ent(\mu^{h}_\alpha)-\ent(\mu^{0}_\alpha))\label{ineq:05/03-2}.
\end{align}
Thus the inequalities \eqref{ineq:D_a^-Ent} and \eqref{ineq:D_a^+Ent} for $a=0,1$ follow by letting $h$ to $0$. The case for $a\in(0,1)$ is obtained similarly.
\end{proof}

\begin{assumption}\label{asm:localtoglobal}
    The time-dependent metric measure space $(X,d_t,\m_t)_{t\in I}$ admits a log-Lipschitz control on the metrics. For each $t$ the space $(X,d_t,\m_t)$ is e.n.b. and locally compact and the entropy $\ent_t$ is upper regular.
\end{assumption}

\begin{proposition}\label{prop:etatoSRF}
    A time-dependent mms $\left(X,d_t,\m_t \right)_{t\in I}$  satisfying \cref{asm:localtoglobal} is a weak super-Ricci flow provided that for a.e. $t\in I$
    \[\eta^-(t,x,x)\geq 0\quad \text{ for all $x\in X$}\;.\]
    
Moreover, the statement holds without the assumption of a log-Lipschitz control if instead $X$ is compact.
\end{proposition}
\begin{proof}
i) We first consider the case that $X$ is compact.
Fix $\mu^0,\mu^1\in\mathcal{P}_t(X)\cap D(\ent_t)$ and let $\pi$ be an optimal dynamical plan which is concentrated on a non-branching Borel set $\Gamma\subset \geo(X,d_t)$. Fix $\delta>0$ and  
for $n\in\N$ set $a_i\coloneqq i/n$, $i\in\{0,1,\cdots,n\}$.
We make the following
\smallskip

\textbf{Claim}: There is $n\in\N$ and a finite partition $\{\Gamma_\alpha\}_\alpha$ of $\Gamma$ into sets of positive $\pi$-measure s.t. for $\mu^a_{\alpha}\coloneqq(e_a)_{\#}\frac{\pi\llcorner \Gamma_\alpha}{\pi(\Gamma_\alpha)}$ we have for all $i$ and $\alpha$
 \begin{align}
       (a_{i+1}-a_i)\left(\partial^-_a \ent_t(\mu^{a_{i+1}-}_\alpha)-\partial_a^+ \ent_t(\mu^{a_{i}+}_\alpha)\right)
+\frac12\partial_{t}^-W_{t^-}^2(\mu^{a_{i}}_\alpha,\mu^{a_{i+1}}_\alpha)\geq -\delta\cdot W_t^2(\mu^{a_{i}}_\alpha,\mu^{a_{i+1}}_\alpha).
 \end{align}
 
Assuming that the claim is true, we conclude as follows. By \cref{lemma:D_aEnt}, for all $i$ and $\alpha$ we have
\begin{align}
      \partial_a^-\ent_t(\mu^{a_{i+1}-})-\partial_a^+\ent_t(\mu^{a_{i}+})&\geq \sum_\alpha \pi(\Gamma_\alpha)\left(\partial_a^-\ent_t(\mu^{a_{i+1}-}_\alpha)-\partial_a^+\ent_t(\mu^{a_{i}+}_\alpha)\right).
\end{align}
For each $\alpha$, $a\mapsto \mu^a_{\alpha}$ is a $W_t$-geodesic. Hence analogous to \eqref{ineq:WassersteinDeri}, one has
\begin{align}
     \frac12\partial_{t}^-W_{t^-}^2(\mu^{0},\mu^{1})\geq \sum_\alpha\pi(\Gamma_\alpha) \frac12\partial_{t}^-W_{t^-}^2(\mu^{0}_\alpha,\mu^{1}_\alpha)\geq \sum_{i,\alpha}\pi(\Gamma_\alpha)\frac{n}{2}\partial_{t}^-W_{t^-}^2(\mu^{a_{i}}_\alpha,\mu^{a_{i+1}}_\alpha).\label{ineq:16/05-1}
\end{align}
Then, with the upper regularity of entropy, it follows
\begin{align}
    \partial_a^-\ent_t(\mu^{1-})-\partial_a^+\ent_t(\mu^{0+})&\geq \sum_i  \partial_a^-\ent_t(\mu^{a_{i+1}-})-\partial_a^+\ent_t(\mu^{a_{i}+})\\
    &\geq\sum_i\sum_\alpha \pi(\Gamma_\alpha)\left(\partial_a^-\ent_t(\mu^{a_{i+1}-}_\alpha)-\partial_a^+\ent_t(\mu^{a_{i}+}_\alpha)\right)\\
    &\geq \sum_{i,\alpha}-n\pi(\Gamma_\alpha)\left(\frac12\partial_{t}^-W_{t^-}^2(\mu^{a_{i}}_\alpha,\mu^{a_{i+1}}_\alpha)+\delta W_t^2(\mu^{a_{i}}_\alpha,\mu^{a_{i+1}}_\alpha)\right)\\
    &\geq  -\frac12\partial_{t}^-W_{t^-}^2(\mu^{0},\mu^{1})-\delta\sum_{i,\alpha}n^{-1}\pi(\Gamma_\alpha)W_t^2(\mu^{a_{i}}_\alpha,\mu^{a_{i+1}}_\alpha)\\
    &= -\frac12\partial_{t}^-W_{t^-}^2(\mu^{0},\mu^{1})-\delta W_t^2(\mu^0,\mu^1).
\end{align}
Since $\delta$ and the geodesic $(\mu^a)_{a\in[0,1]}$ were arbitrary, we conclude the strong dynamic convexity of the entropy.
\medskip

ii) We now prove the claim. By definition, $\eta^-(t,x,y)$ is lower semi-continuous jointly in $x$ and $y$.
Therefore by assumption, for any $x\in X$, there is $r_x>0$ s.t. $\eta^-(t,x_1,x_2)> -\delta$ for all $x_1,x_2\in B_t(x,r_x)$. 
By compactness, we can find finitely many such balls covering $X$.
In particular, this implies an $\varepsilon>0$, a finite partition of $X$ by $\{L_j\}_{j\in J}$ and a family of compact subsets $\{X_j\}_{j\in J}$ s.t. $B_t(L_j,\varepsilon)\subset X_j$ and $\eta^-(t,x,y)>-\delta$ for all pair of $x,y$ in the same element of $\{X_j\}_j$.

Next, for each pair $(x,y)$ in $X_j\times X_j$ for some $j$, by definition, there exists some $d>0$ s.t. $\eta^-_d(t,x,y)>-\delta$.
In particular, $\eta^-_{d/2}(t,\tilde{x},\tilde{y})>-\delta$ for all $(\tilde{x},\tilde{y})\in B_t(x,\frac{d}{2})\times B_t(y,\frac{d}{2})$.
Again based on the compactness, one has finitely many balls $\{B_t(x_k,r_k)\}_{k\in \Lambda}$ s.t. 
\begin{align}
    \cup_{j\in J} X_j\times X_j\subset \cup_{k_1,k_2\in \Lambda} B_t(x_{k_1},r_{k_1})\times B_t(x_{k_2},r_{k_2}),
\end{align}
and for all $W_t$-geodesic $(\mu^a)_{a\in [0,1]}$ that $\spt(\mu^0)\times \spt(\mu^1)\subset B_t(x_{k_1},r_{k_1})\times B_t(x_{k_2},r_{k_2})$, one has
 \begin{equation}
        \partial^-_a \ent_t(\mu^a)\big|_{a=1}-\partial_a^+ \ent_t(\mu^a)\big|_{a=0}
+\frac12\partial_{t}^-W_{t^-}^2(\mu^0,\mu^1)\geq -\delta\cdot W_t^2(\mu^0,\mu^1).
    \end{equation}
By intersecting above products of balls with $\{L_{j_1}\times L_{j_2}\}_{j_1,j_2\in J}$, we can find a finite collection of mutually disjoint subsets $\{P_{\beta}\}_{\beta\in B}$ in $X\times X$ s.t. each $P_{\beta}$ is contained in some $B_t(x_{k_1},r_{k_1})\times B_t(x_{k_2},r_{k_2})$ and 
\begin{align}
   \cup_{j\in J}X_j\times X_j\subset \cup_{\beta\in B}P_{\beta}.
\end{align}

Now we choose $n$ to be an integer with $\mathrm{diam}(X)/n<\varepsilon$.
This means, for any geodesic $\gamma$ and $i$, $(\gamma^{a_i},\gamma^{a_{i+1}})$ is in $X_j\times X_j$ for the $j\in J$ that has $\gamma^{a_i}\in L_j$ and thus the pair is contained in some $P_{\beta}$.

Finally we can express the decomposition of $\Gamma$.
Let $\alpha\coloneqq (\alpha_1,\dots,\alpha_{n})$ be a multi-index so that $\alpha_i\in[0,|B|]\cap \N$.
Denote by $\Gamma_\alpha$ for each $\alpha$ as follows
\begin{align}
    \Gamma_\alpha\coloneqq\{\gamma\in\Gamma:(\gamma^{a_i},\gamma^{a_{i+1}})\in P_{\alpha_i},\forall i=0,\dots, n\}.
\end{align}
This construction yields the desired partition $\{\Gamma_\alpha\}_{\alpha}$ and conlcudes the claim.
\medskip

iii) Finally, we consider the case that $X$ is not compact.
By Hopf-Rinow theorem, $(X,d_t)$ is proper.
Hence whenever $\mu^0$ and $\mu^1$ have bounded supports, we can reduce the problem to the compact situation.
In general, we can decompose $\spt(\pi)$ into a countable partition $\{\Gamma_n\}_{n\in\N}$ s.t. for each $n\in\N$, $\pi(\Gamma_n)>0$ and $\Gamma_n$ is contained in a bounded set.

As usual, we denote $\mu^a_n\coloneqq (e_a)_{\#}\frac{\pi\llcorner \Gamma_n}{\pi(\Gamma_n)}$ for each $a\in[0,1]$ and $n\in\N$. 
For each $n\in \N$, as the curve $(\mu^a_n)_{a\in [0,1]}$ is contained in a compact subset, we have
\begin{align}\label{ineq:etafrombelow2}
       \partial^-_a \ent_t(\mu^a_n)\big|_{a=1-}-\partial_a^+ \ent_t(\mu^a_n)\big|_{a=0+}
+\frac12\partial_{t}^-W_{t^-}^2(\mu^{0}_n,\mu^{1}_n)\geq 0.
\end{align}

By the upper regularity, $\ent_t(\mu^a)<\infty$ for all $a$ hence by Fubini that \eqref{ineq:05/03-1} and \eqref{ineq:05/03-2} are still true.
Furthermore, \cref{prop:etatoCD} together with upper log-Lipschitz control of distances ensures that $\ent_t$ is strongly $-\kappa_t$-convex for a.e. $t\in I$.
This indicates that for every $h,n$ 
\begin{equation}
    \ent_t(\mu^{1}_n)-\ent_t(\mu^{1-h}_n)-(\ent_t(\mu^{h}_n)-\ent_t(\mu^{0}_n))\geq -\kappa_th(1-h)W_t^2(\mu^{1}_n,\mu^{0}_n).
\end{equation}
So we can apply Fatou's lemma, obtaining 
\begin{align}
    &\partial^-_a \ent_t(\mu^a)\big|_{a=1-}-\partial_a^+ \ent_t(\mu^a)\big|_{a=0+}\\
   \geq & \liminf_{h\searrow0}\sum_n \frac{\pi(\Gamma_n)}{h}\big[ \ent_t(\mu^{1}_n)-\ent_t(\mu^{1-h}_n)-\ent_t(\mu^{h}_n)+\ent_t(\mu^{0}_n)\big]\\
    \geq& \sum_n\liminf_{h\searrow0}\frac{\pi(\Gamma_i)}{h}\big[ \ent_t(\mu^{1}_n)-\ent_t(\mu^{1-h}_n)-\ent_t(\mu^{h}_n)+\ent_t(\mu^{0}_n)\big] \\
   \geq & \sum_n\pi(\Gamma_n) \big[\partial^-_a \ent_t(\mu^{a}_n)\big|_{a=1-}-\partial_a^+ \ent_t(\mu^{a}_n)\big|_{a=0+}\big].
\end{align}
Similarly, with the lower log-Lipschitz control of the distance, we have
\begin{equation}
    \partial_{t}^-W_{t^-}^2(\mu^{0},\mu^{1})\geq \sum_n\pi(\Gamma_n)\partial_{t}^-W_{t^-}^2(\mu^{0}_n,\mu^{1}_n).
\end{equation}
The proof is completed by combining the inequalities above with \eqref{ineq:etafrombelow2}.
\end{proof}

Analogously to \cref{thm:RFvsTheta}, the weak super/sub-Ricci flow can be characterized using the quantities $\eta$ as follows.
\begin{theorem}\label{thm:charaSRF}
	Let $\left(X,d_t,\m_t \right)_{t\in I}$ be a time-dependent m.m.s. 
    It is a weak sub-Ricci flow if and only if for a.e. $t\in I$,  \[\eta^*(t,x)\leq 0\quad \text{ for all $x\in X$}\;.\]

    Moreover, if $\left(X,d_t,\m_t \right)_{t\in I}$ satisfies \cref{asm:localtoglobal}, then the following are equivalent:
	\begin{enumerate}
		\item $\left(X,d_t,\m_t \right)_{t\in I}$ is a weak super-Ricci flow;
		\item for a.e. $t\in I$, $\eta^-(t,x,y)\geq 0$ for all $x,y\in X$;
		\item for a.e. $t\in I$, $\eta^-(t,x,x)\geq 0$ for all $x\in X$.
	\end{enumerate}
\end{theorem}

\begin{proof}
Assume that for a.e. $t\in I$, $\eta^{*}(t,\cdot)\leq 0$. 
By definition, if $\eta^*(t,x)\leq 0$, then for all $\varepsilon>0$ there exists $\delta>0$ s.t. for all $y,z\in B_t(x,\delta)$ we have $\eta^+(t,y,z)<\varepsilon$.
Hence for any open sets $V_0,V_1\subset B_t(x,\delta)$, by the definition of $\eta^+$, there exists a non-constant $W_t$-geodesic $(\mu^a)_{a\in[0,1]}$ with $\spt(\mu^i)\subset V_i$ for $i=0,1$ s.t. \eqref{eq:EntSub} is satisfied.
Repeating the exhaustion argument from the proof of \cref{thm:RFvsTheta}, we obtain an open cover as required in \cref{def:WeakRF}. Thus, $(X,d_t,\m_t)_{t\in I}$ is a weak sub-Ricci flow. The reverse direction is clear.

For the equivalence of the weak super-Ricci flow, the implications $(1) \Rightarrow (2)$ and $(2) \Rightarrow (3)$ are straightforward, while the direction $(3) \Rightarrow (1)$ is given by \cref{prop:etatoSRF}.
\end{proof}

We conclude this section by establishing a local-to global property for weak $N$-super-Ricci flows.

\begin{definition}
   A time-dependent mms $\left(X,d_t,\m_t \right)_{t\in I}$ is called a \emph{local weak $N$-super-Ricci flow} for $N\in[1,\infty]$ if for a.e. $t\in I$ every point $x\in X$ has a neighborhood $U$ s.t. along every $W_t$-geodesic $(\mu^\tau)_{\tau\in [0,1]}$ in $\p_t(X)$ with $\mu^0,\mu^1$ supported in $U$ and with finite entropy, the function $[0,1]\ni\tau\mapsto \ent_t(\mu^{\tau})$ is absolutely continuous and \eqref{ineq:weakNSRF} holds.
\end{definition}

We stress that the intermediate points of the geodesic in the definition above are not required to be supported in $U$.

\begin{theorem}\label{prop:localtoglobal}
  Let $\left(X,d_t,\m_t \right)_{t\in I}$ be a time-dependent m.m.s. satisfying \cref{asm:localtoglobal}.
If $\left(X,d_t,\m_t \right)_{t\in I}$ is a local weak $N$-super-Ricci flow, then it is a weak $N$-super-Ricci flow.
\end{theorem}
\begin{proof}
	The case $N=\infty$ is a direct consequence of \cref{prop:etatoSRF}, since on any local super-Ricci flow we immediately conclude that $\eta(t,x,x)\geq 0$ for all $t$ and $x$.

	 Assume that $\left(X,d_t,\m_t \right)_{t\in I}$ is a local weak $N$-super-Ricci flow for $N<\infty$. Then it is in particular a local and thus also a global weak super Ricci flow as we just argued. Similar to the proof of \cref{prop:etatoSRF}, for any $\delta>0$ and any non-constant $W_t$-geodesic $(\mu^a)_{a\in [0,1]}$ with finite entropy at end points, the optimal dynamical plan $\pi$  is concentrated on some non-branching set $\Gamma$ and we can find a countable partition $\{\Gamma_\alpha\}_\alpha$ of $\Gamma$ provided that $\pi(\Gamma_\alpha)>0$ s.t. for all $i,\alpha$ and $\mu^a_{\alpha}\coloneqq(e_a)_{\#}\frac{\pi\llcorner \Gamma_\alpha}{\pi(\Gamma_\alpha)}$ we have that
	 \begin{align}
	 &	(a_{i+1}-a_i)\left(\partial^-_a \ent_t(\mu^{a_{i+1}-}_\alpha)-\partial_a^+ \ent_t(\mu^{a_{i}+}_\alpha)\right)
	 	+\frac12\partial_{t}^-W_{t^-}^2(\mu^{a_{i}}_\alpha,\mu^{a_{i+1}}_\alpha)\\
	 	\geq &-\delta\cdot W_t^2(\mu^{a_{i}}_\alpha,\mu^{a_{i+1}}_\alpha)+\frac1{N}|\ent_t(\mu^{a_{i}}_\alpha)-\ent_t(\mu^{a_{i+1}}_\alpha)|^2,
	 \end{align}
 where $a_i\coloneqq i/n$ for $i\in\{0,1,\cdots,n\}$ and $n$ can be arbitrarily large.
Analogous to the proof of \cref{prop:etatoSRF}, summing the above inequality over all $\alpha$ and $i$ yields: 
 \begin{align}
 	& \partial_a^-\ent_t(\mu^{1-})-\partial_a^+\ent_t(\mu^{0+})+\frac12\partial_{t}^-W_{t^-}^2(\mu^{0},\mu^{1})\\
 	\geq &-\delta W_t^2(\mu^0,\mu^1)+ \frac1{N}\sum_{i,\alpha}n\pi(\Gamma_\alpha)|\ent_t(\mu^{a_{i}}_\alpha)-\ent_t(\mu^{a_{i+1}}_\alpha)|^2.
 \end{align}
For all $1\leq i\leq n-2$, one has the following equation on entropy due to \eqref{eq:restrictdensity}
\begin{equation}
	\ent_t(\mu^{a_{i}}_\alpha)-\ent_t(\mu^{a_{i+1}}_\alpha)=\pi(\Gamma_{\alpha})^{-1}\left(\int_{e_{a_i}(\Gamma_\alpha)}\rho^{a_i}\log\rho^{a_i}-\int_{e_{a_{i+1}}(\Gamma_\alpha)}\rho^{a_{i+1}}\log\rho^{a_{i+1}}\right)\;.
\end{equation}
Therefore, applying twice the Jensen inequality gives
\begin{align}
	&\sum^{n-2}_{i=1}\sum_{\alpha}n\pi(\Gamma_\alpha)|\ent_t(\mu^{a_{i}}_\alpha)-\ent_t(\mu^{a_{i+1}}_\alpha)|^2\\
	=&\sum^{n-2}_{i=1} n\left(\sum_{\alpha}\int_{e_{a_i}(\Gamma_\alpha)}\rho^{a_i}\log\rho^{a_i}-\int_{e_{a_{i+1}}(\Gamma_\alpha)}\rho^{a_{i+1}}\log\rho^{a_{i+1}}\right)^2\\
	\geq & \sum^{n-2}_{i=1} n \Big(\ent_t(\mu^{a_i})-\ent_t(\mu^{a_{i+1}})\Big)^2\geq \frac{n}{n-2} \Big(\ent_t(\mu^{1-1/n})-\ent_t(\mu^{1/n})\Big)^2.
\end{align}
Combining the fact that $(X,d_t,\m_t)_{t\in I}$ is a weak super Ricci flow and the log-Lipschitz control from \cref{asm:localtoglobal}, we infer that for every $t$ the space $(X,d_t,\m_t)$ satisfies $\cd(K_t,\infty)$ for some $K_t\in\R$. Hence $\ent_t$ is continuous along each $W_t$-geodesic.
The conclusion follows by letting $n\to\infty$ and $\delta\to 0$.
\end{proof}

\section{Smooth flows}\label{sec:smoothRF}
In this section we consider \emph{smooth flows} $(M,g_t,f_t)_{t\in I}$, i.e. smooth families of complete Riemannian manifolds $(M,g_t)$ and weight functions $f_t$. We say that $(M,g_t,f_t)_{t\in I}$ is a \emph{compact smooth flow} if in addition $M$ is a closed manifold. We denote by $\m_t=e^{-f_t}\mathrm{vol}_{g_t}$ the weighted volume measure and by $d_t$ the Riemannian distance induced by $g_t$.

\subsection{Rough super/sub-Ricci flow}\label{sec:Rough_infty}
We begin this section by proving precise (two-sided) expansion estimates for the Wasserstein distance along heat flows.

Let $\geo(M,d_t)$ denote the set of all constant speed geodesics parameterized by $[0,1]$ on $(M,d_t)$, equipped with the supremum distance $d_{\infty}(\gamma,\tilde{\gamma})\coloneqq \sup_{a\in[0,1]}d_t(\gamma^a,\tilde{\gamma}^a)$ .
For each $\gamma\in \geo(M,d_t)$, we the \emph{average Ricci flow excess} along $\gamma$ given by
\begin{align}\label{def:rfe}
    \rfe_{t}(\gamma)\coloneqq \aint{[0,1]}\frac{1}{|\dot{\gamma}|^2}\left[\Ric_{f_t}(\dot\gamma^a)+\frac12\partial_tg_t(\dot\gamma^a)\right]\d a,
\end{align}
and set $\rfe_t(x,y)\coloneqq \inf_{\gamma}\rfe_t(\gamma)$, where the infimum is taken over all $\gamma\in \geo(M,d_t)$ connecting $x$ to $y$.

\begin{lemma}\label{lemma:rfe}
    Given a smooth flow $(M,g_t,f_t)$, the following holds
    \begin{enumerate}
        \item\label{item:09-06-1} for each $t\in I$ and $x,y\in M$, there is $\gamma\in \geo(M,d_t)$ s.t. $\rfe_t(x,y)=\rfe_t(\gamma)$;
        \item\label{item:09-06-3} for any sequence $(x_m,y_m)\to (x,y)$ and $t_m\to t$,
        \begin{equation}
            \liminf_{m\to \infty}\rfe_{t_m}(x_m,y_m)\geq \rfe_t(x,y).
        \end{equation}
    \end{enumerate}
\end{lemma}
\begin{proof}
   Consider a sequence $t_k\to t_0$ and a sequence $\gamma_k\in \geo(M,d_{t_k})$ converging to $\gamma_0$ under the supremum distance.
    Assume that there exists $N$ s.t. $\gamma_k|_{[\frac{i-1}{N},\frac{i}{N}]}$ is contained in a common coordinate chart for all $k\in \N$ and $1\leq i\leq N$. By definition we have
\[\rfe_{t_k}(\gamma_k)=\sum_{i=1}^{N}\frac1N\cdot \rfe_{t_k}(\gamma_k|_{[\frac{i-1}{N},\frac{i}{N}]})\;.\]
Recall that each $\gamma_k$ locally solves the geodesic equation, a second order ODE in the coordinate chart with smooth coefficients given by the Christoffel symbols. Hence $\dot\gamma_k$ continuously depend on boundary conditions $\gamma_k(\frac{i-1}{N})$ and $\gamma_k(\frac{i}{N})$ (see e.g. \cite[Theorem D.1]{Lee-manifold}).
In particular, $\gamma_0\in \geo(M,d_{t_0})$ and $\rfe_{t_k}(\gamma_k)\to \rfe_{t_0}(\gamma_0)$.
By the Arzela--Ascoli Theorem, the set of geodesics from $x$ to $y$ form a compact subset in $\geo(M,d_t)$. Thus by the continuity of $\rfe_t$ we have just shown, there exists a minimizer $\gamma$ s.t. $\rfe_t(x,y)=\rfe_t(\gamma)$.

Consider $t_k\to t$, $x_k\to x$ and $y_k\to y$. For each $k$, we take $\gamma_k\in \geo(M,d_{t_k})$ s.t. $\rfe_{t_k}(x_k,y_k)=\rfe_{t_k}(x_k,y_k)$.
Up to passing to a subsequence, $\gamma_k$ converges to a geodesic $\gamma$ connecting $x$ to $y$ on $(M,g_{t})$.
Hence the assertion follows as $\rfe_{t}(\gamma_k)\leq \rfe_t(\gamma)=\lim_{k\to \infty}\rfe_{t_k}(\gamma_k)$.
\end{proof}

\begin{theorem}\label{thm:sharperestimates}
Let $(M,g_t,f_t)_{t\in I}$ be a compact smooth flow.
For each $t\in I$,

\begin{enumerate}
     \item for any $s_0<t$ there exist $C,\varepsilon_0>0$ so that  for every $x,y\in M$ with $d_t(x,y)< \varepsilon_0$ and $s\in(s_0,t)$,
     \begin{align}\label{eq:refupperest}
   W_s^2(\ap{\delta_x},\ap{\delta_y})\geq d_t^2(x,y)\Big(1-2(t-s)\Big[\rfe_{t}(x,y)+C\cdot\big(d_t(x,y)^2+t-s\big)\Big]\Big).
     \end{align}
    \item for every $x,y\in M$ and $\delta>0$, there exists $s_0<t$ s.t. for all $s\in (s_0,t)$
    \begin{equation}\label{eq:reflowerest}
         W_s^2(\ap{\delta_x},\ap{\delta_y})\le d_t^2(x,y)\Big(1-2(t-s)\Big[\rfe_t(x,y)-\delta\Big]\Big);
    \end{equation}
    moreover, there exists $\varepsilon_0>0$ such that \eqref{eq:reflowerest} holds for all $x,y\in M$ with $\varepsilon=d_t(x,y)<\varepsilon_0$ and $s\in(s_0,t)$ with $s_0$ depending only on $\varepsilon$.
\end{enumerate}
In particular, for every $\delta>0$ there exists $\varepsilon_0>0$ s.t. for all $x,y\in M$ with
$\varepsilon\coloneqq d_t(x,y)<\varepsilon_0$ there exists $s_0<t$ depending only on $\delta,\varepsilon$ and $M$ so that
 \begin{equation}
   e^{-(\rfe_t(x,y)+\delta)(t-s)}\cdot d_t(x,y)\leq W_s(\ap{\delta_x},\ap{\delta_y})\leq e^{-(\rfe_t(x,y)-\delta)(t-s)}\cdot d_t(x,y), 
\end{equation}
for all $s\in(s_0,t)$.
\end{theorem}

\begin{proof}

{\bf (Lower estimate for Wasserstein distances under dual heat flows)} We adapt the same strategy as the proof of the analogous inequality in \cite[Theorem 3.4]{sturm2017remarks} for the static case, taking into account the additional complexity arising from the time-dependence of the metric with the aid of \cref{lemma:spacetimepotentials}.

Let $s_0,t\in I$ be fixed. 
Choose $\varepsilon>0$ sufficiently small so that \cref{lemma:spacetimepotentials} applies.
Let $x,y\in M$ with $2r\coloneqq d_t(x,y)\le \varepsilon,$ and let $\phi^a_s$ with $a\in[-r,r],s\in[s_0,t]$ be the family of potentials constructed in \cref{lemma:spacetimepotentials}.

By Kantorovich duality \eqref{eq:KantorovichDuality}, we have that
\begin{align}
   \frac{W_s^2(\ap{\delta_x},\ap{\delta_y})-d^2_t(x,y)}{2r\cdot 2(t-s)}&\ge \frac1{t-s}\left[\int\phi_s^{-r}\d\ap{\delta_x}-\int\phi^r_s\d\ap{\delta_y}-\int \phi^{-r}_t\d\delta_x+\int\phi^r_t\d\delta_y\right]
    \\ &=\frac1{t-s}\big[P_{t,s}\phi^{-r}_s(x)-\phi^{-r}_t(x)-(P_{t,s}\phi^{r}_s(y)-\phi^{r}_t(y))\big]
    \\ &=-\partial_s\lvert_{s=t} (P_{t,s}\phi_s^{-r}(x))+\partial_s\lvert_{s=t} (P_{t,s}\phi_s^{r}(y))
    \\ &\phantom{=}+\aint{[s,t]}(s-\tau)\partial_{\tau}^2(P_{t,\tau}\phi_{\tau}^{r}(y)-P_{t,\tau}\phi_{\tau}^{-r}(x))\d\tau
    \\&=\Delta_{f_t}\phi_t^{-r}(x)-\Delta_{f_t}\phi_t^r(y)-\partial_t\phi_t^{-r}(x)+\partial_t\phi_t^{r}(y)
    \\ &\phantom{{=}}+\aint{[s,t]}(s-\tau)\partial_{\tau}^2(P_{t,\tau}\phi_{\tau}^{r}(y)-P_{t,\tau}\phi_{\tau}^{-r}(x))\d\tau,\label{eq:lowerest1}
\end{align}
where we applied Taylor expansion in the second equality, and in the last equality we used that for any smooth function $\phi$ and $\tau\leq t$:
\begin{equation}\label{eq:sderivativeHF}
\partial_{\tau}P_{t,\tau}\phi=-P_{t,\tau}\Delta_{f_{\tau}}\phi,\quad\forall \tau\leq t.
\end{equation}
The above can be argued as follows
\begin{equation}
\partial_{\tau}P_{t,\tau}\phi=\lim_{h\searrow 0} \frac{P_{t,\tau+h}\phi-P_{t,\tau}\phi}{h}=\lim_{h\searrow 0}P_{t,\tau+h}(\frac{\phi-P_{\tau+h,\tau}\phi}{h})=-P_{t,\tau}\Delta_{f_{\tau}}\phi.
\end{equation}
Let $(\gamma^a)_{a\in[-r,r]}$ be a unit speed geodesic from $x$ to $y$.
Using the Hamilton-Jacobi equation and the expression of $\nabla_t\phi^a$ along $\gamma$ from \cref{lemma:spacetimepotentials}, as well as the Bochner identity \eqref{eq:inftyBochner}, we have:
\begin{align}
  \Delta_{f_t}\phi_t^{-r}(x)-\Delta_{f_t}\phi_t^r(y)
  &=-\int_{-r}^r\partial_a(\Delta_{f_t}\phi_t^{a}(\gamma^a))\d a
  \\&=-\int_{-r}^r\left(\frac12 \Delta_{f_t}|\nabla_t\phi^a_t|^2_{g_t}+g_t(\nabla_t\Delta_{f_t}\phi^a_t, \dot \gamma^a)\right)(\gamma^a)\d a   
  \\&=-\int_{-r}^r\left(\frac12 \Delta_{f_t}|\nabla_t\phi^a_t|^2_{g_t}+g_t(\nabla_t\Delta_{f_t}\phi^a_t,-\nabla_t\phi^a_t)\right)(\gamma^a)\d a
  \\&=-\int^r_{-r}\left(||\nabla_t^2\phi^a_t||^2_{\mathrm{HS}}(\gamma^a)+\Ric_{f_t}(\dot\gamma^a)\right)\d a\label{ineq:A}\;.
\end{align}
Arguing exactly as in \cite[Thm.~3.1]{sturm2017remarks}, we have that 
\begin{equation}\label{eq:Hessianphi}
    \|\nabla_t^2\phi^a_t\|^2_{\mathrm{HS}}(\gamma^a)\leq \sigma\tan^2(\sqrt\sigma|a|)\quad \text{for all $a$}\;,
\end{equation}
where $\sigma$ is an upper bound on the modulus of the Riemann tensor along $\gamma$.
Similarly, we obtain
\begin{align}
    \partial_t\phi_t^{r}(y)-\partial_t\phi_t^{-r}(x)&=\int_{-r}^r\partial_a(\partial_t\phi^a_t(\gamma^a))\d a
    \\ &=\int_{-r}^r(\partial_t\partial_a\phi^a_t)(\gamma^a)
    +g_t(\nabla_t\partial_t\phi^a_t(\gamma^a),\dot\gamma^a)\d a
    \\&=\int_{-r}^r (\frac12\partial_t|\nabla_t\phi^a_t|_{g_t}^2-g_t(\nabla_t\partial_t\phi^a_t,\nabla_t\phi^a_t))(\gamma^a)\d a
    \\&=\int^r_{-r}-\frac12\partial_tg_t(\dot\gamma^a,\dot\gamma^a)\d a.\label{ineq:B}
\end{align}
The last equality follows from the fact that in local coordinate $|\nabla_t\phi|^2_{g_t}=g^{ij}_t\partial_i\phi\partial_j\phi$ and $\partial_t g^{ij}_t=-(\partial_t g)^{ij}$.

We are left to estimate the remainder term in \eqref{eq:lowerest1}.
We use again \eqref{eq:sderivativeHF} to write
\begin{align}
    \partial^2_\tau P_{t,\tau}\phi^a_\tau&=\partial_\tau(- P_{t,\tau}\Delta_{f_\tau}\phi^a_\tau+P_{t,\tau}\partial_\tau\phi^a_\tau)
    \\&=P_{t,\tau}\Delta^2_{f_\tau}\phi^a_\tau-P_{t,\tau}\partial_\tau(\Delta_{f_\tau}\phi^a_\tau)-P_{t,\tau}\Delta_{f_\tau}\partial_\tau\phi^a_\tau+P_{t,\tau}\partial^2_\tau\phi^a_\tau
    \\&=P_{t,\tau}\big(\Delta^2_{f_\tau}\phi^a_\tau-\partial_\tau(\Delta_{f_\tau}\phi^a_\tau)-\Delta_{f_\tau}\partial_\tau\phi^a_\tau+\partial^2_\tau\phi^a_\tau\big)
    \\ &\eqqcolon P_{t,\tau}A_{\tau}\phi^a_\tau.
\end{align}
Thus,
\begin{align}
    \partial^2_\tau\big(P_{t,\tau}\phi^r_\tau(y)-P_{t,\tau}\phi^{-r}_\tau(x)\big)&=\int_{-r}^r\partial_a\big(\partial^2_\tau P_{t,\tau}\phi^a_\tau(\gamma^a)\big)\d a
    \\&=\int_{-r}^r\partial_a\big(P_{t,\tau}A_{\tau}\phi^a_\tau(\gamma^a)\big)\d a
    \\&=\int_{-r}^r P_{t,\tau}A_{\tau}\partial_a\phi^a_\tau(\gamma^a)+g_t(\nabla_t P_{t,\tau}A_{\tau}\phi^a_\tau, \dot\gamma^a)\d a
    \\ &=\int_{-r}^r\frac12P_{t,\tau}A_{\tau}|\nabla\phi^a_\tau|^2(\gamma^a)+g_t(\nabla_t P_{t,\tau}A_{\tau}\phi^a_\tau, \dot\gamma^a)\d a.
\end{align}
Since $(M,g_t,f_t)$ is a smooth flow of compact manifolds, there exists $K\in\R$ so that $\Ric_{f_t}+\frac12\partial_tg_t\ge Kg_t$.
Hence by the gradient estimate \eqref{ineq:gradientestimate}, we have
\begin{align}
   |\nabla_tP_{t,\tau}A_{\tau}\phi^a_\tau|^2_{g_t}\le e^{-2K(t-\tau)}P_{t,\tau}|\nabla_{\tau} A_{\tau}\phi^a_{\tau}|^2_{g_\tau}.
\end{align}
Further with the Markovian property of the heat propagator and the fact that fifth order derivatives of potentials are uniformly bounded by \cref{lemma:spacetimepotentials}, there exists a constant $C$ depending only on $s_0$, $\varepsilon$ and $M$ so that
\begin{align}
    \aint{[s,t]}(s-\tau)\partial_\tau^2(P_{t,\tau}\phi_\tau^{r}(y)-P_{t,\tau}\phi_\tau^{-r}(x))\d\tau\leq 2Cr(t-s),\label{ineq:C}
\end{align}
uniformly for all $x,y$ with $d_t(x,y)\leq \varepsilon$.

Combining \eqref{eq:lowerest1} with \eqref{ineq:A}, \eqref{ineq:B} and \eqref{ineq:C} and noting that by a simple reparametrisation argument 
\[\int_{-r}^r\Ric_{f_t}(\dot\gamma^a) +\frac12\partial_tg_t(\dot\gamma^a)\d r = 2r\rfe_t(x,y)\;, \]
we get
\begin{align}\label{ineq:est42}
   \frac{W_s^2(\ap{\delta_x},\ap{\delta_y})-d^2_t(x,y)}{2r\cdot 2(t-s)}\geq -2r\Big[\rfe_{t}(x,y)+H(r))-C(t-s)\Big]\;,
\end{align}
where $H(r)=\sup_{a\in[-r,r]} \|\nabla_t^2\phi^a_t\|^2_{\mathrm{HS}}(\gamma^a)\leq\sigma\tan^2(\sqrt\sigma r)=O(r^2)$ by \eqref{eq:Hessianphi}. This gives the desired estimate \eqref{eq:refupperest}.

{\bf (Upper estimate for Wasserstein distances under dual heat flows I)} The proof goes in the lines of the proof of the analogous claim (Theorem 2) in \cite{McCann-Topping}.
For $r<t$, denote by $\rho^0_r$ and $\rho^1_r$ the densities of $\mu^0_r\coloneqq\hat{P}_{t,r}\delta_x$ and $\mu^1_r\coloneqq\hat{P}_{t,r} \delta_y$ respectively.
Let $(\mu^a_r)_{a\in[0,1]}$ be the constant speed geodesic from $\mu^0_r$ to $\mu^1_r$ under $W_r$.
For all $\tau<r$, we have that $\mu^i_{\tau}=(\psi^i_{r,\tau})_\#\mu^i_r$ for $i\in\{0,1\}$, where $\psi^i_{r,\tau}$ is a family of diffeomorphisms generated by the (time dependent) vector fields $\nabla_{\tau}\log \rho^i_{\tau}$ and with $\psi^i_{r,r}=\id$, see \cite[Chapter 6]{ToppingRFbook}.
Consider $\pi_{r}\in \mathrm{OptGeo}(\mu^0_{r},\mu^1_{r})$ and $\sigma_{r}\coloneqq (e_0,e_1)_{\#}\pi_{r}$.
Then $(\psi^0_{r,\tau},\psi^1_{r,\tau})_\#\sigma_{r}$ is a coupling between $\mu^0_{\tau}$ and $\mu^1_{\tau}$. 
Thus,

\begin{align}
    \lim_{h\searrow 0} \frac{1}{h}(W_{r}^2(\mu^0_r,\mu^1_r)-W_{r-h}^2(\mu_{r-h}^0,\mu^1_{r-h}))&\geq \lim_{h\searrow 0} \frac{1}{h} \int d^2_{r}(x,y)-d^2_{r-h}(\psi^0_{r,r-h}(x),\psi^1_{r,r-h}(y))\d\sigma_r(x,y)\\
    &=\int \partial_{\tau}|_{\tau=r}\left[d^2_{\tau}(\psi^1_{r,\tau}(x),\psi^1_{r,\tau}(y))\right]\d\sigma_r.
\end{align}
From \cite{Cordero-Erausquin2001} the cut locus is $\sigma$-negligible, and outside the cut locus, we can compute the above derivative as in \cite[Remark 6]{McCann-Topping}
\begin{align}
   \partial_{\tau}|_{\tau=r} d^2_{\tau}(\psi^1_{r,\tau}(x),\psi^1_{r,\tau}(y))=\int_0^1(\partial_{r}g_{r})(\dot{\gamma}^a,\dot{\gamma}^a)\d a +2g_{r}(\nabla_r\log\rho^1_{r}(\gamma^1),\dot\gamma^1)-2g_r(\nabla_r\log\rho^0_{r}(\gamma^0),\dot\gamma^0)
\end{align}
where $\gamma$ is the unique $d_r$-geodesic from $x$ to $y$.
\cite[Lemma 8]{McCann-Topping} shows that
\begin{align}
    \left.\frac{\d}{\d a}\right\lvert_{a=1^-} \ent_r(\mu^a_r)&\leq  \int g_r(\nabla_r\log\rho^1_{r}(\gamma^1),\dot\gamma^1)\d\pi_r(\gamma);\\
     -\left.\frac{\d}{\d a}\right\lvert_{a=0^+} \ent_r(\mu^a_r)&\leq \int -g_r(\nabla_r\log\rho^0_{r}(\gamma^0),\dot\gamma^0)\d\pi_r(\gamma)\;.
\end{align}
Applying \cref{thm:displaceconvex} (and using the function $\ell$ defined there) we obtain
\begin{align}\label{ineq:411}
     \lim_{h\searrow 0} \frac{1}{2h}(W_{r}^2(\mu^0_r,\mu^1_r)&-W_{r-h}^2(\mu_{r-h}^0,\mu^1_{r-h}))\\
     &\geq \int \int_0^1(\frac12\partial_{r}g_{r}+\Ric_{N,f_r})(\dot{\gamma}^a,\dot{\gamma}^a)+\frac{(\partial_a\ell(a,\gamma^0))^2}{N}\d a\d \pi_r(\gamma).\label{ineq:06/10-1}
\end{align}
Taking $N=\infty$ and keeping in mind that for $\pi_r$-a.e. $\gamma$, $\gamma_0$ and $\gamma_1$ are non-conjugate, we get 
\begin{align}
    \partial_r[ W^2_r(\mu^0_r,\mu^1_r)]&\geq 
    \int 2\rfe_r(p,q)d^2_r(p,q)\d \sigma_r(p,q).\label{ineq:DW_2^2}
\end{align}

By stability of optimality (cf. \cite[Theorem 5.20]{Villani}), $r\mapsto \sigma_r$ is continuous under the weak topology (dual of $C_b(M\times M)$), and $r\mapsto W_r^2(\mu^0_r,\mu^1_r)$ is locally absolutely continuous by \cref{lemma:dualheatflow}.
Hence integrating \eqref{ineq:DW_2^2} over $(s,t)$ implies
\begin{equation}\label{ineq:DW^2_2-2}
    W_t^2(\mu^0_t,\mu^1_t)-W_s^2(\mu^0_s,\mu^1_s)\geq \int^t_s\int 2\rfe_{\tau}(p,q)d^2_{\tau}(p,q)\d \sigma_{\tau}(p,q)\d \tau.
\end{equation}
By the weak convergence $\sigma_{\tau}\rightharpoonup \sigma_t=\delta_x\otimes \delta_y$ and the lower semi-continuity of $\rfe_{\tau}(p,q)d^2_{\tau}(p,q)$ ensured by \cref{lemma:rfe}, we obtain
\begin{equation}
    \liminf_{\tau\to t}\int\rfe_{\tau}(p,q)d^2_{\tau}(p,q)\d \sigma_{\tau}\geq \rfe_t(x,y)d^2_t(x,y).
\end{equation}
Hence we can pass to the limit in \eqref{ineq:DW^2_2-2} to obtain \eqref{eq:reflowerest}.

{\bf (Upper estimate for Wasserstein distances under dual heat flows II--error estimate)}
To refine the upper estimate in Part I, we build on \eqref{ineq:DW^2_2-2} and progress towards \eqref{eq:reflowerest} by expressing it as follows
\begin{align}
      W_s^2(\mu^0_s,\mu^1_s)&\leq d^2_t(x,y)- \int^t_s\int 2\rfe_{\tau}(p,q)d^2_{\tau}(p,q)\d \sigma_{\tau}(p,q)\d \tau\\
      &=d^2_t(x,y)\big(1-2(t-s)\rfe_t(x,y)\big)\\
      &\quad +2\int^t_s\int \underbrace{\rfe_t(x,y)d^2_t(x,y)-\rfe_{\tau}(p,q)d^2_{\tau}(p,q)}_{E}\d \sigma_{\tau}(p,q)\d \tau\label{term:E}.
\end{align}
We decompose the error term $E$ in \eqref{term:E} as follows
\begin{align}
    E\coloneqq \underbrace{\rfe_t(x,y)d^2_t(x,y)-\rfe_{t}(x,y)d^2_{\tau}(p,q)}_{E_1}&+ \underbrace{\rfe_t(x,y)d^2_{\tau}(p,q)-\rfe_{\tau}(x,y)d^2_{\tau}(p,q)}_{E_2}\\
 + \underbrace{\rfe_\tau(x,y)d^2_\tau(p,q)-\rfe_\tau(p,q)d^2_\tau(p,q)}_{E_3}.
\end{align}
The term $E_1$ can be bounded with the help of the lower estimate \eqref{eq:refupperest}.
That is 
\begin{align}
 &\int E_1\d\sigma_{\tau}=\rfe_t(x,y)(d^2_t(x,y)-W^2_{\tau}(\mu^0_{\tau},\mu^1_{\tau}))\\
 \leq & 2\rfe_t(x,y)(t-s)d^2_t(x,y)\big(\rfe_{t}(x,y)+C\cdot[d_t(x,y)+t-s]\big)
\end{align}
with $\varepsilon_0,C>0$ chosen properly as in the lower estimate.

Now we consider the rest two terms.
Notice that inside the injectivity radius, $\rfe_t(x,y)$ is smooth in $x,y$ and $t$.
 As a smooth family of compact manifolds, $M$ admits a positive lower bound for the injectivity radius, see e.g. \cite{Enrlich1974}.
 Hence there exist $\varepsilon_2$ and $C_2>0$ s.t. for all $x$ and $t$, $\rfe_t(x,\cdot)$ is $C_2$-Lipschitz on $B_t(x,\varepsilon_2)$.
Then
\begin{align}\label{ineq:06/06-1}
    \int E_2+E_3\d \sigma_\tau&\leq \int[ C_2|t-\tau|+C_3(d_{\tau}(x,p)+d_{\tau}(y,q))]d^2_{\tau}(p,q)\d \sigma_\tau(p,q).
\end{align}
We need a preliminary estimate for the distribution of $\sigma_\tau$ on points far away from $(x,y)$.
By the same argument in \cref{thm:eta-theta} via heat kernel estimates, we have for each $l>0$
\begin{align}
    e_{\tau,l}\coloneqq 1-\sigma_\tau\big(B_\tau(x,l)\times B_\tau(y,l)\big)\leq C'\exp(-\frac{l^2}{C'(t-\tau)}),
\end{align}
where $C'$ is independent of $x,y,t,\tau$.
Thus we can proceed \eqref{ineq:06/06-1} by decomposing the integral into two parts, one over $ B_\tau(x,l)\times B_\tau(y,l)$ and the other over the complement.
With the symbol $\lesssim$ omitting unimportant constants, we can write
\begin{align}
     \int E_2+E_3\d \sigma_\tau& \lesssim  (|t-\tau|+l)(d_\tau(x,y)+2l)^2+e_{\tau,l}.
\end{align}
Choosing $l=d_\tau(x,y)\sqrt{C'}$ and $s_0$ s.t. $t-s_0<\frac{-d^2_\tau(x,y)}{\log d^3_\tau(x,y)}$, then
\begin{equation}
    \int E_2+E_3\d \sigma_\tau\lesssim (|t-\tau|+d_{\tau}(x,y))d^2_\tau(x,y).
\end{equation}
Finally, by the log-Lipschitzness of metrics, we have 
\begin{equation}
    \int^t_s \int E\d\sigma_\tau \d \tau\lesssim (t-s)^2d^2_t(x,y)+(t-s)d^3_t(x,y)
\end{equation}
for all $s\in (s_0,t)$, given by $t-s_0\lesssim\frac{-d^2_t(x,y)}{\log d^3_t(x,y)}$.
This is sufficient to conclude the upper estimate.
\end{proof}

\begin{rem}
   The assumption of compactness is made only to avoid the technical difficulty of showing existence and smoothness of the heat flow on non-compact time-dependent Riemannian manifolds. If we assume a priori that the time-dependent heat flow exists and is smooth, we can obtain a similar a similar statement as in the previous theorem, where the constants appearing depend can be chosen uniformly on compact subsets. We however need one global geometric assumption, namely that $(M,g_t,f_t)_{t\in I}$ is a $(K,\infty)$-super-Ricci flow for some $K\in\R$ in order to control the remainder term in \eqref{eq:lowerest1} via the gradient estimate for the heat flow. Then one can obtain the following statement:
   \begin{quote}
         For each compact subset $M_0\subset M$ and $s_0<t$, there exist $\varepsilon>0$ and $C>0$ which depends only on $s_0$, $\varepsilon$ and $M_0$ so that \eqref{eq:refupperest} holds for $x,y\in M_0$ with $d_t(x,y)\le \varepsilon$, and $s\in(s_0,t)$.
   \end{quote}
   The same applies to other consistency results on the rough (super/$N$-super-/sub-)Ricci flow in this section.
\end{rem}

As a direct consequence of \cref{thm:sharperestimates}, the following relation holds between the $\vartheta$ quantities and the Ricci flow excess on closed manifolds.
\begin{corollary}\label{cor:theta+}
Let $(M,g_t,f_t)_{t\in I}$ be a compact smooth flow. Then we have for all $t\in I$, $x,y\in M$:
\begin{align}\label{eq:excess1}
{\rm RFex}(t,x,y) \leq \vartheta^-(t,x,y)\;.
\end{align}
For every $t_0\in I$ there is $\varepsilon>0$ such that for all $t\in I$, $x,y\in M$ non-conjugate with $|t-t_0|, d_t(x,y)<\varepsilon$: 
\begin{align}\label{eq:excess2}
 \vartheta^+(t,x,y) \leq {\rm RFex}(t,x,y) + \sigma_t \tan^2\big(\sqrt{\sigma_t} d_t(x,y)\big)\;,
\end{align}
where $\sigma_t$ is an upper bound on the modulus of the Riemann tensor along the geodesic from $x$ to $y$.
\end{corollary}

\begin{proof}
This follows immediately by letting $s\to t$ in \eqref{eq:reflowerest} and in \eqref{ineq:est42}.
\end{proof}

\begin{corollary}\label{cor:ThetavsRFex}
For any $t\in I$ and $x\in M$, 
\begin{align}
    \vartheta^*(t,x)&=\sigma_{\max}(\frac12\partial_t g_t+\Ric_{f_t})(x),\\
    \liminf_{y,z\to x}\vartheta^{\pm}(t,y,z)&=\sigma_{\min}(\frac12\partial_t g_t+\Ric_{f_t})(x),
\end{align}
where $\sigma_{\max}(\cdot)(x)$ and $\sigma_{\min}(\cdot)(x)$ denote the largest and smallest eigenvalue of the symmetric linear map on $T_xM$ respectively.
In particular, $\vartheta^{*}(t,x)$ is jointly continuous on $I\times M$.
\end{corollary}

\begin{theorem}\label{thm:smoothrough_infty}
  Let $(M,g_t,f_t)_{t\in I}$ be a compact smooth flow. Then
\begin{enumerate}
\item $(M,g_t,f_t)_{t\in I}$ is a rough super-Ricci-flow if and only if $\Ric_{f_t}+\frac12\partial_tg_t\geq 0$ if and only if $\vartheta^{\pm}(t,x,y)\ge 0$ for all $t\in I$, $x,y\in M$;
\item $(M,g_t,f_t)_{t\in I}$ is a rough sub-Ricci-flow if and only if $\Ric_{f_t}+\frac12\partial_tg_t\leq 0$ if and only if $\vartheta^*(t,x)\le 0$ for all $t\in I$, $x\in M$.
\end{enumerate}
\end{theorem}
\begin{proof}
By Theorem \ref{thm:RFvsTheta}, $(M,g_t,f_t)_{t\in I}$ is a rough super-Ricci-flow if and only if $\vartheta^-(t,x,y)\geq 0$ for all $t,x,y$.
Therefore the equivalence is proved by completing the following chain of implications
\begin{equation}
    \begin{array}{lll}
       &\vartheta^-(t,x,y)\geq 0, \quad \forall t\in I,x,y\in M&\Rightarrow \vartheta^+(t,x,y)\geq 0, \quad \forall t\in I,x,y\in M\\
  \Rightarrow  & \Ric_{f_t}+\frac12\partial_tg_t\geq 0, \quad \forall t\in I, x\in M&\Rightarrow   \rfe_t(x,y)\geq 0,\quad \forall t\in I ,x,y\in M\\
  \Rightarrow  & \vartheta^-(t,x,y)\geq 0, \quad \forall t\in I,x,y\in M&
    \end{array}
\end{equation}
where the second and fourth implication follow from \cref{cor:ThetavsRFex} and \cref{cor:theta+} respectively.

Now, consider sub-Ricci-flow.
 Combining \cref{thm:RFvsTheta} gives that $(M,g_t,f_t)_{t\in I}$ is a rough sub-Ricci flow if for a.e. $t$, $\vartheta^{*}(t,x)\leq 0$ for all $x$, and hence for all $t,x$ by the continuity from \cref{cor:ThetavsRFex}.
Then the equivalence can be shown similarly as above.
\end{proof}

\subsection{Rough \texorpdfstring{$N$}{N}-super-Ricci flow}\label{sec:RoughN}
\begin{definition}
Let $\mu,\nu\in \p_t(X)$, $s\le t$. 
Define
\begin{align}
    \vartheta^+(t,\mu,\nu)\coloneqq -\liminf_{s\nearrow t}\frac1{t-s}\log\frac{W_s(\ap\mu,\ap\nu)}{W_t(\mu,\nu)}.
\end{align}
\end{definition}
Replacing $\Ric_{f_t}$ by the $N$-tensor $\Ric_{N,f_t}$ in \eqref{def:rfe}, we define the Ricci flow excess $\rfe_{N,t}$ along $\gamma\in \mathrm{Geo}(M,d_t)$ with the extra dimension parameter $N$ as follows
\begin{align} \label{eq:rfe_n}
	\rfe_{N,t}(\gamma)\coloneqq \aint{}\frac{1}{|\dot{\gamma}|^2}\left[\Ric_{N,f_t}(\dot\gamma^a)+\frac12\partial_tg_t(\dot\gamma^a)\right]\d a.
\end{align}

\begin{proposition}\label{prop:theta_N}Let $(M,g_t,f_t)_{t\in I}$ be a compact smooth flow.
Then
\begin{enumerate}
    \item for any $W_t$-geodesic $ (\mu^a)_{a\in[0,1]}$ with $\mu^0,\mu^1\in \p_t(X)\cap D(\ent_t)$ and $\pi\in \mathrm{OptGeo}_{t}(\mu^0,\mu^1)$,
    \begin{align}\label{ineq:Roughup}
   \int \rfe_{N,t}(\gamma)d_t^2(\gamma^0,\gamma^1)\d\pi(\gamma)+\frac1N\big(\ent_t(\mu^0)-\ent_t(\mu^1)\big)^2\le W^2_t(\mu^0,\mu^1)\vartheta^+(t,\mu^0,\mu^1);
   \end{align}
   \item for every $\varepsilon>0$, there exists an open cover $\{U_i\}_{i\in\N}$ such that for every $i$ and open sets $V_0,V_1\subset U_i$, there exists a $W_t$-Wasserstein geodesic $(\mu^a)_{a\in[0,1]}$ so that $\spt\mu^0\subset V_0$, $\spt\mu^1\subset V_1$ and
   \begin{align}\label{ineq:Roughsub}
W^2_t(\mu^0,\mu^1)\left(\vartheta^+(t,\mu^0,\mu^1)-\varepsilon\right)\le    \int \rfe_{N,t}(\gamma)d_t^2(\gamma^0,\gamma^1)\d\pi(\gamma)+\frac{1}N\big(\ent_t(\mu^0)-\ent_t(\mu^1)\big)^2.
   \end{align}
\end{enumerate}
\end{proposition}
\begin{proof}
The first statement is obtained by following with small modification the same argument as for the upper estimate in \cref{thm:sharperestimates}. Namely, after \eqref{ineq:06/10-1} we apply \eqref{eq:1orderDentropy} and the Cauchy-Schwarz inequality to  obtain
\begin{align}
\iint_0^1\left(\partial_a\ell(a,\gamma^0)\right)^2\d a\d \pi_r(\gamma)&\geq \left(\iint \partial_a\ell(a,\gamma^0)\d\pi(\gamma)\right)^2\\
	&= \left(\ent_t(\mu^0)-\ent_t(\mu^1)\right)^2.\label{ineq:NEnt}\;.
\end{align}
This gives rise to the additional term in (1).

We now focus on the second assertion. Fix $\varepsilon>0$.
    Let $x,y\in M$ with $2r\coloneqq d_t(x,y)<\delta$ with $\delta$ small enough to apply \cref{lemma:spacetimepotentials}.
    For unit speed geodesic $\tilde{\xi}$ from $x$ and $y$, we have a family of smooth potentials $(\tilde\phi^{\tau})_{\tau\in[-r,r]}$ given by  \cref{lemma:spacetimepotentials}.
    Thus after reparametrization, we have $(\phi^a_s)^{a\in[0,1]}_{s\in[t-\delta,t+\delta]}$ and $\xi:[0,1]\to M$ by
    \begin{equation}
        \xi^a=\tilde\xi^{2ar-r},\quad \phi^a=2r\cdot\tilde\phi^{2ar-r}
    \end{equation}
    so that the following properties are satisfied:
    \begin{itemize}
        \item for all $a\in[0,1]$, $ \nabla_t \phi^{a}(\xi^{a})=-\dot{\xi}^{a}$ and at the point $\xi^{\frac12}$, 
        \begin{align}
       \quad \nabla^2_t\phi^{\frac12}(\xi^{\frac12})=\lambda I_n,\quad \quad \Delta_{g_t}\phi^{\frac12}(\xi^{\frac12})=n\lambda,\quad  \lambda\coloneqq \frac1{N-n}\nabla_t f_t(\xi^{\frac12})\cdot \dot{\xi}^{\frac12},
    \end{align}
    \item for each $s$, $(\phi^0_s,-\phi^1_s)$ is an admissible pair for the cost $d^2_s$;
    \item for each $s$, the function $(a,z)\mapsto\phi^a_s(z)$ on $[0,1]\times M$ holds
    \begin{equation}\label{eq:10/06 HJ}
          \partial_a\phi^a_s=\frac{1}{2}|\nabla_t\phi^a_s|^2_{g_s},
    \end{equation}
    \item there exists a constant $C$ depending only on $\delta,\varepsilon$ and $M$ that $\|\phi^a_s\|_{C^5([t-\delta,t+\delta]\times M)}\leq C\cdot r$ for all $a\in [0,1]$.
    \end{itemize}
    For convenience, we use the following short-hand notation
    \begin{equation}
    \mathcal{E}_{t}({\gamma^a})\coloneqq \big\lVert \nabla_t^2\phi^{a}-\left(\frac{\Delta_{g_t} \phi^{a}}{n}\right)I_n\rVert^2_{\mathrm{HS}}(\gamma^{a})+\frac{n}{N(N-n)}\left[\left(\frac{N-n}n\right)\Delta_{g_t}\phi^{a}+\nabla_t f\cdot \nabla_t\phi^{a}\right]^2(\gamma^{a}).
\end{equation}
In particular, by the choice of potentials, $\mathcal{E}_{t}(\xi^{\frac12})=0$.

Let now $\mu^0\coloneqq \frac{m_t\llcorner B_t(x,r_0)}{\m_t(B_t(x,r_0))}$ with $r_0\ll r$. 
Let $\pi\in \mathrm{OptGeo}_t(\mu^0,\mu^1)$ where $\mu^1\coloneqq T_\#\mu^0$ and $T(x)\coloneqq \exp{(-\nabla_t\phi^0(x))}$. 
Then for $\pi$-almost every $\gamma$ we have, that $\dot\gamma^a=-\nabla_t\phi^a(\gamma^a)$ for all $a\in [0,1]$.
Repeating the lower estimate in \cref{thm:sharperestimates} shows
\begin{align}
    \frac{W_s^2(\ap\mu,\ap\nu)-W_t^2(\mu,\nu)}{2(t-s)}\geq \frac{1}{t-s}\left( \int P_{t,s}\phi^0_s\d \mu^0-\int P_{t,s}\phi^1_s\d \mu^1-\int \phi^0_t\d \mu^0 +\int \phi^1_t\d \mu^1\right)
\end{align}
which after taking liminf as $s$ goes to $t$ gives
\begin{align}
    - W^2_t(\mu^0,\mu^1)\cdot \vartheta^+(t,\mu^0,\mu^1)\geq \int \Delta_t \phi^0_t\d \mu^0-\int \Delta_t\phi^1_t\d \mu^1 -\int \partial_t\phi^0_t\d\mu^0+\int \partial_t\phi^1_t\d \mu^1.
\end{align}
Proceeding as in \cref{thm:sharperestimates} and using \eqref{eq:NBochner}, we can write 
\begin{align}
    \int \Delta_t \phi^0_t\d \mu^0-\int \Delta_t\phi^1_t\d \mu^1&=-\int\int^1_0 \partial_a (\Delta_t\phi^a_t(\gamma^a))\d a\d \pi(\gamma)\\
    &=-\iint \frac{(\Delta_t\phi^a_t)^2}{N}+\Ric_{N,f_t}(\nabla_t\phi^a_t)+\mathcal{E}_t(\gamma^a)\d a\d \pi(\gamma)
\end{align}
and 
\begin{align}
    -\int \partial_t\phi^0_t\d\mu^0+\int \partial_t\phi^1_t\d \mu^1=\int \int^1_0 -\frac{1}{2}\partial_t g_t(\nabla_t\phi^a_t,\nabla_t\phi^a_t)(\gamma^a)\d a\d \pi(\gamma).
\end{align}
Summarizing the above shows that 
\begin{align}\label{ineq:08/10-1}
   W^2_t(\mu^0,\mu^1)\cdot \vartheta^+(t,\mu^0,\mu^1)\leq \int \rfe_{N,t}(\gamma)d^2_t(\gamma^0,\gamma^1)\d \pi(\gamma)+\iint \frac{(\Delta_t\phi^a_t)^2}{N}+\mathcal{E}_t(\gamma^a)\d a\d \pi(\gamma).
\end{align}
Next we will estimate the double integral on the right-hand side.
Using the equation (10) in \cite{sturm2017remarks}, and the $C^5$-smallness of $\phi^a$, we have
\begin{equation}
   \partial_a( \mathcal{D}_{i,j}\phi^a_t(\xi^a))=\left[R(e_i,\nabla_t\phi^a_t,e_j,\nabla_t\phi^a_t)+\sum_k \mathcal{D}_{i,k}\phi^a_t\cdot\mathcal{D}_{j,k}\phi^a_t\right](\xi^a)=O(r^2),
\end{equation}
where $\mathcal{D}_{i,j}$ represents the second order derivative with respect to a fixed orthonormal frame $\{e_i\}_i$ around $\xi$ (see more details in \cite{sturm2017remarks}). 
This clearly implies that
\begin{align}
    \partial_a(\Delta_{g_t}\phi^a_t(\xi^a))=O(r^2),\quad \ \partial_a\|\nabla^2_t\phi^a_t(\xi^a)\|_{\mathrm{HS}}=O(r^2).
\end{align}
By simple calculation, we have that
\begin{align}
    \partial_a(\nabla_t f_t\cdot \nabla_t\phi^a_t(\xi^a))&=\nabla_t f_t\cdot \nabla_t\frac{|\nabla_t\phi^a_t|^2_{g_t}}{2}+\nabla_t(f_t\cdot\nabla_t\phi^a_t)(-\nabla_t\phi^a_t)\\
    &=\nabla_t f_t\cdot \nabla_t\frac{|\nabla_t\phi^a_t|^2_{g_t}}{2}+\nabla^2_t\phi^a_t(\nabla_tf_t,-\nabla_t\phi^a_t)+\nabla^2_tf_t(\nabla_t\phi^a_t,-\nabla_t\phi^a_t)\\
    &=\nabla^2_tf_t(\nabla_t\phi^a_t,-\nabla_t\phi^a_t)=O(r^2)
\end{align}
Therefore, recalling $\mathcal{E}_t(\xi^{\frac12})=0$, we get by chain rule that $\mathcal{E}_t(\xi^{a})=O(r^3)$ for all $a$.

For the squared-Laplacian term in \eqref{ineq:08/10-1}, recall that $\Delta_t \phi^a_t=\Delta_{g_t}\phi^a_t-\nabla_tf_t\cdot\nabla_t\phi^a_t$.
Then we argue in the same way to have that
\begin{align}
    \int(\Delta_t\phi^a_t)^2(\xi^a)\d a-\Delta_t\phi^{\frac12}(\xi^{\frac12}_t)^2\leq \int 2|\Delta_t\phi^a_t|\cdot|\partial_a (\Delta_t\phi^{a}_t(\xi^a))|\d a=O(r^3)
\end{align}
and
\begin{align}
   & \Delta_t\phi^{\frac12}_t(\xi^{\frac12})^2-\left(\int\Delta_t\phi^a_t(\xi^a)\d a\right)^2\\
  = & \left(\Delta_t\phi^{\frac12}_t(\xi^{\frac12})+\int\Delta_t\phi^a_t(\xi^a)\d a \right)\cdot \left(\Delta_t\phi^{\frac12}_t(\xi^{\frac12})-\int\Delta_t\phi^a_t(\xi^a)\d a \right)  =O(r^3).
\end{align}
As a brief summary, there exists a constant $C$ depending only on $\delta,\varepsilon$ and $M$, that
\begin{align}\label{ineq:24/09-01}
   \mathcal{E}_t(\xi^a)\leq Cr^3,\quad \int(\Delta_t\phi^a_t)^2(\xi^a)\d a- \left(\int\Delta_t\phi^a_t(\xi^a)\d a\right)^2\leq  C\cdot r^3.
\end{align}
Now, we can shrink measures $\mu^0$ to $\delta_x$ by letting $r_0\to 0$.
Meanwhile, we have also limits $\mu^1\rightharpoonup \delta_y$ and $\pi\rightharpoonup \delta_{\xi}$, which in particular implies that 
\begin{align}\label{eq:24/09-03}
    \iint (\Delta_t\phi^{a}_t)^2(\gamma^{a})\d{a}\d\pi&\to \int(\Delta_t\phi^a_t)^2(\xi^a)\d a\\
    \left(\iint \Delta_t\phi^{a}_t(\gamma^{a})\d{a}\d\pi\right)^2&\to \left(\int\Delta_t\phi^a_t(\xi^a)\d a\right)^2\label{eq:24/09-04}.
\end{align}
Hence, combining \eqref{eq:1orderDentropy}, \eqref{ineq:24/09-01}, \eqref{eq:24/09-03} and \eqref{eq:24/09-04}, we conclude that there exists $\delta$ s.t. for any $x,y$ with $d_t(x,y)<\delta$ and every neighborhood $U,V$ of $x,y$, that there exist $\mu^0$ and $\mu^1$ supported inside $U$ and $V$ respectively having
\begin{align}
    \iint\mathcal{E}_t(\gamma^a)\d \pi \d a&\leq\varepsilon/2\cdot \iint|\nabla_t\phi^a|^2(\gamma^a)\d \pi\d a =\varepsilon/2
    \cdot W^2_t(\mu^0,\mu^1)\\
    \frac1N\iint (\Delta_t\phi^{a})^2(\gamma^{a})\d{a}\d\pi&\leq \frac{1}{N}\left(  \ent_t(\mu^1)-\ent_t(\mu^0)\right)^2+\varepsilon/2\cdot W_t^2(\mu^1,\mu^0).
\end{align}
Finally, \eqref{ineq:Roughsub} is concluded by plugging above inequalities into \eqref{ineq:08/10-1}.
The desired open cover can be chosen by the compactness of $M$ and the uniformity of the estimate with respect to the reference points taken.
\end{proof}

\begin{theorem}\label{thm:smoothroughN}
    Let $(M,g_t,f_t)_{t\in I}$ be a compact smooth flow. 
    Then $(M,d_t,\m_t)$ is a rough super-N-Ricci flow if and only if 
        \begin{align}
            \Ric_{N,f_t}\ge -\frac12\partial_t g_t.
        \end{align}
\end{theorem}
\begin{proof}
Suppose first that $\Ric_{N,f_t}\ge -\frac12 \partial_t g_t$. 
Fix two measures $\mu^0,\mu^1\in \p_t(X)$.
    At each $r<t$, applying \cref{prop:theta_N} for $\hat{P}_{t,r}\mu^0$ and $\hat{P}_{t,r}\mu^1$ yields
\begin{align}\label{ineq:09/10-2}
    \int 2\rfe_{N,r}(\gamma)d^2_r(\gamma^0,\gamma^1)\d\pi_r(\gamma)+\frac2N\big(\ent_r(\hat{P}_{t,r}\mu^0)-\ent_r(\hat{P}_{t,r}\mu^1)\big)^2\le \partial^+_r\big[W^2_r(\hat{P}_{t,r}\mu^0,\hat{P}_{t,r}\mu^1)\big],
\end{align}
where $\pi_r\in \mathrm{OptGeo}(\hat{P}_{t,r}\mu^0,\hat{P}_{t,r}\mu^1)$.      
Integrating above inequality over $r\in[s,t]$, we conclude that the flow is a rough $N$-super-Ricci flow.

Next suppose that $\Ric_{N,f_t}\ge -\frac12 \partial_t g_t$ does not hold for some $N$-super-Ricci flow i.e. there exist $\varepsilon_0>0$, $(t_0,v_0)\in I\times TM$ and a neighbourhood $U$ of $(t_0,v_0)$ such that 
\begin{align}\label{ineq:superRiccifail1}
\Ric_{N,f_{t}}(v)\leq-\frac12\partial_t g_{t}(v)-\varepsilon_0 g_t(v)\end{align}
for every $(t,v)\in U$.
Take $\varepsilon<\varepsilon_0/4$ sufficiently small s.t. there exist $\ell>0$ and a $W_t$-geodesic $(\mu^a)_{a\in[0,1]}$ satisfying \eqref{ineq:Roughsub} and $(t,\frac{1}{2\ell}\dot\gamma^a)\in U$ for $a\in[0,1]$ and each $\gamma$ in the support of the dynamical plan.
In other words, tangential vectors of the rescaled curve $(\xi^a\coloneqq \gamma^{\frac{1}{2}+\frac{a}{2\ell}})_{a\in[-\ell,\ell]}$ satisfy \eqref{ineq:superRiccifail1}.
Notice that the integral average \eqref{eq:rfe_n} is unchanged under reparametrization of curves by constant rescalings.
Therefore, 
\begin{align}
	\rfe_{N,t}(\gamma)=\aint{}\frac{1}{|\dot{\xi}|^2}\left[\Ric_{N,f_t}(\dot\xi^a)+\frac12\partial_tg_t(\dot\xi^a)\right]\d a\leq -\varepsilon_0.
\end{align}
Then by \eqref{ineq:Roughsub}, for $s$ close to $t$,
\begin{align}
     &\frac{W_t^2(\mu^0,\mu^1)-W_s^2(\ap\mu^0,\ap\mu^1)}{2(t-s)} \le W_t^2(\mu^0,\mu^1)\vartheta^+(t,\mu^0,\mu^1)+\varepsilon W_t^2(\mu^0,\mu^1)\\
    \leq &\int \rfe_{N,t}(\gamma)d^2_t(\gamma^0,\gamma^1)\d\pi(\gamma)+\frac{1}N\big(\ent_t(\mu^0)-\ent_t(\mu^1)\big)^2+2\varepsilon W^2_t(\mu^0,\mu^1) \\
     \leq & (-\varepsilon_0+3\varepsilon)\cdot W_t^2(\mu^0,\mu^1)+\frac1N\frac1{t-s}\int_{[s,t]}\big(\ent_r(\ap[t,r]\mu^0)-\ent_r(\ap[t,r]\mu^1)\big)^2\d r.
\end{align}
This violates the rough $N$-super-Ricci condition.
\end{proof}

\subsection{Weak \texorpdfstring{$N$}{N}-super-/ sub-Ricci flow}\label{sec:weaksub}
As in \cref{sec:RoughN}, we define
\begin{align}
\eta(t,\mu^0,\mu^1)\coloneqq\frac1{W_t^2(\mu^0,\mu^1)}\cdot\Big[ \partial^{+}_a \ent_t(\mu^a)\big|_{a=1}-\partial_a^{-} \ent_t(\mu^a)\big|_{a=0}+\frac12\partial_{t}^{-}W_{t^-}^2(\mu^0,\mu^1)\Big].
\end{align}
\begin{proposition}\label{prop:eta_N}Let $(M,g_t,f_t)_{t\in I}$ be a smooth flow. 
Then
\begin{enumerate}
    \item for any $W_t$-geodesic $ (\mu^a)_{a\in[0,1]}$ with $\mu^0,\mu^1\in \p_t(X)\cap D(\ent_t)$ and $\pi\in \mathrm{OptGeo}_{t}(\mu^0,\mu^1)$,
    \begin{align}\label{ineq:weakup}
   \int \rfe_{N,t}(\gamma)d^2_t(\gamma^0,\gamma^1)\d\pi(\gamma)+\frac1N\big(\ent_t(\mu^0)-\ent_t(\mu^1)\big)^2\le W^2_t(\mu^0,\mu^1)\eta(t,\mu^0,\mu^1);
   \end{align}
   \item for every $\varepsilon>0$, there exists an open cover $\{U_i\}$ such that for every $i$ and every open sets $V_0,V_1\subset U_i$, there exists a $W_t$-Wasserstein geodesic $(\mu^a)_{a\in[0,1]}$ so that $\spt\mu^0\subset V_0$, $\spt\mu^1\subset V_1$ and 
   \begin{align}\label{ineq:weaksub}
W^2_t(\mu^0,\mu^1)\left(\eta(t,\mu^0,\mu^1)-\varepsilon\right)\le    \int \rfe_{N,t}(\gamma)d^2_t(\gamma^0,\gamma^1)\d\pi(\gamma)+\frac{1}N\big(\ent_t(\mu^0)-\ent_t(\mu^1)\big)^2.
   \end{align}
\end{enumerate}
\end{proposition}
\begin{proof}
    The first item is a consequence of \cref{thm:displaceconvex} and \eqref{ineq:NEnt}.
    Hence we focus mainly on the second statement, though the strategy is similar to the proof of \cref{prop:theta_N}.
    
    Fix $\varepsilon>0$.
    Let $x,y\in M$ with $2r\coloneqq d_t(x,y)<\delta$ small enough.
    As in the proof of \cref{prop:theta_N}, we have a $d_t$-geodesic $\xi:[0,1]\to M$ from $x$ to $y$ and a family of smooth potentials $(\phi^a)_{a\in[0,1]}$ (since we are working on the static space $(M,g_t)$, there is no time-variation on $\phi$) so that the following properties are satisfied:
    \begin{itemize}
        \item for all $a\in[0,1]$, $ \nabla_t \phi^{a}(\xi^{a})=-\dot{\xi}^{a}$ and at the point $\xi^{\frac12}$, 
        \begin{align}
       \quad \nabla^2_t\phi^{\frac12}(\xi^{\frac12})=\lambda I_n,\quad \quad \Delta_{g_t}\phi^{\frac12}(\xi^{\frac12})=n\lambda,
    \end{align}
    where
    \[
    \lambda\coloneqq \frac1{N-n}\nabla_t f_t(\xi^{\frac12})\cdot \dot{\xi}^{\frac12},
    \]
    \item on $[0,1]\times M$, 
    \begin{equation}\label{eq:25/09-01}
          \partial_a\phi^a=\frac{1}{2}|\nabla_t\phi^a|^2_{g_t},
    \end{equation}
    \item there exists a constant $C$ depending only on $\delta,\varepsilon$ and $M$ that $\|\phi^a\|_{C^5(M)}\leq C\cdot r$ for all $a\in [0,1]$.
    \end{itemize}

Denote $\m_t\coloneqq e^{-f_t}\d\mathrm{vol}_{g_t}$.
Let now $\mu^0\coloneqq \rho^0\m_t\in\p(M)$ with $\rho^0\in C_{c}^{\infty}(B_t(x,r_0))$ and $r_0\ll r$. 
Let $\pi\in \mathrm{OptGeo}_t(\mu^0,\mu^1)$ where $\mu^1\coloneqq T_\#\mu^0$ and $T(x)\coloneqq \exp{(-\nabla_t\phi^0(x))}$. 
Then for $\pi$-almost every $\gamma$ we have, that $\dot\gamma^a=-\nabla_t\phi^a(\gamma^a)$ for all $a\in [0,1]$.

By direct caclulation, we get
\begin{align}
    \frac{\d}{\d{a}}\ent_t(\mu^{a})&=\int \frac{\d}{\d{a}}\log(\rho^{a}(\gamma^{a}))\d\pi(\gamma)=\int \frac{\partial_{a}\rho^{a}(\gamma^{a})-\nabla_t \rho^{a}(\gamma^{a})\cdot \nabla_t\phi^{a}(\gamma^{a})}{\rho^{a}(\gamma^{a})}\d\pi(\gamma)
    \\ &=\int -\nabla_t \rho^{a}\cdot \nabla_t\phi^{a} \d\m_t=\int \rho^{a}\Delta_t \phi^{a}\d\m_t,
\end{align}
and by the absolutely continuity of the entropy along $(\mu^a)_a$, 
\begin{align}\label{eq:24/09-02}
    \ent_t(\mu^1)-\ent_t(\mu^0)=\int^1_0\int\Delta_t\phi^a(\gamma^a)\d \pi(\gamma)\d a.
\end{align}
By \cref{lemma:continuityeq}, the map $(a,x)\mapsto\rho^a(x)$ is smooth and it satisfies the continuity equation
\begin{equation}
      \partial_a \rho^a=\mathrm{div}_{f_t}(\rho^a\nabla_t\phi^a),
\end{equation}
where $\mathrm{div}_{f_t}$ denotes the weighted divergence by $\mathrm{div}_{f_t}v\coloneqq \mathrm{div}v-\nabla_t f_t\cdot v$.
Then, together with \eqref{eq:25/09-01} and \eqref{eq:NBochner}, we have that
\begin{align}
    \frac{\d^2}{\d{a}^2}\ent_t(\mu^{a})&=\int \partial_{a}\rho^{a}\Delta_t\phi^{a}+\rho^{a}\Delta_t(\partial_{a} \phi^{a})\d\m_t
    \\ &=\int \mathrm{div}_{f_t}(\rho^{a} \nabla_t \phi^{a})\Delta_t\phi^{a}+\rho^{a}\Delta_t(\frac12|\nabla_t \phi^{a}|^2)\d\m_t
    \\ &= \int -\nabla_t\phi^{a}\cdot \nabla_t(\Delta_t \phi^{a})+\Delta_t(\frac12|\nabla_t\phi^{a}|^2)\d\mu^{a}
    \\ &=\int\frac{(\Delta_t\phi^{a})^2}{N}+\Ric_{N,f}(\nabla_t \phi^{a}) +\mathcal{E}_{t}({\gamma^a}) \d \pi(\gamma).\label{eq:''entropy}
\end{align}
 Handling the Laplacian-squared and the error term in \eqref{eq:''entropy} as in \cref{prop:theta_N}, we obtain the desired estimate on $\eta$ for suitable $\mu^0$ and $\mu^1$ that
\begin{align}
&\partial_a\ent_t(\mu^{a})\lvert_{{a}=1}-\partial_{a} \ent_t(\mu^{a})\lvert_{{a}=0}+\frac12\partial_{t}^-W_t^2(\mu^0,\mu^1)\\
=&\iint\frac{(\Delta_t\phi^{a})^2}{N}(\gamma^a)+\Ric_{N,f}(\nabla_t \phi^{a})(\gamma^a)+\frac{1}{2}\partial_t g_t(\nabla_t\phi^a)(\gamma^a) +\mathcal{E}_{t}(\gamma^{a}) \d \pi\d a\\
\le&\int\rfe_{N,t}(\gamma)d^2_t(\gamma^0,\gamma^1)\d\pi(\gamma)+\frac{1}{N}\left(  \ent_t(\mu^1)-\ent_t(\mu^0)\right)^2 +\varepsilon\cdot W_t^2(\mu^1,\mu^0).\qedhere
\end{align} 
\end{proof}

Naturally, performing the same argument as in \cref{sec:RoughN}, we have the following analogy of \cref{thm:smoothroughN}, see also Theorem 2.14 in \cite{Sturm2018Super}.

\begin{theorem}\label{thm:smoothWeak}
    Let $(M,g_t,f_t)_{t\in I}$ be a smooth flow. 
    Then $(M,d_t,\m_t)$ is a weak $N$-super-Ricci flow if and only if
    \begin{align} \label{eq:smoothsuperRicciflow}
    \Ric_{N,f_t}\ge-\frac12\partial_t g_t.
    \end{align}
\end{theorem}

As a byproduct of \cref{prop:eta_N} when $N=\infty$, we get an explicit relation between $\eta$ and the tensor $\Ric_{f_t}+\frac{1}{2}\partial_t g_t$ , parallel to what we have shown in \cref{sec:Rough_infty} for $\vartheta$.
Recall that $\eta(t,x,y)$ is defined by taking the infimum of $\eta(t,\mu^0,\mu^1)$ over $\mu^0$ and $\mu^1$.
Therefore, it is not difficult to deduce the following analogous results of \cref{cor:theta+}, \cref{cor:ThetavsRFex} and \cref{thm:smoothrough_infty} with the help of \cref{prop:eta_N}.
In particular, we recover the $\eta$-characterization of weak Ricci flow from \cref{thm:charaSRF} in the smooth setting.

\begin{corollary}\label{cor:eta+}
Let $(M,g_t,f_t)_{t\in I}$ be a compact smooth flow. Then we have for all $t\in I$, $x,y\in M$:
\begin{align}\label{eq:excess1eta}
{\rm RFex}(t,x,y) \leq \eta^-(t,x,y)\;.
\end{align}
For every $t_0\in I$ there is $\varepsilon>0$ such that for all $t\in I$, $x,y\in M$ non-conjugate with $|t-t_0|, d_t(x,y)<\varepsilon$: 
\begin{align}\label{eq:excess2eta}
 \eta^+(t,x,y) \leq {\rm RFex}(t,x,y) + \sigma_t \tan^2\big(\sqrt{\sigma_t} d_t(x,y)\big)\;,
\end{align}
where $\sigma_t$ is an upper bound on the modulus of the Riemann tensor along the geodesic from $x$ to $y$. In particular, for any $t\in I$ and $x\in M$, 
\begin{align}
    \eta^*(t,x)&:=\liminf_{y,z\to x}\eta^\pm(t,y,z)=\sigma_{\max}(\frac12\partial_t g_t+\Ric_{f_t})(x),\\
    \eta^\pm(t,x,x)&= \liminf_{y,z\to x}\eta^\pm(t,y,z)=\sigma_{\min}(\frac12\partial_t g_t+\Ric_{f_t})(x).
\end{align}
\end{corollary}

\begin{theorem}\label{thm:weak-char}
    Let $\big(M_t\big)_{t\in I}\coloneqq (M,g_t,f_t)$ be a smooth flow. Then
\begin{enumerate}
\item $\big(M_t\big)_{t\in I}$ is a weak super-Ricci-flow if and only if $\Ric_{f_t}+\frac12\partial_tg_t\geq 0$ if and only if $\eta(t,x,y)\ge 0$ for all $t\in I$, $x,y\in M$;
\item $\big(M_t\big)_{t\in I}$ is a weak sub-Ricci-flow if and only if $\Ric_{f_t}+\frac12\partial_tg_t\leq 0$ if and only if $\eta^*(t,x)\le 0$ for all $t\in I$, $x\in M$.
\end{enumerate}
\end{theorem}

\begin{rem}\label{rem:smooth-error}
The proof of the second statement of \cref{prop:eta_N} i.e. the smooth condition implies the synthetic sub-Ricci flow relies on a concrete construction of transports.
Consequently, we can estimate the entropy derivative at all intermediate times based on \eqref{eq:''entropy}.

In particular, the same proof actually gives the following stronger estimate for a smooth flow satisfying $\Ric_{f_t}\leq -\frac12\partial_t g_t$: for every $t\in I$, $\varepsilon>0$, there exists an open cover $\{U_i\}$ such that for every $i$ and every open $V_1,V_1\subset U_i$, there exists a $W_t$-Wasserstein geodesic $(\mu^a)_{a\in[0,1]}$ so that for all $0\leq \sigma<\rho\leq 1$
 \begin{equation}\label{ineq:endtointer}
         \partial_a^+\ent_t(\mu^a)\lvert_{a=\rho-}-\partial_a^-\ent_t(\mu^a)\lvert_{a=\sigma+}\le \frac{1}{\rho-\sigma}\left(-\frac12\partial^-_t W_{t^-}^2(\mu^\sigma,\mu^\rho)+\varepsilon W_t^2(\mu^\sigma,\mu^\rho) \right),
    \end{equation}
    or alternatively that for any $a\in (0,1)$
    \begin{equation}
        \frac{\d^2}{\d a^2}\ent_t(\mu^a)\leq -\frac12\partial^-_t|\dot{\mu}|^a_{W_t}+\varepsilon\cdot W_t^2(\mu^0,\mu^1),
    \end{equation}
   where $|\dot{\mu}|^a_{W_t}$ denotes the $W_t$ metric derivative of the curve $\mu$ at $a$.
   Note that compared with \cref{prop:endtointer}, \eqref{ineq:endtointer} has the expected scaling in the error term.
\end{rem}

\section{Comparing different notions of synthetic Ricci flow}\label{sec:comparison}
In this section, we always assume $(X,d_t,\m_t)_{t\in I}$ to be a time-dependent family of \textbf{compact} metric measure spaces satisfying \cref{asm:dualheatflow}.

\subsection{Preliminaries on synthetic super-Ricci flows}
One of the core contributions of Kopfer--Sturm's paper \cite{Kopfer-Sturm2018} is to establish the equivalence between different descriptions of synthetic super-Ricci flow in terms of dynamic convexity of entropy, contraction property of Wasserstein distances under dual heat flows and monotonicity of gradient estimates under the primal heat flow.

\begin{theorem}[Kopfer--Sturm I]\label{thm:Kopfer-Sturm}  Let $(X,d_t,\m_t)_{t\in I}$ be a family of  time-dependent compact m.m.s. satisfying \cref{asm:dualheatflow}.
Then for any $N\in(0,\infty)$ and $K\in\R$, the following are equivalent.
\begin{enumerate}
    \item\label{item:DCforS+RF} For a.e. $t\in I$ and every $W_t$-geodesic $(\mu^a)_{a\in[0,1]}$ with $\mu^0,\mu^1\in D(\ent_t)$
    \begin{equation}
        \partial^+_a \ent_t(\mu^a)|_{a=1-}-\partial^-_a \ent_t(\mu^a)|_{a=0+}\geq -\frac12\partial^-_t W_{t-}^2(\mu^0,\mu^1)+K\cdot W^2_t(\mu^0,\mu^1)+\frac1N|\ent_t(\mu^0)-\ent_t(\mu^1)|^2
    \end{equation}
    \item\label{item:WineqforS+RF} For all $0\leq s<t\leq T$ and $\mu,\nu\in \p_t(X)$,
    \begin{equation}
        e^{-2Ks}W_s^2(\hat{P}_{t,s}\mu,\hat{P}_{t,s}\nu)\leq e^{-2Kt}W_t^2(\mu,\nu)-\frac2N\int^t_se^{-2Kr}\left(\ent_r(\hat{P}_{t,r}\mu)-\ent_r(\hat{P}_{t,r}\nu)\right)^2\d r
    \end{equation}
    \item For all $u\in D(\ch_t)$ and all $0<s<t<T$
    \begin{equation}
        e^{2Kt}|\nabla_t(P_{t,s}u)|^2\leq e^{2Ks}P_{t,s}(|\nabla_su|^2)-\frac2N\int^t_se^{2Kr}(P_{t,r}\Delta_rP_{r,s}u)^2\d r.
        \end{equation}
\end{enumerate}
Moreover, if any one (hence all) of the above properties holds, then \eqref{item:DCforS+RF} holds for all $t\in I$.

We call a time-dependent m.m.s. a $(K,N)$-super-Ricci flow if it satisfies one (hence all) of the above properties.
\end{theorem}

A powerful tool developed by Kopfer--Sturm is the so-called dynamical EVI-gradient flows.
This necessitates the use of the dynamical Wasserstein distance, which we explain in the sequel.

Given probability measure $\mu,\nu\in\p(X)$, for $s<t$ we define the dynamical Wasserstein distance by
\begin{equation}\label{eq:defdynamicalW}
    W^2_{s,t}(\mu,\nu)\coloneqq \inf\left\{\int^1_0|\dot{\mu}^a|^2_{s+a(t-s)}\d a\right\}\;,
\end{equation}
where the infimum runs over all $2$-absolutely continuous curves $(\mu^a)_{a\in[0,1]}\subset \p(X)$ connecting $\mu$ and $\nu$, see \cite[Section 6.1]{Kopfer-Sturm2018}. Observe that due to the log-Lipschitz bounds,  the metric derivative in \eqref{eq:defdynamicalW} exists for almost every $a\in[0,1]$.
\begin{theorem}[Kopfer--Sturm II]\label{thm:Kopfer-Sturm2} Any $(K,N)$-super-Ricci flow in the sense of \cref{thm:Kopfer-Sturm} satisfies dynamical EVI$^-(K,\infty)$ property, i.e. for every $t\in I$, $\mu\in\p(X)$ and $\sigma\in D(\ent)$, we have 
\begin{equation}
    \frac12\partial_s^-\left(W^2_{s,t}(\hat{P}_{t,s}\mu,\sigma)\right)|_{s=t^-}\geq \ent_t(\mu)-\ent_t(\sigma)+\frac{K}{2}W_t^2(\mu,\sigma).
\end{equation}
\end{theorem}
\begin{proof}
    The case $K=0$ was proved by Kopfer and Sturm. 
    Here we explain how the statement for general $K$ follows from a rescaling argument used in \cite{Kopfer-Sturm2018}.
    Let $(X,d_t,\m_t)$ be a $(K,N)$-super-Ricci flow.
    For each $C\in\R$, we can define another time-dependent m.m.s. $(X,\tilde{d}_t,\tilde{\m}_t)_{t\in \tilde{I}}$ by taking
    \begin{equation}
        \tilde{d}_t=e^{-K\tau(t)}d_{\tau(t)},\quad \tilde{\m}_t=\m_{\tau(t)},\quad \tau(t)=\frac{-1}{2K}\log(C-2Kt)
    \end{equation}
    and $\tilde{I}=\{\tau(t):t\in I,2Kt<C\}$.
   
    By \cite[Theorem 1.11]{Kopfer-Sturm2018}, $(X,\tilde{d}_t,\tilde{\m}_t)_{t\in \tilde{I}}$ is an $N$-super-Ricci flow.
    Therefore applying Theorem 6.13 in \cite{Kopfer-Sturm2018} one obtains the EVI$^-(0,\infty)$ property for $(X,\tilde{d}_t,\tilde{\m}_t)_{t\in \tilde{I}}$.
    One can check that this translates into the EVI$^-(K,\infty)$ property on $(X,d_t,\m_t)_{t\in I}$.
\end{proof}

It is worth mentioning that any time-dependent metric measure space satisfying \cref{asm:dualheatflow} is a $(K-L,N)$-super-Ricci flow in the sense of \cref{thm:Kopfer-Sturm}, with $K$ the uniform synthetic Ricci curvature lower bound for all static space $(X,d_t,\m_t)$ and $L$ the log-Lipschitz bound on metrics.
Therefore, for any $\mu\in\p(X)$ the EVI$^-(K-L,\infty)$ property holds:
\begin{equation}\label{ineq:EVI1}
	\frac12\partial_s^-\left(W^2_{s,t}(\hat{P}_{t,s}\mu,\sigma)\right)|_{s=t^-}\geq \ent_t(\mu)-\ent_t(\sigma)+\frac{K-L}{2}W_t^2(\mu,\sigma)
\end{equation}
for all $\sigma\in \p_t(X)$.

\begin{lemma}
    Let $(\mu_s)_{s\in (s_0,t]}$ be a weakly continuous curve in $\p(X)$, then for any $\sigma\in \p(X)$,
    \begin{equation}\label{ineq:W_st}
    \partial_s^-\left(W^2_{s,t}(\mu_s,\sigma)\right)|_{s=t^-}\leq \partial_s^-\left(W_s^2(\mu_s,\sigma)\right)|_{s=t^-}+2LW^2_t(\mu_t,\sigma).
\end{equation}
\end{lemma}
\begin{proof}
    For any $s<t$, choose $(\mu^a_s)_{a\in[0,1]}$ a $2$-absolutely continuous curve in $\p(X)$ minimizing the $W_{s,t}(\mu_s,\sigma)$. 
    Then by $L$-log-Lipschitz control on metrics, one has
    \begin{align}
        W^2_s(\mu_s,\sigma)-W^2_{s,t}(\mu_s,\sigma)&\leq \int^1_0|\dot{\mu}^a_s|^2_s-|\dot{\mu}^a_s|^2_{s+a(t-s)}\d a\\
        &\leq \int^1_0(e^{2La(t-s)}-1)|\dot{\mu}^a_s|^2_{s+a(t-s)}\d a\\
        &\leq (e^{2L(t-s)}-1)W^2_{s,t}(\mu_s,\sigma).
    \end{align}
    Therefore, 
    \begin{align}
         \partial_s^-\left(W^2_{s,t}(\mu_s,\sigma)\right)|_{s=t^-}&=  \liminf_{s\nearrow t} \frac{1}{t-s}\left(W^2_t(\mu_t,\sigma)-W_{s}^2(\mu_s,\sigma)+W_{s}^2(\mu_s,\sigma)-W^2_{s,t}(\mu_s,\sigma)\right)\\
         &\leq  \partial_s^-\left(W_s^2(\mu_s,\sigma)\right)|_{s=t^-}+\limsup_{s\nearrow t}\frac{1}{t-s}(e^{2L(t-s)}-1)W^2_{s,t}(\mu_s,\sigma)
    \end{align}
    which shows \eqref{ineq:W_st} as $W_{s,t}(\mu_s,\sigma)$ converges to $W_t(\mu_t,\sigma)$ when $s\to t$.
\end{proof}

Note that \cref{asm:dualheatflow} implies that $\ent_t$ is semiconvex along $W_t$ geodesics for all $t$ and in particular upper regular. Hence, \cref{lemma:freeofchoice_eta} gives that $\eta^-_\varepsilon=\eta^+_\varepsilon$ and hence $\eta^+=\eta^-$. We will thus not distinguish between the two and write $\eta_\varepsilon$ and $\eta$ for short.

\subsection{Comparing rough and weak super-Ricci flows}\label{sec:comparingRF}
From studies on static spaces e.g. \cite{Erbar-Sturm,sturm2017remarks}, we know the quantity $\vartheta^{\pm}$ can detect conic singularities, in contrast to $\eta$-quantities.
This is consistent with the fact that $\eta(t,x,y)$ is lower semi-continuous in $x,y$ but  $\vartheta^{\pm}(t,x,y)$ not necessarily.
Hence, we introduce an intermediate quantity $\vartheta^\flat$ as
\begin{equation}
     \vartheta^\flat(t,x,y)\coloneqq \lim_{\varepsilon\to 0}\inf\left\{-\limsup_{s\nearrow t}\frac1{t-s}\log\frac{W_s(\ap\mu^0,\ap\mu^1)}{W_t(\mu^0,\mu^1))}\right\}\;,
\end{equation}
where we infimise over all $\mu^0,\mu^1$ with $\spt(\mu^0)\subset B_{t}(x,\varepsilon)$ and $\spt(\mu^1)\subset B_t(y,\varepsilon)$.
Trivially, we have $\vartheta^\flat\leq \vartheta^-$.
We see below that $\theta^\flat$ is the lower semi-continuous envelope of $\vartheta^-$.

For the sake of convenience in the exposition, we will set $\mu_s\coloneqq \hat{P}_{t,s}\mu$ and denote by $(\mu^a_s)_{a\in[0,1]}$ the $W_s$-geodesic from $\mu^0_s$ to $\mu^1_s$ in this and the next subsection.
\begin{proposition}\label{prop:eta>theta}
    For every $t\in I$ and $x,y\in X$,
    \begin{equation}
      \eta(t,x,y)\geq   \vartheta^\flat(t,x,y)= \liminf_{(\tilde{x},\tilde{y})\to (x,y)}\vartheta^-(t,\tilde{x},\tilde{y}).
    \end{equation}
\end{proposition}
\begin{proof}
i) We first show that $\eta(t,x,y)\geq   \vartheta^\flat(t,x,y)$.
Denote $\kappa\coloneqq \vartheta^\flat(t,x,y)$. 
For any $\delta>0$, by the definition of $\vartheta^\flat$, there exists $\varepsilon>0$ s.t. for any $\mu^0,\mu^1$ supported in $B_t(x,\varepsilon)$ and $B_t(y,\varepsilon)$ respectively, one has
    \begin{align}\label{ineq:bythetaflat}
      \frac12\partial_s^-\left(W^2_s(\mu_s^a,\mu_s^{1-a})\right)|_{s=t^-}\geq (\kappa-\delta)\cdot W^2_t(\mu^a,\mu^{1-a})  
    \end{align}
for all $a\in(0,\frac12)$ small enough making $\spt\mu^a\subset B_t(x,2\varepsilon)$ and $\spt\mu^{1-a}\subset B_t(y,2\varepsilon)$.
Applying \eqref{ineq:EVI1} to $\sigma=\mu^0$ together with \eqref{ineq:W_st} gives
\begin{equation}
    \ent_t(\mu^0)-\ent_t(\mu^a)\geq -\frac12 \partial_s^-\left(W^2_s(\mu^a_s,\mu^0)\right)|_{s=t^-}+\frac{a^2(K-2L)}{2}W^2_t(\mu^0,\mu^1)
\end{equation}
and similarly
\begin{equation}
    \ent_t(\mu^1)-\ent_t(\mu^{1-a})\geq -\frac12 \partial_s^-\left(W^2_s(\mu^{1-a}_s,\mu^1)\right)|_{s=t^-}+\frac{a^2(K-2L)}{2}W^2_t(\mu^0,\mu^1)
\end{equation}
Therefore, 
\begin{align}
  &\frac{1}{a}\left(\ent_t(\mu^1)-\ent_t(\mu^{1-a})+\ent_t(\mu^0)-\ent_t(\mu^a) \right)\\
  \geq &-\frac{1}{2a}\left(\partial_s^-\left(W^2_s(\mu^a_s,\mu^0)|_{s=t^-}+\partial_s^-W^2_s(\mu^{1-a}_s,\mu^1)\right)|_{s=t^-}\right)+a(K-2L)W^2_t(\mu^0,\mu^1)\\
  \overset{(*)}{\geq }& \frac{1}{2(1-2a)}\partial_s^-\left(W^2_s(\mu^a_s,\mu^{1-a}_s)\right)|_{s=t^-}-\frac12(\partial^-_tW^2_{t^-})(\mu^0,\mu^1)+a(K-2L)W^2_t(\mu^0,\mu^1)\\
  \overset{\eqref{ineq:bythetaflat}}{\geq}&\frac{\kappa-\delta}{1-2a} W^2_t(\mu^a,\mu^{1-a})-\frac12(\partial^-_tW^2_{t^-})(\mu^0,\mu^1)+a(K-2L)W^2_t(\mu^0,\mu^1).\label{ineq:applyingEVI}
\end{align}
Here, we shortly explain the inequality $(*)$. 
Since $(\mu^a)_{a\in [0,1]}$ is a $W_t$-geodesic 
\begin{align}
    W^2_t(\mu^0,\mu^1)=\frac1aW^2_t(\mu^0,\mu^a)+\frac{1}{1-2a}W^2_t(\mu^a,\mu^{1-a})+\frac{1}{a}W^2_t(\mu^{1-a},\mu^1)
    \end{align}
and for the $W_s$-distance, we apply the Cauchy's inequality
\begin{align}
    W^2_s(\mu^0,\mu^1)\leq\frac1aW^2_s(\mu^0,\mu^a_s)+\frac{1}{1-2a}W^2_s(\mu^a_s,\mu^{1-a}_s)+\frac{1}{a}W^2_s(\mu^{1-a}_s,\mu^1).
\end{align}
Then we obtain the required inequality 
\begin{align}
  (\partial^-_tW^2_{t^-})&(\mu^0,\mu^1) = \liminf_{s\nearrow t}\frac{1}{t-s}\left(W^2_t(\mu^0_t,\mu^1_t)- W^2_s(\mu^0_t,\mu^1_t)\right)\\
  &\geq \frac{1}{a}\left(\partial_s^-\left(W^2_s(\mu^a_s,\mu^0)\right)+\partial_s^-\left(W^2_s(\mu^{1-a}_s,\mu^1)\right)\right)|_{s=t^-}+\frac{1}{1-2a}\partial_s^-\left(W^2_s(\mu^a_s,\mu^{1-a}_s)\right)|_{s=t^-}.
\end{align}
In \eqref{ineq:applyingEVI}, taking $a$ to 0 leads to 
\begin{equation}
    \partial^+_a\ent_t(\mu^a)|_{a=1}-  \partial^-_a\ent_t(\mu^a)|_{a=0}+\frac12(\partial^-_tW^2_{t^-})(\mu^0,\mu^1)\geq (\kappa-\delta)\cdot W^2_t(\mu^0,\mu^1).
\end{equation}
The thesis follows by the arbitrariness of $\delta$.
\medskip

ii) Next, we show that $\vartheta^\flat$ is the lower semi-continuous envelope of $\vartheta^-$.
For any $\mu^0,\mu^1\in\p_t(X)$ supported in a $\varepsilon$-ball around $x$ and $y$ respectively, consider the optimal transport plan $\sigma$ from $\mu^0$ to $\mu^1$ for the cost $d^2_t$.
As in \cref{rmk:roughsuperRF}, we have 
\begin{align}
\liminf_{s\nearrow t} \frac{1}{t-s}&\left( W^2_t(\mu^0,\mu^1)-W^2_s(\hat{P}_{t,s}\mu^0,\hat{P}_{t,s}\mu^1)\right)\\
\geq & \liminf_{s\nearrow t} \frac{1}{t-s}\left(\int d^2_t(\tilde{x},\tilde{y})-W^2_s(\ap\delta_{\tilde{x}},\ap\delta_{\tilde{y}})\right)\d \sigma(\tilde{x},\tilde{y})\\
\geq & \int \liminf_{s\nearrow t}\frac{1}{t-s}\left(d^2_t(\tilde{x},\tilde{y})-W^2_s(\ap\delta_{\tilde{x}},\ap\delta_{\tilde{y}})\right)\d \sigma(\tilde{x},\tilde{y})\\
=&\int 2\vartheta^-(t,\tilde{x},\tilde{y})d^2_t(\tilde{x},\tilde{y})\d\sigma(\tilde{x},\tilde{y})\\
\geq & 2\inf\{\vartheta^-(t,\tilde{x},\tilde{y}):\tilde{x}\in B_t(x,\varepsilon),\tilde{y}\in B_t(y,\varepsilon)\}\;,
\end{align}
where Fatou's lemma applies in the second inequality since the integrands are non-negative due to the non-expansion of Wasserstein distance along dual heat flows, cf. \cref{thm:Kopfer-Sturm}.
Taking infimum over $\mu^0,\mu^1$ and letting $\varepsilon\to0$, one gets
\begin{equation}
     \vartheta^\flat(t,x,y)\geq \liminf_{(\tilde{x},\tilde{y})\to (x,y)}\vartheta^-(t,\tilde{x},\tilde{y}).
\end{equation}
The equality follows once we observe that $\vartheta^\flat(t,\dot,\cdot)$ is lower semi-continuous and $\vartheta^\flat\leq \vartheta^-$ by construction.
\end{proof}

As a direct consequence of \cref{prop:eta>theta}, together with \cref{thm:RFvsTheta} and \cref{thm:charaSRF}, a rough super-Ricci flow is a weak super-Ricci flow, recovering the implication $(2)\Rightarrow (1)$ in \cref{thm:Kopfer-Sturm}.

Furthermore, \cref{prop:eta>theta} together with 
\begin{itemize}
    \item[] \cref{thm:Kopfer-Sturm} on the equivalence between rough and weak super-Ricci flows
    \item[] \cref{prop:etatoSRF} on the local-to-global property of weak super-Ricci flow
\end{itemize} 
yields the following strengthening of \cref{thm:RFvsTheta}
\begin{corollary}
    Let $(X,d_t,\m_t)_{t\in I}$ be a time-dependent compact m.m.s. satisfying \cref{asm:dualheatflow}. The following are equivalent
 \begin{itemize}
     \item[(1)] $(X,d_t,\m_t)_{t\in I}$ is a weak super-Ricci flow;
     \item[(2)] for each $t\in I$ and $x,y\in X$, $\eta(t,x,y)\ge 0$;
     \item[(3)] for a.e. $t\in I$, $\eta(t,x,x)\ge 0$ for each $x\in X$;
     \item[(4)] $(X,d_t,\m_t)_{t\in I}$ is a rough super-Ricci flow;  
    \item[(5)] for each $t\in I$ and $x,y\in X$, $\vartheta^-(t,x,y)\ge 0$;
    \item[(6)] for each $t\in I$ and $x\in X$, $\vartheta^\flat(t,x,x)\ge 0$.
\end{itemize} 
\end{corollary}  

\subsection{Comparing rough and weak sub-Ricci flows}
The comparison of synthetic sub-Ricci flows demands a reverse direction of \cref{prop:eta>theta}.
The following theorem is analogue for time-dependent mms of \cite[Theorem 4.2]{sturm2017remarks} in the static case.

\begin{theorem}\label{thm:eta-theta}
For each pair of points $x,y\in X$, $\eta(t,x,y)\leq\vartheta^\flat(t,x,y)$ holds for almost each $t$.
\end{theorem}
\begin{proof}
Fix $x,y\in X$, $\delta>0$.
For $\varepsilon>0$ small (to be made more precise later), let $\mu^0,\mu^1$ be given with supports in $B_{t}(x,\varepsilon/2)$ and $B_{t}(y,\varepsilon/2)$ respectively.
 
 \textbf{Lower bound on the difference of Wasserstein distances by Kantorovich duality}
 Fix $s<t'<t$.
 Denote by $\phi_s$ and $\psi_s$ the pair of Kantorovich potentials for optimal transport from $\mu^0_s$ to $\mu^1_s$ w.r.t. the cost $d^2_s$.
Then difference of Wasserstein distances can be estimated as in \cite[Theorem 4.4]{Kopfer-Sturm2018}:
    \begin{align}
           &W_{t'}^2(\mu^0_{t'},\mu^1_{t'})-W^2_s(\mu^0_s,\mu^1_s)=\int^{t'}_s \partial^+_r [W^2_r(\mu^0_r,\mu^1_r)]\d r\\
    =&\int^{t'}_{s}\limsup_{h\searrow 0}\frac{1}{h}[W^2_{r}(\mu^0_{r},\mu^1_{r})-W^2_{r-h}(\mu^0_{r},\mu^1_{r})+W^2_{r-h}(\mu^0_{r},\mu^1_{r})-W^2_{r-h}(\mu^0_{r-h},\mu^1_{r-h})]\d r\\
    \geq &\int^{t'}_{s}\liminf_{h\searrow 0}\frac{1}{h}[W^2_{r}(\mu^0_{r},\mu^1_{r})-W^2_{r-h}(\mu^0_r,\mu^1_r)]\d r+\int^{t'}_{s}\limsup_{h\searrow 0}\frac1h[W^2_{r-h}(\mu^0_{r},\mu^1_{r})-W^2_{r-h}(\mu^0_{r-h},\mu^1_{r-h})]\d r\\
\geq &\int^{t'}_{s}\liminf_{h\searrow 0}\frac{1}{h}[W^2_{r}(\mu^0_{r},\mu^1_{r})-W^2_{r-h}(\mu^0_r,\mu^1_r)]+\liminf_{h\searrow 0}\int^{t'}_{s}\frac1h[W^2_{r-h}(\mu^0_{r},\mu^1_{r})-W^2_{r-h}(\mu^0_{r-h},\mu^1_{r-h})]\label{ineq:thm5.6-1}\\
=&\int^{t'}_{s}(\partial^-_rW^2_{r^-})(\mu^0_r,\mu^1_r)\d r+\liminf_{h\searrow 0}\int^{t'-h}_{s-h}\frac1h[W^2_{r}(\mu^0_{r+h},\mu^1_{r+h})-W^2_{r}(\mu^0_{r},\mu^1_{r})]\d r\\
\geq & \int^{t'}_{s}(\partial^-_rW^2_{r^-})(\mu^0_r,\mu^1_r)\d r +\int^{t'}_{s}2(\ch_r(\rho^0_r,\phi_r)+\ch_r(\rho^1_r,\psi_r))\d r.\label{ineq:thm5.6-2}\\
\geq &\int^{t'}_{s}(\partial^-_rW^2_{r^-})(\mu^0_r,\mu^1_r)\d r+2\int^{t'}_{s}(\partial^-_a \ent_r(\mu^a_r)\big|_{a=1}-\partial_a^+ \ent_r(\mu^a_r)\big|_{a=0})\d r\label{ineq:thm5.6-3}
    \end{align}
In the above, we have used Fatou's lemma in \eqref{ineq:thm5.6-1}, \cite[Proposition 4.3]{Kopfer-Sturm2018} for connecting the Cheeger's energy in \eqref{ineq:thm5.6-2} and \cite[Proposition 4.2]{Kopfer-Sturm2018} in \eqref{ineq:thm5.6-3} to get the entropy derivative (see estimate (6.19) in the proof of Theorem 6.5 in \cite{Ambrosio-Gigli-Mondino-Rajala}).

 Passing $t'$ to $t$, by the continuity of Wasserstein distances, one obtains
 \begin{align}
    W_{t}^2(\mu^0_{t},\mu^1_{t})-W_{s}^2(\mu^0_{s},\mu^1_{s})\geq 2\int^{t}_s (\partial^-_a \ent_r(\mu^a_r)\big|_{a=1}-\partial_a^+ \ent_r(\mu^a_r)\big|_{a=0})+\frac12(\partial^-_rW^2_{r^-})(\mu^0_r,\mu^1_r)\d r.\label{eq:contvsconvie}
 \end{align}

\textbf{Application of the $\eta$-quantity by measure-partitioning.}
Denote by $\pi_r$ the optimal dynamical plan between $\mu^0_r,\mu^1_r$ and $\Gamma_r$ a non-branching subset of $\spt(\pi_r)$ so that $\pi_r(\Gamma_r)=1$. 
 For any $\varepsilon>0$, put $\Gamma^\varepsilon_r=\big\{\gamma\in\Gamma_r: \gamma^0\in B_r(x,\varepsilon), \gamma^1\in B_r(y,\varepsilon)\big\}$, $\lambda_{r,\varepsilon}=\pi_r(\Gamma^\varepsilon_r)$ and
 $$\pi_{r,\varepsilon}=\frac1{\lambda_{r,\varepsilon}}\cdot \pi_r\llcorner\Gamma^\varepsilon_r,\quad
 \overline\pi_{r,\varepsilon}=\frac1{1-\lambda_{r,\varepsilon}}\cdot \pi_r\llcorner(\Gamma^\varepsilon_r)^c.$$
 Moreover, put 
  $\mu_{r,\varepsilon}^a=(e_a)_\sharp \pi_{r,\varepsilon}$ and 
 $\overline\mu_{r,\varepsilon}^a=(e_a)_\sharp \overline\pi_{r,\varepsilon}$ where $e_a$ is the evaluation map.
 Then
 $\pi_r=\lambda_{r,\varepsilon}\cdot \pi_{r,\varepsilon}+(1-\lambda_{r,\varepsilon})\cdot \overline\pi_{r,\varepsilon}$ as well as
  $\mu^a_r=\lambda_{r,\varepsilon}\cdot \mu^a_{r,\varepsilon}+(1-\lambda_{r,\varepsilon})\cdot \overline\mu^a_{r,\varepsilon}$.
 As in the similar situation of \eqref{ineq:16/05-1} and \cref{lemma:D_aEnt}, we have
 \begin{equation}
     (\partial^-_rW^2_{r^-})(\mu^0_r,\mu^1_r)\geq \lambda_{r,\varepsilon}\cdot (\partial^-_rW^2_{r^-})(\mu^0_{r,\varepsilon},\mu^1_{r,\varepsilon}) +(1-\lambda_{r,\varepsilon})\cdot (\partial^-_rW^2_{r^-})(\overline\mu^0_{r,\varepsilon},\overline\mu^1_{r,\varepsilon})
 \end{equation}
 and
\begin{align}
\partial_a^- \ent_r(\mu^a_r)\big|_{a=1}-\partial_a^+ \ent_r(\mu^a_r)\big|_{a=0}
&\ge
\lambda_{r,\varepsilon}\Big[
\partial_a^- \ent_r(\mu^a_{r,\varepsilon})\big|_{a=1}-\partial_a^+ \ent_r(\mu^a_{r,\varepsilon})\big|_{a=0}\Big]\\
&+(1-\lambda_{r,\varepsilon})\Big[
\partial_a^- \ent_r(\overline\mu^a_{r,\varepsilon})\big|_{a=1}-\partial_a^+ \ent_r(\overline\mu^a_{r,\varepsilon})\big|_{a=0}\Big].\label{ineq:s2}
\end{align}
Combining these estimates with the $K$-convexity of entropy, the log-Lipschitzness of metrics and the definition of $\eta$-quantity allows us to conclude that
\begin{align}\label{eq:etaepsilonestimate}
    (\partial^-_rW^2_{r^-})(\mu^0_r,\mu^1_r)/2+&\partial^-_a \ent_r(\mu^a_r)\big|_{a=1}-\partial_a^+ \ent_r(\mu^a_r)\big|_{a=0}\\  
    \geq & \lambda_{r,\varepsilon}\cdot\eta_{\varepsilon}(r,x,y)W^2_r(\mu^0_{r,\varepsilon},\mu^1_{r,\varepsilon})+(1-\lambda_{r,\varepsilon})(K-L)W^2_r(\overline\mu^0_{r,\varepsilon},\overline\mu^1_{r,\varepsilon}).
\end{align}
\textbf{Error estimates via heat kernels}  We show that there exists a constant $C$ depending on $K,N$ only such that the mass of $\pi_r$ outside $\Gamma_r^\varepsilon$ is bounded by $C\exp(-\frac{\varepsilon^2}{C(t-r)})$.
The estimate is established as follows
\begin{align}\label{eq:lambdaepsestimate}
    1-\lambda_{r,\varepsilon}&\leq \hat P_{t,r}\mu^0(B_r(x,\varepsilon)^c)+\hat P_{t,r}\mu^1(B_r(y,\varepsilon)^c)\\
    &=\int_{B_r(x,\varepsilon)^c}\int p_{t,r}(x',y')\d\mu^0(x')\d\m_r(y')+\int_{B_r(y,\varepsilon)^c}\int p_{t,r}(x',y')\d\mu^1(x')\d\m_r(y')\\
    &\leq \sup_{x'\in B_r(x,\varepsilon/2)}
 \int_{B_r(x,\varepsilon)^c}p_{t,r}(x',z)\d\m_r(z)+
 \sup_{y'\in B_r(y,\varepsilon/2)}
 \int_{B_r(y,\varepsilon)^c}p_{t,r}(y',z)\d\m_r(z)\\
 &\overset{(*)}{\leq} 2\sup_{x'\in X}\int_{B_r(x',\varepsilon/2)^c} \frac{C_{1}}{\m_{r}(B_{r}(x',\sqrt{t-r}))}\exp\left(-\frac{d_r^2(x',z)}{C_1(t-r)}\right)\d \m_r(z)\\
 &\leq 2C_1\sup_{x'\in X}\sum_{k\geq 1}\int_{B_r(x',\frac{k+1}{2}\varepsilon)\setminus B_r(x',\frac{k}{2}\varepsilon)}\frac{1}{\m_{r}(B_{r}(x',\sqrt{t-r}))}\exp\left(-\frac{d_r^2(x',z)}{C_1(t-r)}\right)\d \m_r(z)\\
 &\lesssim  \sup_{x'\in X}\sum_{k\geq 1}\frac{\m_r(B_r(x',k\varepsilon))}{\m_{r}(B_{r}(x',\sqrt{t-r}))}\exp\left(-\frac{k^2\varepsilon^2}{C_1(t-r)}\right)\\
&\overset{(**)}{\lesssim}\sum_{k\geq 1} \exp\left(C_2\frac{k\varepsilon}{\sqrt{t-r}}-\frac{k^2\varepsilon^2}{C_1(t-r)}\right)\label{ineq:sumexp}\;,
\end{align}
where the inequality $(*)$ with the constant $C_1$ comes from the Gaussian estimate \eqref{ineq:Gaussian} and the inequality $(**)$ with the constant $C_2$ results from the Bishop-Gromov volume comparison on $\mathrm{CD}(K,N)$ spaces, cf. \cite[Section 2.3]{BGNZ-Crelle23}.

It suffices to consider $\frac{\varepsilon}{\sqrt{t-r}}>2C_1C_2$, otherwise we trivially have
\[
1-\lambda_{r,\varepsilon}\leq 1\lesssim \exp(-C_1C_2)\lesssim \exp(-\frac{\varepsilon^2}{C(t-r)}).
\]
Then continuing the estimate \eqref{ineq:sumexp}, one has
\begin{align}
	&\sum_{k\geq 1} \exp\left(C_2\frac{k\varepsilon}{\sqrt{t-r}}-\frac{k^2\varepsilon^2}{C_1(t-r)}\right)\lesssim \sum_{k\ge 1}\exp\left(-\left(\frac{k\varepsilon}{\sqrt{C_1(t-r)}}-C_2\sqrt{C_1}\right)^2\right)\\
	&\leq  \int^{\infty}_{\frac{\varepsilon}{2}}\exp\left(-\left(\frac{x}{\sqrt{C_1(t-r)}}-C_2\sqrt{C_1}\right)^2\right)\d x\lesssim C\exp\left(-\frac{\varepsilon^2}{C(t-r)}\right).
\end{align}
\textbf{Passing to the limit}
To simplify the demonstration, we use the following short-hand notations
\begin{align}
    g(\mu^0,\mu^1;r)\coloneqq (\partial_r^+W^2_{r^-})(\mu^0_r,\mu^1_r)/2+\partial^-_a \ent_r(\mu^a_r)\big|_{a=1}-\partial_a^+ \ent_r(\mu^a_r)\big|_{a=0}
\end{align}
and 
\begin{align}
   \varepsilon(\mu^0,\mu^1;r)\coloneqq (1-\lambda_{r,\varepsilon})W^2_r(\overline\mu^0_{r,\varepsilon},\overline\mu^1_{r,\varepsilon}).
\end{align}
For fixed $\varepsilon>0$, summarizing the previous steps, we know that the error term $\varepsilon(\mu^0,\mu^1;r)$ goes to $0$ as $r\to t$ and for all $r$
\begin{align}\label{inequ:eta-g-error}
    \eta_\varepsilon(r,x,y)\le \frac{g(\mu^0,\mu^1;r)+(L-K)\varepsilon(\mu^0,\mu^1;r)}{W^2_r(\mu^0_r,\mu^1_r)-\varepsilon(\mu^0,\mu^1;r)}.
\end{align}
Denote by $\sigma_r$ the optimal plan between $\mu^0_r$ and $\mu^1_r$ w.r.t. the cost $d^2_r$.
By the stability of optimal transports (cf. \cite[Theorem 5.20]{Villani}), $r\mapsto \sigma_r$ is continuous on $\mathcal{P}(X^2)$, the space of probability measures on $X^2$ metrized by the weak convergence.
In particular, $\lambda_{r,\varepsilon}$ and hence $ \varepsilon(\mu^0,\mu^1;r)$ are measurable in $r$ since $ \lambda_{r,\varepsilon}=\sigma_r(B_r(x,\varepsilon)\times B_r(y,\varepsilon))$.

Note that $t\mapsto \eta_\varepsilon(t,x,y)$ is not necessarily measurable. For an interval $[s_0,t_0]$ consider 
\[\inf\Big\{\int_{s_0}^{t_0} \beta(t)\d t\colon \beta \text{measurable}\;, \beta\geq \eta_\varepsilon(\cdot,x,y) \text{ a.e.}\Big\}\;.\]
By \ref{inequ:eta-g-error} the infimum is finite for $s_0$ sufficiently close to $t_0$. Pick $\beta$ which realises the infimum (existence of a minimiser follows from the fact that any minimising sequence can be assumed to be monotonically decreasing without restriction). Then $\beta$ realises the infimum also for any subinterval of $[s_0,t_0]$. By the Lebesgue differentiation theorem we obtain for a.e. $t$:
\begin{align}
    \eta_\varepsilon(t,x,y)\le \beta(t)= \lim_{s\to t}\aint[{[s,t]}]{}\beta(r)\d r\;.
\end{align}
Due to \ref{inequ:eta-g-error} we further have for any such $t$ and $s<t$ sufficiently close, that $\beta(r)$ is bounded by the right-hand side of \ref{inequ:eta-g-error} for all $r\in [s,t]$ (note that the latter depends implicily on $t$). Thus we have
\begin{align}
    \eta_\varepsilon(t,x,y)&\le\liminf_{s\to t}\left[\aint{[s,t]} \frac{g(\mu^0,\mu^1;r)+(L-K)\varepsilon(\mu^0,\mu^1;r)}{W^2_r(\mu^0_r,\mu^1_r)-\varepsilon(\mu^0,\mu^1;r)}\d r\right]
    \\ &= \liminf_{s\to t} \frac{1}{(t-s)W_t^2(\mu^0,\mu^1)}\int_{[s,t]} g(\mu^0,\mu^1;r)\d r
    \le \liminf_{s\to t}\frac{W_t^2(\mu^0,\mu^1)-W_s^2(\mu_s^0,\mu_s^1)}{(t-s)W_t^2(\mu^0,\mu^1)}\;.
\end{align}

Taking the infimum on both sides over $\mu^0$ and $\mu^1$ we conclude $\eta_\varepsilon(t,x,y)\le \vartheta^\flat(t,x,y)$ and the thesis follows.
\end{proof}

\begin{rem}
Part of the additional complexity in the above proof arises from the unknown measurability of $\eta$.
We know from \cref{cor:eta+} that on smooth flows, $\eta(t,x,y)$ is jointly continuous.
However, for general metric-measure spaces, while this is not  clear from the definition, a posteriori, using the pointwise comparison in \cref{prop:eta>theta} and \cref{thm:eta-theta}, we see that $t\mapsto \eta(t,x,y)$ is Lebesgue measurable. Indeed, $\vartheta^-$ is measurable from the definition, and so is $\vartheta^\flat$ due to \cref{prop:eta>theta}. But for any $x,y\in X$, the latter agrees with  $\eta(t,x,y)$ for a.e. $t$.
\end{rem}

We can now show that any rough sub-Ricci flow is a weak sub-Ricci flow.
For $x\in X$ and $t\in I$ we set
\[
\vartheta^\flat(t,x)\coloneqq\limsup\limits_{y,z\to x}\vartheta^\flat(t,y,z)\;.
\] 
\begin{corollary}\label{cor:Roughweaksub}
For almost every $t$ we have $\eta^*(t,x)=\vartheta^{\flat}(t,x)$ for all $x\in X$.
In particular, if $(X,d_t,\m_t)_{t\in I}$ is a rough sub-Ricci flow, then it is a weak sub-Ricci flow.
 \end{corollary}
 \begin{proof}
 Since $X$ is separable, there exists a countable dense (w.r.t. all $d_t$) subset $\{x_n\}_{n\in\N}$ and thanks to \cref{thm:eta-theta}, for a.e. $t$, 
 \begin{equation}
     \eta(t,x_n,x_m)\leq \vartheta^{\flat}(t,x_n,x_m),\quad \forall n,m\in \N.
 \end{equation}
 For any $\delta>0$, one can find sequences $(y_k)_k$ and $(z_k)_k$ converging to $x$ s.t. $\eta^*(t,x)\leq\limsup\limits_{k\to \infty}\eta(t,y_k,z_k)+\delta$.
     By definition, $\eta(t,y,z)$ is lower semi-continuous w.r.t. $y,z$. 
     Hence, for each $k$, there exist $x'_k,x''_k\in\{x_n\}$ so that
     \begin{equation}
         d_t(x'_k,y_k),d_t(x''_k,z_k)<1/k,\quad \eta(t,y_k,z_k)<\eta(t,x'_k,x''_k)+1/k.
     \end{equation}
Hence we have
     \begin{align}
         \eta^*(t,x)-\delta\leq\limsup\limits_{k\to \infty}\eta(t,y_k,z_k)\leq \limsup_{k\to \infty}\eta(t,x'_k,x''_k)\leq  \limsup_{k\to \infty}\vartheta^{\flat}(t,x'_k,x''_k)\leq \vartheta^{\flat}(t,x)\leq \vartheta^*(t,x).
     \end{align}
     The reverse direction is due to \cref{prop:eta>theta}.
   Then the implication for a rough sub-Ricci flow to be a weak sub-Ricci flow is an immediate consequence of \cref{thm:RFvsTheta} and \cref{thm:charaSRF}.  
 \end{proof}

\section{Examples and non-examples}\label{sec:examples}
\subsection*{Static cones}

The constant flow of an $\rcd(0,N)$ metric measure space can be interpreted as a weak or rough $N$-super-Ricci flow. 
Notably, this includes any Euclidean $(N-1)$-cone over an $\rcd(N-2,N-1)$ metric measure space, provided $N \geq 2$.
However, by Theorem 1.1 in \cite{Erbar-Sturm}, the only Euclidean $N$-cone with rough Ricci curvature bounded above by 0, and satisfying $\rcd(K,N')$ for some $K, N' \in \mathbb{R}$, is Euclidean space $\mathbb{R}^{N+1}$. 
Therefore, the only static Euclidean $N$-cone that forms a rough Ricci flow is Euclidean space.

\subsection*{Spherical Suspension}
Let $(X,d,\m)$ be a metric measure space.
The \emph{$N$-spherical suspension} is the mms defined on the set $\Sigma(X)=X\times [0,1]/\sim$ where $(x,r)\sim (y,r)$ $\Leftrightarrow$ $r=s=0$ or $r=s=\pi$ i.e. $\mathcal{S}=X\times\{0\}$ and $\mathcal{N}=X\times \{\pi\}$ are contracted to a point, the south and north pole respectively.
$\Sigma(X)$ is equipped with the following distance $d_{\Sigma}$ and measure $\m_{\Sigma}$:
\begin{align}
\label{eq:d_conic}
    &\cos\left(d_{\Sigma}((x,s),(x',s')) \right)\coloneqq \cos s\cos s'+\sin s\sin s'\cos(d(x,x')\wedge \pi
    )\quad (x,s),(x',s')\in \Sigma(X)\;,\\
 &\d \m_{\Sigma}(x,s)\coloneqq\d \m(x)\otimes \sin^N(s)\d s\;.
\end{align}
It is well known after Ketterer \cite{Ketterer-JMPA15} that when the base space $(X,d,\m)$ satisfies $\rcd(N-1,N)$ for some $N\geq 1$ and $\mathrm{diam}(X)\leq \pi$, then $(\Sigma(X),d_{\Sigma},\m_{\Sigma})$ satisfies $\rcd(N,N+1)$.
Consider the following time scaling of the metric
\begin{equation}
    d_t\coloneqq(1-2Nt)^{\frac12}d_{\Sigma},\quad \m_t\coloneqq c_t\m_{\Sigma},
\end{equation}
with $c:I\to \R$ a suitably regular function.
Then the time-dependent space $(\Sigma(X),d_t,\m_t)_{t\in(0,\frac{1}{2N})}$ is a $(N+1)$-super-Ricci flow, see e.g. \cite[Proposition 2.7]{Sturm2018Super} for a short explanation.

The following result shows that when the base space has a prescribed constant Ricci curvature, the time-scaled suspension becomes a weak Ricci flow.
However it qualifies as a rough sub-Ricci flow only in the absence of singularities at the poles, i.e., when it is a sphere.
\begin{theorem}\label{prop:suspension}
    Let $M$ be an $n$-dimensional Einstein manifold with $\Ric_{g_M}\equiv (n-1)g_M$.
    Then the time-scaled spherical suspension $(\Sigma(M),d_t,\m_t)_{t\in (0,\frac{1}{2n})}$ is a weak Ricci flow, where 
    \begin{equation}\label{eq:scaling}
        d_t\coloneqq(1-2nt)^{\frac12}d_{\Sigma},\quad \m_t\coloneqq(1-2nt)^{\frac{n+1}{2}}\m_{\Sigma}.
    \end{equation}
     Moreover, it is a rough super Ricci flow and it is rough sub-Ricci flow if and only if $M$ is the unit sphere $\mathbb{S}^{n}$ with the round metric and a multiple of volume measure.
\end{theorem}
\begin{proof}
    The proof relies on several useful facts about spherical suspensions from \cite{BacherSturm14}.

    Denote by $\Sigma_0\coloneqq \Sigma(M)\setminus \{\mathcal{S},\mathcal{N}\}$ the punctured cone, which is an incomplete smooth manifold regarded as a warped product $M_{\sin (s)}\times (0,\pi)$.
The metric tensor on $\Sigma_0$ is given by $g_0\coloneqq \sin^2sg_{M}\oplus \d s^2$.
The length distance and the volume measure induced by $g_0$ coincide with $d_{\Sigma}$ and $\m_{\Sigma}$, respectively.

    For any $(x,s)\in \Sigma_0$, and $(v,t)\in T_{(x,s)}\Sigma_0$, we have (see \cite[Lemma 8]{BacherSturm14}):
    \begin{equation}
        \Ric_{g_0}((v,t),(v,t))=\Ric_{g_M}(v,v)+(1-n\cos^2s)\cdot g_M(v,v)+nt^2.
    \end{equation}
    Combining this with $\Ric_{g_M}\equiv (n-1)g_M$ and
    \begin{equation}
        g_0((v,t),(v,t))=\sin^2s \cdot g_M(v,v)+t^2
    \end{equation}
    shows that $\Ric_{g_0}\equiv n\cdot g_0$. 
    Furthermore, the distance $d_t$ is induced by $g_t=(1-2nt)g_0$ and we have $\Ric_{g_t}\equiv n\cdot g_t$ on $\Sigma_0$ for all $t$ as the Ricci tensor is invariant under uniform scaling.
\smallskip

\emph{Weak non-collapsed Ricci flow.} 

We obtain that $\Ric_{g_t}+\frac12\partial_t g_t=0$ on the incomplete smooth manifold $\Sigma_0$. By \cite[Theorem 6]{BacherSturm14}, every optimal dynamical plan between absolutely continuous measures assigns no mass to geodesics passing through the poles.
Hence by arguments in \cref{sec:weaksub} we conclude that the time-dependent space is a weak sub-Ricci flow and a $(n+1)$-super-Ricci flow.
\smallskip

\emph{Rough super-Ricci flow but not rough sub-Ricci flow.} 

 Note that $(\Sigma(M),d_\Sigma,\m_\Sigma)$ is an RCD$(N,N+1)$ space and hence by \cite{Kopfer-Sturm2018} it is a rough super-Ricci flow as the time-dependent scaling of an RCD space, in particular it satisfies \cref{asm:dualheatflow}. Denote by $(P_t)_{t\geq 0}$ and $(P_{t,s})_{t\geq s}$ the heat semigroup on $(\Sigma(M),d_{\Sigma},\m_{\Sigma})$ and the heat propagator on $(\Sigma(M),d_t,\m_t)_{t\in(0,\frac{1}{2n})}$ respectively.  
 Observe that for all $s<t$
\begin{equation}\label{eq:twoheatflow}
    P_{t,s}= P_{\phi(t)-\phi(s)},\quad \phi(t)\coloneqq \frac{-\log(1-2nt)}{2n}.
\end{equation}
Indeed, for any $f\in C^{\infty}_c(\Sigma_0)$, since $\Delta_{g_t}=(1-2nt)^{-1}\Delta_{g_0}=\phi'(t)\Delta_{g_0}$ for each $g_t=(1-2nt)g_0$, it holds
\begin{align}
    \partial_t (P_{\phi(t)-\phi(s)}f)&=\Delta_{g_0}((P_{\phi(t)-\phi(s)}f))\phi'(t)=\Delta_{g_t}P_{\phi(t)-\phi(s)}f.
\end{align}
In other words, \eqref{eq:twoheatflow} holds on $C^{\infty}_c(\Sigma_0)$, a dense subset of $L^2(\Sigma(M))$ and hence on the whole $L^2(\Sigma(M))$ by continuity.
The same relation holds also for the dual heat propagators.

For every $t$ and $o,p\in \Sigma(M)$, we can compute the $\vartheta$-quantity as follows
\begin{align}
    &\frac{1}{t-s}\left(d^2_t(o,p)-W^2_s(\hat{P}_{t,s}\delta_o,\hat{P}_{t,s}\delta_p)\right)\\
    = &\frac{1}{t-s}\big((1-2nt)d^2_0(o,p)-(1-2ns)W^2_0(P_{\phi(t)-\phi(s)}\delta_o,P_{\phi(t)-\phi(s)}\delta_p)\big)\\
    =& \frac{1-2ns}{t-s}\left(d^2_0(o,p)-W^2_0(P_{\phi(t)-\phi(s)}\delta_o,P_{\phi(t)-\phi(s)}\delta_p)\right)+\frac{1-2nt-(1-2ns)}{t-s}d^2_0(o,p).
\end{align}
   Taking limsup as $s\nearrow t$ implies
   \begin{align}
       \vartheta^+(t,o,p)= (1-2nt)\vartheta^+(o,p)\phi'(t)-n=\vartheta^+(o,p)-n.
   \end{align}
  In \cref{thm:suspension}, we show in analogy with the result in \cite{Erbar-Sturm} for cones that $\vartheta^+(o,p)=+\infty$ when $o$ a pole and $p$ in the same hemisphere of $o$ unless $M$ is the unit sphere $\mathbb{S}^n$ with the round distance and a multiple of the volume measure.
   Therefore, $\vartheta^*(t,o)=\infty$ for all $t$ and $o$ being the south or north pole when $M\neq \mathbb{S}^n$.
   	Hence by \cref{thm:RFvsTheta}, $(\Sigma(M),d_t,\m_t)_{t\in (0,\frac{1}{2n})}$ is not a rough sub-Ricci flow unless $M=\mathbb{S}^n$.
\end{proof}

In particular, the induced time-dependent mms $(\Sigma(M),d_t,\m_t)_{t\in [0,\frac{1}{8})}$ for $M=S^2(1/\sqrt{3})\times S^2(1/\sqrt{3})$ is not a rough Ricci flow, which was conjectured in Example 1.4 of \cite{Kopfer-Sturm2018}.

\begin{rem}
    In the proof of \cref{prop:suspension}, the argument on the rough sub-Ricci flow part is robust enough to apply to general $\rcd$-spaces.
    That is, let $\Sigma(X)$ be a spherical suspension over an $\rcd(n-1,n)$ space $(X,d,\m)$ for some $n\geq 1$.
    Then the scaled suspension $(\Sigma(X),d_t,\m_t)_{t\in [0,\frac{1}{2n})}$, given by \eqref{eq:scaling}, is a rough sub-Ricci flow if and only if $X$ is the unit sphere with the round metric and a multiple of volume measure.
\end{rem}

\subsection*{Gaussian weights} 
Let $X=\R^n$.
Take a family of distances $d_t$ induced by the inner product $\langle \cdot, A_t\cdot\rangle$ where $A:I\to \R^{n\times n}$ is positive definite for all $t\in I$.
Consider $\m_t$ to be the weighted Lebesgue measure $e^{-f_t}\mathcal L^n$ with 
\[
    f_t(x)=\frac12\langle x,a_t x\rangle+\langle x,b_t\rangle+c_t\;, 
\]
where $a: I\to \R^{n\times n}$, $b: I\to \R^n$, $c:I\to\R$ are suitably regular functions.
Then $(X,d_t,\m_t)_{t\in I}$ is a weak/rough super-resp.~sub-Ricci flow if and only if 
\[\dot A_t \geq -2 a_t\;,\quad \text{resp.}\quad \dot A_t\leq -2a_t\;.\]
However, it will not be a $N$-super-Ricci flow for some $N\in [n,\infty)$ unless $a\equiv 0$ and $b\equiv 0$.

In particular, consider $A_t=1-2t$, $a_t\equiv 1$, $b_t\equiv 0$ and $c\equiv 0$ for $t\in(0,\frac12)$ corresponding to $d_t=\sqrt{1-2t}$ and $f_t(x)=\frac{|x|^2}{2}$.
Then the time-dependent space $(\R^n,d_t,e^{-f}\mathcal{L}^n)_{t\in(0,\frac12)}$, often referred to as a shrinking Gaussian, is a rough Ricci flow but not a non-collapsed Ricci flow.

\appendix
\section{time-dependent potentials}\label{sec:appendix}
Let $(M,g)$ be a smooth complete Riemannian manifold with Riemannian distance $d$. For a function $\phi:M\to\R$ and $a\in[-1,1]$ we define the propagated potentials $(\phi^a)_{a\in[-1,1]}$ by setting $\phi^0=\phi$ and 
\begin{equation}\label{eq:interpotential}
    \phi(a,x)=\phi^a(x)\coloneqq \begin{cases}
     \sup_{z\in M}[\phi(z)-\frac{d^2(z,x)}{2a}],& a\in(0,1] \;, \\
    \inf_{z\in M}[\phi(z)+\frac{d^2(z,x)}{2|a|}], & a\in[-1,0)\;.
\end{cases}
\end{equation}
In other words, for $a\in(0,1]$ we set $\phi^a=-\phi^c$ and $\phi^{-a}=(-\phi)^c$ for using the $c$-transform defined in \eqref{eq:ctransform} with the cost $d^2/2a$. 

\begin{lemma}\label{lemma:potential}
    Let $M$ be a smooth complete Riemannian manifold, and let $K$ be a compact subset of $M$. Then, there is $\varepsilon>0$ such that for any function $\phi\in C^\infty_c(M)$ with $\spt(\phi)\subset K$ and $\|\phi\|_{C^2}\leq \varepsilon$, the following hold:    
\begin{enumerate}
        \item the map $\Phi^a\colon x\mapsto \exp_x(-a\nabla\phi(x))$ is a diffeomorphism on $M$ for all $a\in[-1,1]$;
    \item\label{item:tangent+HJ} the map $(a,x)\mapsto \phi(a,x)$ defined in \eqref{eq:interpotential} is smooth on $[-1,1]\times M$ and satisfies the Hamilton-Jacobi equation
\begin{equation}\label{eq:HamiltonJacobi}
    -\partial_a\phi^a(x)+\frac{1}{2}|\nabla\phi^a(x)|^2_g=0\;,
\end{equation}
as well as
\begin{equation}\label{eq:velocity}
    \frac{\d}{\d a}\Phi^a(x)=-\nabla\phi^a(\Phi^a(x))\;;
\end{equation}
 \item\label{item:optimalityachieved}  both $\phi$ and $-\phi$ are $d^2/2$-concave. For any $a\in[0,1]$ the pair $(2a\phi^{-a},-2a\phi^{a})$ is admissible in the Kantorovich duality \eqref{eq:KantorovichDuality} w.r.t. the cost $d^2/2$, i.e. we have
    \begin{equation}\label{eq:admissible}
    2a\phi^{-a}(x)-2a\phi^a(y)\leq \frac{1}{2}d^2(x,y),\quad \forall x,y\in M\;.
    \end{equation}
Moreover, for every $z\in M$ we have
   \begin{equation}\label{eq:optimalityachieved}
        2a\phi^{-a}(\Phi^{-a}(z))-2a\phi^a(\Phi^a(z)=
 \frac{1}{2}d^2(\Phi^{-a}(z),\Phi^a(z))\;.
    \end{equation}
\end{enumerate}
\end{lemma}

\begin{proof}
 As $\phi$ is smooth, $\Phi^a$ is smooth. Taking $\varepsilon$ smaller than the injectivity radius of $K$, we have that $\gamma^a_z\colon[-1,1]\ni a\mapsto \Phi^a(z)$ is a geodesic for all $z\in M$. 
 
By \cite[Thm.~13.5]{Villani} there is $\varepsilon>0$ (depending only on the metric tensor within a neighborhood of $K$) such that any function $\phi$ with $\|\phi\|_{C^2}\leq 2\varepsilon$ is $d^2/2$-concave. Hence, for every $a\in(0,1]$ the pairs $(\phi,-\phi^a)$ and $(\phi^{-a},-\phi)$ are conjugate for the cost $d^2/2|a|$ by construction. In particular, we have
\begin{equation}\label{inequ:admissible}
 \begin{split}
    -\phi^a(x)+\phi(z)\leq \frac{d^2(x,z)}{2a},&\quad\forall x,z\in M\;,\\
    \phi^{-a}(y)-\phi(z)\leq \frac{d^2(y,z)}{2a}, &\quad\forall y,z\in M\;.
\end{split}
\end{equation}
    The proof of \cite[Thm.~13.5]{Villani} shows moreover (see also \cite[Lem.~3.3]{Cordero-Erausquin2001}), that equality is obtained uniquely at $x=\Phi^a(z)$ and $y=\Phi^{-a}(z)$ respectively, i.e.
    \begin{equation}\label{eq:eqachieved42}
-\phi^a(\Phi^a(z))+\phi(z)= \frac{d^2(\Phi^a(z),z)}{2a},\quad\forall z\in M\;,
    \end{equation}
    and similarly for $-a$. Thus (3) is established. For all $x,y\in M$ we then have
\begin{equation}
    -\phi^a(x)+\phi^{-a}(y)\leq \inf_{z\in M}\Big[\frac{d^2(x,z)}{2a}+\frac{d^2(y,z)}{2a}\Big]=\frac{d^2(x,y)}{4a}\;,
\end{equation}
 i.e. $(2a\phi^{-a},-2a\phi^a)$ is admissible for $d^2/2$. Moreover, we have for any $z$ that
 \begin{align}
     \phi^{-a}(\Phi^{-a}(z))-\phi^a(\Phi^a(z)) &= \phi^{-a}(\Phi^{-a}(z))-\phi(z)+ \phi(z)-\phi^a(\Phi^a(z))\\
     & = \frac{1}{2a} d^2(\Phi^{-a}(z),z) +\frac{1}{2a} d^2(z,\Phi^a(z)=\frac{1}{4a} d^2(\Phi^{-a}(z),\Phi^a(z))\;,
 \end{align}
where we used that $[-1,1]\ni a\mapsto \Phi^a(z)$ is a geodesic. Thus (3) is established.

As also $2\phi$ is $d^2/2$-concave, the set 
$\big\{\big(z,\exp_x(-2\phi(z)\big)\colon z\in M\big\}$ is $d^2/2$-cyclically monotone. By the Mather shortening lemma \cite[Cor.~8.2]{Villani}
we have that 
\[d\big(\Phi^a(x),\Phi^a(y)\big)\geq C d(x,y)\quad \forall a\in[0,1],\; x,y\in M\;,\]
for some constant $C$ depending on $K$. In particular, for all $a\in[0,1]$ the differential of $\Phi^a$ is non-degenerate.
Hence, by the inverse function theorem $\Phi^a$ is a diffeomorphism on $M$ for all $a\in [0,1]$. The same argument applies to $a\in[-1,0]$ and we conclude (1).

To show (2), note that the joint smoothness of $(a,x)\mapsto \phi^a(x)$ is a consequence of \eqref{eq:eqachieved42} and $\Phi^a$ being a diffeomorphism. Since $(a,x)\mapsto -\phi^a(x)$ is constructed via the Hopf-Lax semigroup, it solves  the Hamilton-Jacobi equation in viscosity sense (see cf. \cite[Section 3]{AGSInvention}) and hence classically due to smoothness.
\end{proof}

\begin{lemma}\label{lemma:continuityeq}
   Let $(M,g,f)$ be a weighted smooth Riemannian manifold with reference measure $\m\coloneqq e^{-f}\d \mathrm{vol}$.
    Let $\Phi^a$ and $\phi^a$ for $a\in[-1,1]$ be as in \cref{lemma:potential}
    and let $\mu^0=\rho^0\m$ be a probability measure compactly supported smooth density $\rho^0$ w.r.t.~$\m$. Then for all $a\in [-1,1]$ the measure $\mu^a\coloneqq (\Phi^a)_{\#}\mu^0=\rho^a\m$ has smooth compactly supported density $\rho^a$ satisfying the continuity equation
  \begin{equation}\label{eq:continuityeq}
    \partial_a \rho^a=\mathrm{div}_{f}(\rho^a\nabla\phi^a).
   \end{equation}
\end{lemma}
Here $\mathrm{div}_{f}$ denotes the weighted divergence given by 
\[\mathrm{div}_{f}w\coloneqq \mathrm{div}w-\nabla f\cdot w\;.\]
   
 \begin{proof}
 Since $\Phi^a$ is a diffeomorphism for all $a$, by change-of-variable formula (see e.g. \cite{Mccann2001}), we have the following Monge-Amp\`ere equation:
\begin{equation}
    \rho^a(\Phi^a(x))\det (\d\Phi^a(x))=\rho^0(x),\quad \forall x\in M\;.
\end{equation}
In particular, $(a,x)\mapsto \rho^a(x)$ is smooth.
To show \eqref{eq:continuityeq}, we integrate against an arbitrary test function $\chi\in C^{\infty}_c(M)$ and use \eqref{eq:velocity} to obtain
\begin{align}
   \int \partial_a\rho^a\chi\d \m&=\partial_a\int \chi\d \mu^a=\partial_a\int \chi \d(\Phi^a)_{\#}\mu^0=\partial_a\int \chi\circ \Phi^a\d \mu^0\\
   &=\int \partial_a(\chi\circ \Phi^a)\d \mu^0=\int \nabla\chi\cdot (-\nabla\phi^a)(\Phi^a)\d \mu^0=\int-\nabla\chi\cdot \nabla\phi^a\d \mu^a\\
   &=\int-\nabla\chi\cdot \nabla\phi^a \rho^a\d \m=\int \mathrm{div}_{f}(\rho^a\nabla\phi^a)\chi\d\m\;.\qedhere
\end{align}
 \end{proof}

Now let $(M,g_t,f_t)_{t\in I}$ be a smooth flow of complete manifolds.

\begin{lemma}\label{lemma:spacetimepotentials} 
	For each $t\in I$ and given $\delta>0$, a compact subset $K\subset M$, and a symmetric matrix $A$ there exist $\varepsilon>0$ such that for any $x,y\in K$ with $2r\coloneqq d_t(x,y)<\varepsilon$, there exists a smooth function $(s,a,z)\mapsto\phi^a_s(z)$ on $[t-\delta,t+\delta]\times[-r,r]\times M$, satisfying 
\begin{enumerate}
    \item\label{item:gradhessian} for all $a\in[-r,r]$ we have  $-\nabla_t \phi^a_t(\gamma^a_t)=\dot{\gamma_t}(a)$ and moreover $\nabla_t^2\phi^0_t(\gamma^0_t)=A$ where $(\gamma^a_t)_{a\in[-r,r]}$ is the unit speed $d_t$-geodesic from $x$ to $y$;
    \item\label{item:conjuacy} for all $s$ and $a\in[0,r]$ the pair $(2a\phi^{-a}_s,-2a\phi^{a}_s)$ is admissible for the cost $d^2_s/2$, i.e.
    \[
    2a\phi^{-a}_s(p)-2a\phi^a_s(q)\leq \frac{1}{2}d^2_s(p,q),\quad \forall p,q\in M\;;
    \]
    moreover $2a\phi^{-a}_t(\gamma_t^{-a})-2a\phi^{a}_t(\gamma_t^a)=d_t^2(\gamma_t^{-a},\gamma_t^a)/2$;
    \item\label{item:HamiltonJacobi} for all $s$, $\phi^a_s$ solves the Hamilton-Jacobi equation 
    \[\partial_a\phi^a_s=\frac{1}{2}|\nabla_s\phi^a_s|^2_{g_s}\;;\]
    \item\label{item:5th_order} we have
    \[
    \sup_{a\in[-r,r]} \|\phi^a_\cdot\|_{C^5([t-\delta,t+\delta]\times M)}\leq C\;,\] for a constant $C$ depending only on $\varepsilon$, $\delta$, $A$, and the spatial-temporal derivatives of the metric tensor in $K$.
\end{enumerate}
\end{lemma}
\begin{proof}
\emph{Step 1: Construction of  $(\phi^a)$.} We can find a compact set $K_1$ such that for all $x,y \in K$ and $s\in[t-\delta,t+\delta]$ such that $d_s(x,y)<1$ any connecting minimizing $d_s$-geodesic is contained in $K_1$. The injectivity radius $\mathrm{inj}_{g_s}(x)$ of $x$ on $(M,g_s)$ is jointly continuous in $s$ and $x$ (see \cite{Enrlich1974}). Let $r_0>0$ be its minimum on $[t-\delta,t+\delta]\times K_1$. Let $\chi:\R^n\to[0,1]$ be a smooth cut-off function with $\chi=1$ in a neighborhood of $0$ and $\chi=0$ outside of $B_{r_0/2}(0)$ and let $C=\|\chi\|_{C^2}$. Let $\varepsilon_0$ be the threshold given in Lemma \ref{lemma:potential} and choose $\varepsilon <\min(r_0, \varepsilon_0/C)$.

Now for any two fixed points $x,y\in K$ with $d_t(x,y)=2r<\varepsilon$ there is a unique minimising unit speed geodesic $(\gamma^a_t)_{a\in[-r,r]}$ for the metric $g_t$ connecting them. Set $p:=\gamma^0_t$. In normal coordinates (w.r.t. $g_t$) around $\gamma^0_t$, chosen such that $x=(-r,0,\dots,0)$ and $y=(r,0,\dots,0)$ we define the function $\phi^0$ via
    \[\phi^0(z) = \Big( z_1 + \frac12 z\cdot Az\Big)\cdot\chi(z)\;.\]
By the choice of $\varepsilon$, the statement of Lemma \ref{lemma:potential} applies to $r\phi^0$ on the space $(M,g_s)$ for each $s\in[t-\delta,t+\delta]$. For such $s$ and  $a\in[-r,r]$ define $\phi^a_s$ by $\eqref{eq:interpotential}$ using the metric $d_s$. Note that by construction of $\phi^0$, we have $\gamma^a_t=\exp_p\big(-a\nabla\phi^0(p)\big)$ (with the exponential map of $g_t$) and $\nabla^2_t\phi^0(p)=A$. The assertions (1), (2), and (3) are then given by Lemma \ref{lemma:potential}. 
\medskip

\emph{Step 2: Bound on the derivatives.} 
It remains to prove asserstion (4). Note that all potentials $(\phi^a_s)_{s,a}$ are supported in a common compact set in which normal coordinates around $p$ are defined. Indeed, \eqref{eq:interpotential} implies that $\phi^a_s(z)= 0$ for all $z$ with $\inf\{ d_s(z,p): p\in\spt(\phi^0)\}\geq \sqrt{2\|a\phi^0\|_{\infty}}$. 
Let $\Phi^a_s(z)=\exp_z\big(-a\nabla_s\phi^0(z)\big)$ with the exponential map of $g_s$ and recall that $\Phi^a_s$ is a diffeomorphism by Lemma \ref{lemma:potential}. For $z\in K$, let $(\gamma_{z,s}^a)_{a\in[-r,r]}$ be the geodesic given by $\Phi^a_s(z)$ and recall that \begin{equation}\label{eq:integralcurve_phi}
    \frac{\d}{\d a} \gamma^a_{z,s}=-\nabla_s\phi^a_s(\gamma^a_{z,s}),\quad \gamma^0_{z,s}=z\;.
\end{equation}

It now suffices to proof the following
\smallskip

\textbf{Claim}: For any $k\in \N_0$, there exists $C_{k}$ so that for any $z\in \spt(\phi^0)$, $s\in[t-\delta,t+\delta]$ and $0\leq i_1,...,i_k\leq n$ we have
\begin{equation}
    \sup_{a\in[-r,r]} |\mathcal{D}_{i_1,i_2,..,i_k}\phi^a_s|(\gamma^a_{z,s})\leq C_k\;,
\end{equation}
where, $\mathcal{D}_{i_1,i_2,..,i_k}\phi^a_s=\frac{\partial^k\phi^a_s}{\partial x^{i_1}\cdots \partial x^{i_k}}$
denotes the $k$-th order space-time derivatives with the convention $\partial/\partial x^0\coloneqq \partial/\partial_s$.
\smallskip

For easier notation set $f(a)\coloneqq \phi^a_s(\gamma^a_{z,s})$
and for $0\leq i_1,...,i_k\leq n$, denote $f_{i_1,i_2,..,i_k}(a)\coloneqq (\mathcal{D}_{i_1,i_2,..,i_k}\phi^a_s)(\gamma^a_{z,s})$.
The claim in the case $k=0$ follows directly from \eqref{eq:optimalityachieved} and the construction of $\phi^0$ in Step 1. The claim for $k=1$ and $i_1=1,\dots,n$ follows immediately from \eqref{eq:integralcurve_phi}. Setting $h(a):=\|\nabla_s^2\phi^a_{s}\|_{\mathrm{HS}}(\gamma^a_{z,s})$ and arguing as in \cite[Thm.~3.1]{sturm2017remarks}, one obtains the Riccati-type differential inequality $h'(a)\leq \sigma +h^2(a)$ with initial condition $h(0)=\|\nabla_s^2\phi^0\|_{\mathrm{HS}}(z)$, where $\sigma$ is an upper bound on the modulus of the Riemann tensor along $\gamma_{z,s}$. The solution to the corresponding differential equation is given by $\sqrt{\sigma}\tan\big(\sqrt{\sigma}(a+\theta)\big)$ for a suitable $\theta\geq0$ chosen to match the initial condition. Sturm's comparision principle thus yields 
\begin{equation}\label{eq:Hessianphi-appendix}
    \|\nabla_s^2\phi^a_{s}\|_{\mathrm{HS}}(\gamma^a_{z,s})\leq \sqrt{\sigma}\tan\big(\sqrt\sigma(a+\theta)\big)\;.
\end{equation}
 Upon possibly reducing the value of $\varepsilon$, this gives the claim in the case $k=2$ and $i_1,i_2=1,\dots,n$.
In order to prove the remaining cases, note that using \eqref{item:gradhessian} and \eqref{item:HamiltonJacobi}, we have 
\begin{align}
   \frac{\d }{\d a} f_{i_1,i_2,..,i_k}(a)&=\bigg[ \mathcal{D}_{i_1,i_2,..,i_k}\frac{|\nabla_s\phi^a_s|^2_{g_s}}{2}-g_s(\nabla_s\mathcal{D}_{i_1,i_2,..,i_k}\phi^a_s,\nabla_s\phi^a_s)
    \bigg](\gamma^a_s) \\
 &=\bigg[\mathcal{D}_{i_1,i_2,..,i_k}(\frac{g_s^{ij}}{2}\frac{\partial \phi^a_s}{\partial x^i}\frac{\partial \phi^a_s}{\partial x^j})-g^{i,j}_s\mathcal{D}_{i,i_1,i_2,..,i_k}\phi^a_s\mathcal{D}_j\phi^a_s
    \bigg](\gamma^a_s)\label{eq:a_derivative}\;,
\end{align}
where in the Einstein summation $i,j=1,\dots,n$. Notice that in this expression all terms containing $k+1$-order derivatives of $\phi$ cancel. Singling out the terms containing $k$-order derivatives, we have
\begin{align}\label{eq:lot-general}
\frac{\d }{\d a}f_{i_1,i_2,..,i_k}(a)&=\sum_{\sigma=1}^k\frac{\partial g^{ij}_s}{\partial x^{i_\sigma}}f_{i,i_1,\dots,\hat{i_\sigma},\dots, i_k}f_j
 + \text{lower order terms}\;.
\end{align}
In particular, when $k=1$ and $k=2$, 
\begin{align}\label{eq:lot-12}
    \frac{\d }{\d a}{f_{i_1}}&=\frac{\partial g^{i,j}_s}{\partial x^{i_1}}f_if_j\;,\quad
    \frac{\d }{\d a}{f_{i_1,i_2}}=\frac{\partial g^{i,j}_s}{\partial x^{i_1}}f_{i,i_2}f_j+\frac{\partial g^{i,j}_s}{\partial x^{i_2}}f_{i,i_1}f_j + \text{ l.o.t.}\;.
\end{align}
Since $f_i,f_j$ have already been shown to be bounded, \eqref{eq:lot-12} gives the claim also for $k=1$ and $i_1=0$. Moreover, we see that for $k=2$, if $i_1=0$ or $i_2=0$, $\eqref{eq:lot-12}$ is a linear ODE for $(f_{0,0},\dots, f_{0,n})$ with coefficients given by derivatives of the metric tensor and $f_i, f_{i,j}$ with $i,j=1,\dots,n$. Since the latter are already shown to be bounded, we obtain the claim also in this case. 
Finally, for $k\geq 3$ \eqref{eq:lot-general} gives a linear ODE for the $k$-order quantities $f_{i_1,\dots,i_k}$ with coefficients given by derivatives of the metric tensor and the quantities $f$ up to order $k-1$. Hence the proof of the claim can be easily completed by induction on $k$.
\end{proof}

\section{Auxiliary results on upper regularity and \texorpdfstring{$\cd$}{CD}-condition}\label{sec:appendix3}
Let $(X,d,\m)$ be a Polish metric measure space with length metric $d$, locally finite measure $\m$ and $\spt(\m)=X$.

\begin{theorem}\label{thm:enb+ur}
    If $(X,d,\m)$ is an e.n.b. m.m.s. such that entropy is upper regular, then for any $\mu^0,\mu^1\in \mathcal{P}(X,d,\m)$, there exists a unique optimal dynamical plan $\pi$, which is induced by a map.
\end{theorem}
\begin{rem}
    Actually, it is sufficient to assume that for any 2-Wasserstein geodesic between $\mu^0,\mu^1\in\mathcal{P}(X,d,\m)$, entropy is upper semi-continuous at $a=0,1$.
\end{rem}
\begin{proof}
    We follow the proof of the analogous statement in \cite[Theorem 3.3]{GigliGAFA}.
    It suffices to show that any $\pi\in\mathrm{OptGeo}(\mu^0,\mu^1)$ is induced by a map. We can always assume the marginal $\mu^0,\mu^1$ are compactly supported with density bounded away from $0$ and $\infty$, since the union
    \begin{align}
        \bigcup_{n}\Gamma_n, \quad \Gamma_n\coloneqq\{\gamma:\rho^0(\gamma^0),\rho^1(\gamma^1)\leq n,\gamma^a\in K_n,\forall a\in[0,1]\}
    \end{align}
    has full $\pi$-measure, where $K_n$ is an increasing sequence of compact subsets in $X$ (whose existence is guaranteed by tightness of probability measures on Polish spaces).
    In particular, $\mu^0,\mu^1\in D(\ent)$.

    Performing the same contradiction argument, we can find two probability measures $\pi^1,\pi^2\ll \pi$ s.t. $\frac{\d \pi^1}{\d \pi}, \frac{\d \pi^2}{\d \pi}$ are bounded, $\pi^1\perp\pi^2$ and $(e_0)_{\#}\pi^1=(e_0)_{\#}\pi^2=\frac{\m\llcorner{D}}{\m(D)}$ for some compact subset $D$.  
    In particular, entropy is upper regular along $\pi^1$ and $\pi^2$ as well.
    By local finiteness of $\m$, there is an open set $U\supset D$, with finite $\m$-measure.
    Since $D$ is compact, there is $r>0$ s.t. $B(D,r)\subset U$ and hence $\spt((e_a)_{\#}\pi^i)\subset U$ for sufficiently small $a$ and $i=1,2$.
    By Jensen's inequality and the upper semi-continuity of entropy:
    \begin{align}
        -\log\m(D)&=\ent((e_0)_{\#}\pi^i)\geq \limsup_{a\to 0}\ent((e_a)_{\#}\pi^i)\\
        &\geq \limsup_{a\to 0}-\log\m\left(\{\rho^{i,a}>0\}\right),\quad \rho^{i,a}\coloneqq \frac{\d(e_a)_{\#}\pi^i}{\d \m},
    \end{align}
   for $i=1,2$.
   When $a$ is small enough, both $\{\rho^{1,a}>0\}$ and $\{\rho^{2,a}>0\}$ have $\m$-measure between $3\m(D)/4$ and $3\m(D)/2$.
   This indicates that the support of $(e_a)_{\#}\pi^1$ and $(e_a)_{\#}\pi^2$ must intersect on a set of positive measure, which contradicts to the non-branching assumption and $\pi^1\perp\pi^2$.
\end{proof}

As a classic result obtained by Sturm in \cite[Theorem 4.17]{sturm2006--1}, we have globalization theorem for $\mathrm{CD}(K,\infty)$.
\begin{theorem}\label{thm:localtoglobalCD}
    Let $(X,d,\m)$ be an e.n.b. locally compact m.m.s.
    If it is a local $\mathrm{CD}(K,\infty)$ space, then it is a $\mathrm{CD}(K,\infty)$ space.
\end{theorem}
In \cite{sturm2006--1}, the theorem was proven assuming $(X,d)$ is compact. 
However the same proof works for the locally compact case under mild modifications similar to the proof of \cref{prop:etatoSRF}.  
More precisely, fix any $\mu^0,\mu^1\in\mathcal{P}(X,d,\m)\cap D(\ent)$ and a Borel non-branching subset $\Gamma$ where some optimal dynamical plan is concentrated.
For any non-null compact subset $\tilde\Gamma\subset\Gamma$, one can apply the argument of Sturm to the compact set $e_{[0,1]}(\tilde\Gamma)$, finding a Wasserstein geodesic s.t. entropy is $K$-convex along it.
Finally, a geodesic between $\mu^0$ and $\mu^1$ can be obtained by some gluing procedure as one can take a countable partition $\{\Gamma_n\}$ of $\Gamma$ by pre-compact subsets (up to some zero measure set).
The desired $K$-convexity then follows from non-branching and analogous estimates as in \cref{lemma:D_aEnt}.
\medskip

Now define function $\eta_\varepsilon^\pm$ and $\eta^\pm$ on $X^2$ as in \cref{def:eta} by regarding a static space as a constant family of spaces. We have the following local-to-global property for static spaces.
\begin{proposition}\label{prop:etatoCD}
    Assume $(X,d,\m)$ is e.n.b., locally compact and the entropy is upper regular. 
    If $\eta^-(x,x)\geq K$ for all $x\in X$, then $(X,d,\m)$ satisfies strong $\mathrm{CD}(K,\infty)$.
\end{proposition}
\begin{proof}
Fix any $\delta>0$. By assumption and the definition of $\eta^-$, for any $x\in X$, there is $r>0$ s.t. $\eta^-_r(x,x)>K-\delta$.
Notice that any Wasserstein geodesic $(\mu^a)_{a\in [0,1]}$ with $\spt(\mu^0),\spt(\mu^1)\subset B(x,\frac{r}{2})$ is contained in $B(x,r)$ in the sense that $\spt(\mu^a)\subset B(x,r)$ for all $a\in [0,1]$.
Therefore, by the definition of $\eta_r$, upper regularity of entropy and \cite[Lemma 1.2]{Sturm2018Super}, entropy is $(K-\delta)$-convex along all geodesics with endpoints supported in $B(x,\frac{r}{2})$.
In particular, $(X,d,\m)$ verifies local $\mathrm{CD}(K-\delta,\infty)$ for any $\delta>0$.
Hence the theorem follows from \cref{thm:localtoglobalCD} and \cref{thm:enb+ur}.
\end{proof}

\section{Rigidity theorem of spherical suspensions}\label{sec:appendix4}
\begin{theorem}\label{thm:suspension}
Let $(\Sigma(X),d_{\Sigma},\m_{\Sigma})$ be the $N$-spherical suspension over some m.m.s. $(X,d,\m)$ with $\mathrm{diam}(X)\leq \pi$.
Assume that $(\Sigma(X),d_{\Sigma},\m_{\Sigma})$ satisfies $\rcd(K',N')$ for some $K'\in\R$ and $N'\in[0,\infty)$.
Then either 
\begin{enumerate}
    \item $\vartheta^*(\mathcal{S})=\vartheta^*(\mathcal{N})=+\infty$;\quad or
    \item $N$ is an integer and $(\Sigma(X),d_{\Sigma},\m_{\Sigma})$ is isomorphic to the unit sphere $\mathbb{S}^{N+1}$ with the round distance and a multiple of the volume measure.
\end{enumerate}
\end{theorem}
The proof is parallel to the proof of Theorem 1.1 in \cite{Erbar-Sturm} with suitable modifications.
We refer interested readers to \cite{Erbar-Sturm} for details of arguments that are identical as in the case of cones.
\begin{proof}
By \cite[Theorem 2.8]{Erbar-Sturm}, the RCD-condition on the suspension ensures that the base space $(X,d,\m)$ satisfies $\rcd(N-1,N)$ (note that the RCD and $\rcd^*$ conditions have been shown to be equivalent in \cite{cavalletti2021globalization}).
As in \cite{Erbar-Sturm}, we will proceed by showing the following claim.

    \textbf{Claim 1:} For any $x\in X$ with $\int_X \cos (d(x,y))\d \m(y)>0$, then
    \begin{equation}
        \vartheta^+(o,(x,r))=+\infty,\quad \text{for }\left\{\begin{array}{c}
            \forall r\in(0,\frac{\pi}{2}], o=\mathcal{S}\\
          \forall r\in[\frac{\pi}{2},\pi), o=\mathcal{N}
        \end{array} \right.
    \end{equation}

Then the proof is concluded by applying \cite[Corollary 1.4]{Erbar-Sturm}, which says an $\rcd(N-1,N)$ space with $\mathrm{diam}(X)\leq \pi$ is the unit sphere $\mathbb{S}^N$ with the round metric, a multiple of the volume measure and $N$ being an integer if and only if 
\begin{equation}\label{eq:cos-sphere}
    \int_X\int_X \cos(d(x,y))\d\m(x)\d \m(y)=0. 
\end{equation}
    
    To prove the Claim 1, denote $a\coloneqq \aint{X}{}\cos(d(x,y))\d\m(y)>0$.
    The idea is to reduce the problem to the Euclidean cone situation, and the latter is proved in Proposition 4.1 of \cite{Erbar-Sturm}.
      Following \cite{Erbar-Sturm}, we denote by $\bar{\nu}^t_p$ the marginal of $\hat{P}_t\delta_p$ in the radial component and by $\{\nu^t_{p,s}\}_{s\in [0,\pi]}$ the disintegration of $\hat{P}_t\delta_p$ over $\bar{\nu}^t_p$ i.e. $\nu^t_{p,s}$ are measures on $\Sigma(X)$ s.t.
   \begin{equation}
      \hat{P}_t\delta_p=\int^\pi_0 \nu^t_{p,s}\d \bar\nu^t_p(s).
   \end{equation}
   Owing to the symmetry between upper and lower half suspension, it is enough to consider $o=\mathcal{S}$ and $r\leq \frac{\pi}{2}$.

   Recall that the conic metric $d_{\Sigma}$ coincides with the length structure induced by
   \[
\mathrm{Length}(\gamma)\coloneqq \int^b_a \sqrt{|r'(s)|^2+\sin^2(r(s))|\dot\theta|^2_X(s)}\d s
   \]
   where $\gamma\colon [a,b]\ni s\mapsto (\theta(s),r(s))\in \Sigma(X)$ is Lipschitz, see e.g. \cite[Chapter 3]{BBImetricgeo}.
   Then, as a consequence of $\sin s\leq s$, for every $(y,s),(y',s')\in \Sigma(X)$
   \begin{equation}
       d_{\Sigma}((y,s),(y',s'))\leq d_C((y,s),(y',s'))\coloneqq \sqrt{s^2+(s')^2-2ss'\cos(d_X(y,y'))}.
   \end{equation}
   Hence as in the Step 2 of the proof of \cite[Proposition 4.1]{Erbar-Sturm}, we have
   \begin{align}
       W^2(\hat{P}_t\delta_o,\hat{P}_t\delta_p)&\leq \int d_\Sigma^2\d \hat{P}_t\delta_o\otimes\hat{P}_t\delta_p\leq \int d_C^2\d \hat{P}_t\delta_o\otimes\hat{P}_t\delta_p\\
       &\leq \int r^2 \d \bar\nu^t_p(r)+\int s^2\d\bar\nu^t_o(s)-2\int s\d\bar\nu^t_o(s)\cdot \int f\d \nu^t_p,\label{ineq:moment}
   \end{align}
      where $f(y,s)=s\aint{X}{}\cos(d(x,y))\d\m(y)$.

\textbf{Claim 2:} There exists $C>0$ s.t.
\begin{equation}
    \int f\d \nu^t_p\geq ar-C\sqrt{t},\quad \forall t>0. 
\end{equation}
    We modify the value of $f$ on the north hemisphere $\{(y,s):s\geq \frac{\pi}{2}\}$ by considering 
   \begin{align}
       \tilde{f}(y,s)\coloneqq \frac{\pi}{2}\sin s\cdot \aint{X}{}\cos(d(x,y))\d\m(y),s\geq \frac{\pi}{2};\quad \tilde{f}(y,s)=f(y,s),s\leq \frac{\pi}{2}.
   \end{align}
   Observe that $\tilde f$ is $\frac{\pi}{2}$-Lipschitz on $\Sigma(X)$.
   Indeed, on the south hemisphere, $d_{\Sigma}\geq \frac{2}{\pi}d_{C}$ as $\sin s\geq \frac{2}{\pi}s$ for $s\in[0,\frac{\pi}{2}]$. 
   Hence the Lipschitzness is clear by \cite[Lemma 4.3]{Erbar-Sturm}.
   On the north hemisphere, we show similarly to \cite[Lemma 4.3]{Erbar-Sturm} that $(y,s)\mapsto \sin s\cdot \cos(d(x,y))$ is $1$-Lipschitz. 
   For every $(y,s),(y',s')\in \Sigma(X)$, consider the geodesic $\gamma\colon [0,1]\ni s\mapsto (\theta(s),r(s))$ connecting them. Then by the Cauchy-Schwarz inequality, it holds that
   \begin{align}
      & |\sin s\cdot \cos(d(x,y))-\sin s'\cdot \cos(d(x,y'))|=\int^1_0 \frac{\d}{\d \tau}\sin(r(\tau))\cos(\theta(\tau))\d \tau\\
       \leq &\int^1_0 |r'(\tau)\cos(r(\tau))\cos(\theta(\tau))|+|\dot\theta|_X(\tau)\cdot|\sin r(\tau)\sin(\theta(\tau))|\d \tau\\
       \leq &\int^1_0 \sqrt{|r'(\tau)|^2+\sin^2(r(\tau))|\dot\theta|^2_X(\tau)}\d \tau=d_{\Sigma}((y,s),(y',s')).
   \end{align}
   Thanks to the non-negativity of $\int_X\cos(d(x,y))\d\m(y)$ for all $x$ (due to Bishop--Gromov volume comparison), we have $\tilde{f}\leq f$.
   Now applying Kantorovich duality with candidate $\tilde{f}$, together with \cite[Lemma 2.2]{Erbar-Sturm}, we obtain
   \begin{align}
       \int f\d \nu^t_p&\geq \int \tilde f\d \nu^t_p= \left(\int \tilde f\d \nu^t_p-\int \tilde f\d \delta_p\right)+\tilde f(p)\\
       &\geq -\frac{\pi}{2}\sqrt{W_2(\nu^t_p,\delta_p)}+\tilde f(p)\geq -\frac{\pi}{2}\sqrt{2Nt}+ar.
   \end{align}
We are left to estimate moments of the heat flow on spherical suspensions.
We denote by $P^I_t,\hat{P}^I_t$ and $L^I$ the heat semigroup, its adjoint acting on measures and its generator on the weighted interval $([0,\pi],|\cdot|, \sin^N r\d r)$, respectively.
Denote by $m_\alpha(p,t)\coloneqq \int s^\alpha \d\bar\nu^t_p(s)$ the $\alpha$-order moments.
\smallskip

\textbf{Claim 3:} It holds
\begin{equation}
    m_2(p,t)=r^2+O(t),\quad m_1(o,t)\geq c\cdot\sqrt{t}, c>0.
\end{equation}
Repeating the proof of \cite[Lemma 4.2]{Erbar-Sturm} gives that $\bar\nu^t_p=\hat{P}^I\delta_r$.
We now have
\begin{equation}
    (L^I  u)(r)=u''(r)+N\frac{\cos r}{\sin r}u'(r),\quad \forall u\in C^\infty_0(I).
\end{equation}
The stochastic process $Y_t$ associated with the generator $\frac12 L^I$ is the solution of following SDE
\begin{equation}
    \d Y_t=\frac{N}{2}\frac{\cos(Y_t)}{\sin(Y_t)}\d t+\d B_t
\end{equation}
with $B_t$ a standard Brownian motion.
Then the moment $m_\alpha(p,t)$ is expressed via the identity
\begin{equation}
    \int r^\alpha\d \bar\nu^t_p(r)=\mathbb{E}[(Y_{2t})^\alpha],\quad Y_0=r.
\end{equation}
By Ito's formula, for $\alpha=2$, 
\begin{align}
    \d (Y_t)^2=2 Y_t\d Y_t+\d \langle Y\rangle_t= 2Y_t\left(\frac{N}{2}\frac{\cos(Y_t)}{\sin(Y_t)}\d t+\d B_t\right)+\d t.
\end{align}
Hence with the fact $\frac{1}{\tan(r)}\leq \frac{1}{r}$ for $r\in [0,\pi]$, the second moment can be bounded as follows
\begin{align}
    m_2(p,\frac{t}{2})&=\mathbb{E}(Y_0^2)+t+N\mathbb{E}\left[\int^t_0 Y_s\frac{\cos(Y_s)}{\sin(Y_s)}\d s\right]\\
    &\leq r^2+ t+N\mathbb{E}\left[\int^t_0 Y_s\frac{1}{Y_s}\d s\right]=r^2+t(N+1).
\end{align}
To bound the first moment, we use lower Gaussian estimate (see e.g. \cite[Theorem 1.1]{Jiang-Li-Zhang}) of heat kernel $p_t(\cdot,\cdot)$ on the $\rcd(N,N+1)$ space $([0,\pi],|\cdot|,\m^I)$.
Then for any $t>0$ small enough s.t. $\sin\sqrt{t}\geq \sqrt{t}/2$, we have
\begin{align}
    m_1(o,t)&=\int^\pi_0 s\d \hat{P}^I\delta_0=\int^\pi_0 p_t(s,0)s\d\m^I(s)\\
    & \gtrsim\int \frac{1}{\m^I((0,\sqrt{t}))}\exp(-\frac{s^2}{t})s\d \m^I(s)\\
    &\geq\int^{\sqrt{t}}_0\exp(-s^2/t)\left(\int^{\sqrt{t}}_{0}\sin^N a\d a\right)^{-1}s\cdot\sin^Ns\d s\\
    &\geq e^{-1}\int^{\sqrt{t}}_0  \left(\int^{\sqrt{t}}_0 a^N\d a\right)^{-1}s(\frac{s}{2})^N\d s\gtrsim \sqrt{t}.
\end{align}
Finally substituting estimates in Claim 2,3 into \eqref{ineq:moment} we conclude that $\vartheta^+(o,p)=+\infty$ whenever $a$ is positive.
\end{proof}

\bibliographystyle{abbrv}
\bibliography{ms.bib}
\end{document}